\documentclass[reqno]{amsart}

\usepackage{color}
\usepackage{mathrsfs}
\usepackage{amscd}
\usepackage{amsmath}
\usepackage{latexsym}
\usepackage{amsfonts}

\usepackage{amssymb}
\usepackage{amsthm}
\usepackage{graphicx}
\usepackage{hyperref}
\usepackage{makecell}
\usepackage{array,color}
\usepackage{booktabs}
\usepackage{multirow}
\usepackage{bm}
\usepackage{bbm}
\usepackage{verbatim}
\usepackage{booktabs}
\usepackage{tabularx}
\usepackage{array}
\usepackage{booktabs}
\usepackage{multirow}
\usepackage{makecell}
\usepackage{float}
\newtheorem{theorem}{Theorem}[section]
\newtheorem{proposition}[theorem]{Proposition}

\newtheorem{lemma}[theorem]{Lemma}
\newtheorem{definition}[theorem]{Definition}

\newtheorem{remark}[theorem]{Remark}

\newcommand\R{\mathbb{R}}

\def\Id{\mathop{\mathrm{Id}}\nolimits}
\def\<{{\langle}}
\def\>{{\rangle}}

\numberwithin{equation}{section}

\title[Time-Decay Estimates with Threshold Singularities]
{Time-Decay Estimates for Two-Dimensional Fourth-Order
	Schr\"odinger Operators with Threshold Singularities}

\author{Zijun Wan and Xiaohua Yao\textsuperscript{\dag}}

\address{Zijun Wan, Institute of Applied Physics and Computational Mathematics, Beijing, 100088, P.R. China}
\email{zijunwan@mails.ccnu.edu.cn}

\address{Xiaohua Yao, Department of Mathematics and Key Laboratory of Nonlinear Analysis and Applications(Ministry of Education), Central China Normal University, Wuhan, 430079, P.R. China}
\email{yaoxiaohua@ccnu.edu.cn}

\date{\today}

\begin{document}

\begin{abstract}

We establish time-decay estimates for the two-dimensional fourth-order
Schr\"odinger operator \(H=\Delta^2+V\) with a real-valued decaying
potential \(V\), covering all possible zero-energy threshold
obstructions.
When zero is a regular point or a first-kind resonance, we prove
	\[
	\left\|
	H^{\frac{\alpha}{4}}e^{-itH}P_{\mathrm{ac}}(H)
	\right\|_{L^1\to L^\infty}
	\lesssim
	|t|^{-\frac{2+\alpha}{4}},
	\qquad -2<\alpha\leq2,
	\]
	which matches  with the free sharp decay rate throughout the full range of
	\(\alpha\). 
	For a second-kind resonance, the decay rate is \(|t|^{-(2+\alpha)/4}(\log(2+|t|))^2\) for every \(-2<\alpha\leq2\), with only a logarithmic loss.
	
	For the stronger threshold singularities, we show that the large-time
	behavior is governed by the presence of a \(d\)-wave resonance.
	If zero is a third-kind resonance, or an eigenvalue accompanied by a
	\(d\)-wave resonance, we obtain the sharp decay
	\((\log|t|)^{-1}\) for \(\alpha=0\) and
	\(|t|^{-\alpha/4}(\log|t|)^{-2}\) for \(0<\alpha\leq2\).
	If zero is an eigenvalue without a \(d\)-wave resonance, the
	second-kind estimate
	is recovered for \(-2<\alpha\leq2\).

In addition, in the regular and first-kind resonance cases, we obtain
the logarithmically improved weighted estimate
for every  \(-2<\alpha\leq2\) and \(s>0\):
\[
\left\|
\omega^{-s}
H^{\frac{\alpha}{4}}e^{-itH}P_{\mathrm{ac}}(H)\omega^{-s}
\right\|_{L^1\to L^\infty}
\lesssim
\frac{1}
{|t|^{\frac{2+\alpha}{4}}(\log|t|)^s},
\qquad |t|\geq2,
\]
where \(\omega(x)=\log(2+|x|)\). By contrast, zero is a
second-kind resonance for the free operator \(\Delta^2\), and the
free evolution admits no such logarithmic gain. Thus, in the regular
and first-kind cases, the potential changes the zero-energy spectral
structure of the free operator, and this change is accompanied by
improved weighted decay.


\end{abstract}

	\maketitle
	\tableofcontents
\section{Introduction and main results}
\subsection{Introduction}

In this paper, we study the time decay of solutions to the
two-dimensional fourth-order Schr\"odinger equation
\begin{equation}\label{eq:fourth-schrodinger}
	\begin{cases}
		i\partial_tu(t,x)=(\Delta^2+V)u(t,x),\\
		u(0,x)=f(x),
	\end{cases}
	\qquad (t,x)\in\mathbb R\times\mathbb R^2,
\end{equation}
where \(V\) is a real-valued potential satisfying
\(
|V(x)|\lesssim\langle x\rangle^{-\mu}
\)
for some \(\mu>0\). Setting \(H=\Delta^2+V\), the solution to the equation \eqref{eq:fourth-schrodinger} is
\(u(t)=e^{-itH}f\). Our main interest is the large-time behavior of
this evolution on the absolutely continuous spectral subspace of \(H\).

Fourth-order Schr\"odinger equations arise in nonlinear models
incorporating higher-order dispersion, which plays an important role
in wave propagation and the stability of nonlinear waves; see Karpman
\cite{Karpman96} and Karpman--Shagalov
\cite{KarpmanShagalov00}. A basic homogeneous model is the
biharmonic nonlinear Schr\"odinger equation
\[
i\partial_tu+\Delta^2u+\lambda |u|^{p-1}u=0.
\]
The dispersive properties of the corresponding free evolution were
studied by Ben-Artzi, Koch, and Saut \cite{BKS00}, while the nonlinear
theory, including global well-posedness and scattering, has been
developed, for example, by Pausader \cite{Pausader09},
Miao--Xu--Zhao \cite{MiaoXuZhao09,MiaoXuZhao11}, and
Ruzhansky--Wang--Zhang \cite{RuzhanskyWangZhang16}. These works
motivate the study of how an external potential and the resulting
spectral structure affect the long-time behavior of fourth-order
Schr\"odinger evolutions.

For the free operator \(H_0=\Delta^2\), the sharp kernel estimates
in Ben-Artzi, Koch, and Saut \cite{BKS00} give the natural fourth-order decay
\[
\|e^{-it\Delta^2}\|_{L^1(\mathbb R^n)\to L^\infty(\mathbb R^n)}
\lesssim |t|^{-n/4}.
\]
 In the
two-dimensional fractional formulation relevant to this paper,
\begin{equation}\label{eq:free-alpha-intro}
	\left\|
	(-\Delta)^{\frac{\alpha}{2}}e^{-it\Delta^2}
	\right\|_{L^1(\R^2)\to L^\infty(\R^2)}
	\sim
	|t|^{-\frac{2+\alpha}{4}}, \ \ |t|>2,
	\quad -2<\alpha\leq2.
\end{equation}
The interval \(-2<\alpha\leq2\) is the optimal range for the uniform
\(L^1\to L^\infty\) estimates; see 
Proposition~\ref{free_case}. Thus \eqref{eq:free-alpha-intro} provides the
natural benchmark for the perturbed problem, and leads us to consider the time decay estimate of the following evolutions:
\begin{equation}\label{eq:U-alpha-intro}
	H^{\frac{\alpha}{4}}e^{-itH}P_{\mathrm{ac}}(H),	\qquad -2<\alpha\leq2,
\end{equation}
where \(P_{\mathrm{ac}}(H)\) is the projection onto the absolutely
continuous spectral subspace. The case \(\alpha=0\) is the
Schr\"odinger evolution itself, \(\alpha>0\)
gives derivative estimates,  while \(\alpha<0\) amplifies the contribution of the low-energy spectrum.

The time-decay theory for the classical Schr\"odinger operator
\(-\Delta+V\) has been extensively studied. Early works of Jensen--Kato \cite{JensenKato79} and Jensen \cite{Jensen80} developed
low-energy resolvent expansions and weighted decay estimates, while
Journ\'e--Soffer--Sogge \cite{JSS91} first established
\(L^1\to L^\infty\) decay estimates under suitable spectral
assumptions for $n\ge 3$. The effects of zero-energy resonances and eigenvalues,
together with a general framework for threshold resolvent expansions,
were further developed in subsequent works; see, for example,
\cite{JensenNenciu01,ErdoganSchlag04,ErdoganSchlag06}.
In dimension two, the
logarithmic behavior of the free resolvent near zero leads to a
different threshold structure. Dispersive and weighted estimates in
this setting were obtained by Schlag \cite{Schlag05},
Erdo\u{g}an--Green
\cite{ErdoganGreen13,ErdoganGreenWeighted13}, and Toprak
\cite{Toprak17}.

Higher-order Schr\"odinger operators
\((-\Delta)^m+V\) (\(m\geq2\)) present an additional difficulty at
zero energy. The symbol \(|\xi|^{2m}\) of \((-\Delta)^m\) has a
degenerate critical point at \(\xi=0\), corresponding to zero energy.
Consequently, the classical analysis of Murata \cite{Murata82} for
elliptic symbols \(P(\xi)\) with nondegenerate critical points
\(\xi_0\), satisfying
\(\nabla P(\xi_0)=0\) and
\(\det\nabla^2P(\xi_0)\neq0\), does not directly apply to this
setting. 
Low-energy resolvent expansions and Kato--Jensen-type decay estimates for higher-order elliptic operators were subsequently developed by Feng--Soffer--Wu--Yao \cite{FSWY20}. More recent works have further advanced the dispersive theory for higher-order and fractional Schr\"odinger operators; see \cite{EGG25,CHHZ24a,CHHZ24b,EGGFractionalDisp25}.

For the biharmonic operator \(H=\Delta^2+V\), the threshold behavior
depends strongly on the spatial dimension. Feng--Soffer--Yao
\cite{FSY18} developed resolvent, decay, and Strichartz estimates in
dimension $n=3$ and in dimensions \(n\geq5\). The effects of
zero-energy obstructions were subsequently analyzed in dimension
three by Erdo\u{g}an--Green--Toprak \cite{EGT21} and in dimension
four by Green--Toprak \cite{GreenToprak19}; see also
\cite{GoldbergGreen22,LWWY25} for further refinements in dimension
three. In dimension one, Soffer--Wu--Yao \cite{SWY22} obtained
the sharp \(|t|^{-1/4}\) decay for all zero-energy configurations.
These results illustrate the strong dimension dependence of the
threshold theory for fourth-order operators.

The two-dimensional problem is particularly delicate. It combines
the degeneracy of the fourth-order symbol at zero with the
logarithmic structure characteristic of dimension two. Li--Soffer--Yao \cite{LSY23} classified the possible zero-energy
obstructions, obtained the corresponding low-energy resolvent
expansions, and proved the  \(|t|^{-1/2}\) decay for
\(e^{-itH}P_{\mathrm{ac}}(H)\) when zero is a regular point or a
first-kind resonance. For stronger threshold singularities, however,
their estimates required the insertion of \(H^{1/2}\) in the
second-kind case and \(H^{3/2}\) in the third-kind resonance and
zero-eigenvalue cases. In particular, the decay of the Schr\"odinger
evolution itself remained open in these singular configurations ;
see Remark~\ref{rem:comparison-LSY} for a detailed comparison.

The present paper is devoted to this problem and extends the
two-dimensional theory to the full range
\(-2<\alpha\leq2\). When zero is regular or a first-kind
resonance, we recover the following  sharp decay estimates with regular terms:
\[
\left\|
H^{\frac{\alpha}{4}}e^{-itH}P_{\mathrm{ac}}(H)
\right\|_{L^1(\R^2)\to L^\infty(\R^2)}
\lesssim
|t|^{-\frac{2+\alpha}{4}},
\qquad -2<\alpha\leq2.
\]
For a second-kind resonance, the same polynomial decay holds with only a logarithmic loss
\((\log(2+|t|))^2\).  For the stronger
threshold singularities, we identify the presence of a \(d\)-wave
resonance as the decisive feature of the large-time behavior. In
particular, if zero is a third-kind resonance, or an eigenvalue
accompanied by a \(d\)-wave resonance, the decay is sharply of order
\((\log|t|)^{-1}\) for \(\alpha=0\) and
\(|t|^{-\alpha/4}(\log|t|)^{-2}\) for \(0<\alpha\leq2\).
If zero is an eigenvalue without a \(d\)-wave resonance, the
second-kind estimate is recovered. Thus, in particular, the optimal decay of
\(e^{-itH}P_{\mathrm{ac}}(H)\) is determined for every possible
zero-energy configuration.

We also obtain an additional logarithmic improvement in time decay
in the weighted setting for the regular and first-kind resonance
cases:
\[
\left\|
\omega^{-s}
H^{\frac{\alpha}{4}}e^{-itH}P_{\mathrm{ac}}(H)\omega^{-s}f
\right\|_{L^\infty(\mathbb R^2)}
\lesssim
\frac{1}
{|t|^{\frac{2+\alpha}{4}}(\log|t|)^s}
\left\|f\right\|_{L^1(\mathbb R^2)},
\qquad s>0,\quad |t|\geq2,
\]
where \(\omega(x)=\log(2+|x|)\) and \(-2<\alpha\leq2\).
Although the corresponding unweighted decay agrees with the free
rate, this additional logarithmic gain is absent for
\((-\Delta)^{\alpha/2}e^{-it\Delta^2}\). Related weighted
improvements were obtained for the two-dimensional second-order
Schr\"odinger operator by Erdo\u{g}an--Green
\cite{ErdoganGreenWeighted13} and Toprak \cite{Toprak17}, and for
fourth-order Schr\"odinger operators in dimensions three and four by
Goldberg--Green \cite{GoldbergGreen22} and Green--Toprak
\cite{GreenToprak19}, respectively. These results suggest a broader
relation between improved weighted decay and a weakening of the
zero-energy threshold obstruction relative to the free operator; see
Remark~\ref{remark:threshold-log-gain} for further discussion.

\subsection{Some notations}\label{Notations}
In this subsection, we collect some notations used throughout the paper.
\begin{itemize}

	\item For $a,b\in\mathbb{R}^+$, $a\lesssim b$ (resp. $a\gtrsim b$) means $a\le cb$ (resp. $a\ge cb$) for some $c>0$. In particular, if $b\lesssim a\lesssim b$, we write $a\sim b$. Moreover, for $a\in \mathbb{R}$,   $a\pm$ denote $a\pm \epsilon$   for any arbitrarily small  $\epsilon >0$.
	\vskip0.2cm
	\item We write $f(\lambda) = {O}_k(g(\lambda))$ if
	\[
	\frac{d^\ell}{d\lambda^\ell} f(\lambda) = O\Bigl( \frac{d^\ell}{d\lambda^\ell} g(\lambda) \Bigr), \quad \ell = 0, 1, \cdots,k.
	\]
	\vskip 0.2cm
	\vskip 0.2cm
	
	\item Let $\chi_1 \in C_c^{\infty}(\mathbb{R})$ such that $\chi_1(\lambda)=1$ for $|\lambda|<\lambda_0\ll1$ and $\chi_1(\lambda)=0$ if $|\lambda|>2\lambda_0$, where $\lambda_0$ is some sufficiently small positive constant depending on low energy expansion of $\left( M^{\pm}(\lambda)\right)^{-1}$ in Theorem \ref{thm:M_inverse}. Define $\chi_2(\lambda):=1-\chi_1(\lambda)$ and $\widetilde{\chi}_j(\lambda):=\chi_j(\lambda^4)$  for  $j=1,2.$
	\vskip0.2cm
	\item
	$ \langle f, g \rangle := \int_{\mathbb{R}^2} f(x) \overline{g(x)} \, dx$ denotes the inner product  or the dual pair.
	For $x\in\mathbb{R}^n$, write $\langle x \rangle := 1 + | x |$.
	For $s \in \mathbb{R}$, we  define the weighted $L^2(\mathbb{R}^2) $ spaces:
	$$L^{2,s}(\mathbb{R}^2) := \{ f \in L_{\text {loc}}^2(\mathbb{R}^2) \mid  \langle \cdot \rangle^s f \in L^2(\mathbb{R}^2) \} .$$
\end{itemize}
	\subsection{Main results}\label{main result}
  In the following, we first give the precise definition of zero resonance types of $H=\Delta^2+V(x)$ on $\mathbb{R}^2$.
For $\sigma \in \mathbb{R}$, let $W_{-\sigma}\left(\mathbb{R}^2\right)$ denote the intersection space
$$
W_{-\sigma}\left(\mathbb{R}^2\right)=\bigcap_{s>\sigma} L^{2,-s}\left(\mathbb{R}^2\right) ,
$$
which is increasing in $\sigma\in\R$ and satisfies $L^{2,-\sigma}(\mathbb{R}^2) \subset W_{-\sigma}(\mathbb{R}^2)$. In particular, $W_0\left(\mathbb{R}^2\right) \supset L^2\left(\mathbb{R}^2\right)$.
\begin{definition}\label{definition1}
	{\rm Let $H = \Delta^2 + V(x)$ with $|V(x)| \lesssim \langle x \rangle^{-\mu}$ for some $\mu > 0$. We say that
		\begin{itemize}
			\item[(i)] Zero is a {\it first kind resonance}  of $H$ if the equation $H\phi = 0$  has only zero solution  in $W_{-1}(\mathbb{R}^2)$, but has a nonzero solution  $\phi \in W_{-2}(\mathbb{R}^2)$ and there exists a polynomial $\sum_{|\gamma|=1} C_\gamma x^\gamma$ such that $\phi-\sum_{|\gamma|=1} C_\gamma x^\gamma\in W_{-1}(\mathbb{R}^2)$.  
			\vskip0.1cm
			\item[(ii)] Zero is a {\it second kind resonance} of $H$ if
			the equation $H\phi = 0$ has a nonzero solution  in $W_{-1}(\mathbb{R}^2)$, but only has a zero solution  $\phi$ in $W_{0}(\mathbb{R}^2)$.
			\vskip0.1cm
			\item[(iii)] Zero is a {\it third kind resonance} of $H$ if the equation $H\phi = 0$ has a nonzero solution  in $W_0(\mathbb{R}^2)$, but there exists only a zero solution  $\phi$ in  $L^2(\mathbb{R}^2)$.
			\vskip0.1cm
			\item[(iv)] Zero is an \emph{eigenvalue}  of $H$ if there exists a nonzero $\phi \in L^2(\mathbb{R}^2)$ satisfying $H\phi = 0$.
			\vskip0.1cm
			\item[(v)] Zero is a {\it regular point} of $H$ if $H$ has neither zero eigenvalues nor zero resonances.
	\end{itemize}}
\end{definition}

For later use, we say that \(H\) has a \textbf {\(d\)-wave resonance}
at zero if
$
H\phi=0
$
admits a nonzero solution
$
\phi\in W_0(\mathbb R^2)\setminus L^2(\mathbb R^2).
$
Thus a third-kind resonance necessarily contains a \(d\)-wave
component, whereas a \(d\)-wave resonance may coexist with a zero
eigenvalue. 

The decay assumption on \(V\) in the main results below depends on
the type of the zero-energy threshold. More precisely, we require
\begin{equation}\label{condition}
	\mu>
	\begin{cases}
		11, & \mathbf{k}=0,1,\\
		14, & \mathbf{k}=2,\\
		18, & \mathbf{k}=3,4,
	\end{cases}
\end{equation}
where \(\mathbf{k}=0,1,2,3,4\) correspond respectively to the
regular, first-kind, second-kind, third-kind, and zero-eigenvalue
cases. The conditions in \eqref{condition}, taken from \cite{CCWY},
are used to derive the low-energy resolvent expansion; see also
Theorem~\ref{thm:M_inverse}. By contrast, the high-energy analysis
requires only \(\mu>4\); see
Theorem~\ref{main-theorem-high-schrodinger}. Thus, the decay
assumptions in \eqref{condition} may not be optimal, even in the
regular case.

We are now ready to state our main time-decay estimates. Throughout
the paper, the implicit constants in \(\lesssim\) are allowed to
depend on fixed parameters such as \(\alpha\) and \(s\), and this
dependence will not be indicated explicitly.

\begin{theorem}\label{main_theorem-1}
Let $H = \Delta^2 + V$ and  $|V(x)| \lesssim \langle x \rangle^{-\mu}$ for some $\mu>0$ satisfying the  condition 
\eqref{condition}. 
Assume that $H$ has no positive embedded eigenvalues and 
\(P_{\mathrm{ac}}(H)\) denote the orthogonal
projection onto the absolutely continuous spectral subspace of \(H\). Then the following statements hold:
\vskip0.2cm
\noindent{\rm(i)} If zero is \textbf{a regular point or a first-kind resonance} of $H$,  then 
\begin{equation}\label{regular_first}
	\left\|
	H^{\frac{\alpha}{4}}e^{-itH}
	P_{\mathrm{ac}}(H)
	\right\|_{L^1(\mathbb{R}^2)\to L^\infty(\mathbb{R}^2)}
	\lesssim
	{|t|^{-\frac{2+\alpha}{4}}},\quad -2<\alpha\leq2.
\end{equation}

\noindent {\rm(ii)} If zero is \textbf{a second-kind resonance} of $H$, then 
\begin{align}\label{log-cos-sin}
	\left\|
	H^{\frac{\alpha}{4}}e^{-itH}
	P_{\mathrm{ac}}(H)
	\right\|_{L^1(\mathbb{R}^2)\to L^\infty(\mathbb{R}^2)}
	\lesssim
	\frac{\bigl(\log(2+|t|)\bigr)^2}{|t|^{\frac{2+\alpha}{4}}},\quad -2<\alpha\leq2.\end{align}
    
\noindent{\rm(iii)} If zero is \textbf{a third-kind resonance} of $H$, then  for $|t|\gg1,$
	\begin{equation}\label{sharp-3rd}
	\left\|
	H^{\frac{\alpha}{4}}e^{-itH}
	P_{\mathrm{ac}}(H)
	\right\|_{L^1(\mathbb{R}^2)\to L^\infty(\mathbb{R}^2)} \sim
	\begin{cases}
		\bigl(\log |t|\bigr)^{-1}, & \ \  \alpha = 0,\\[6pt]
		|t|^{-\frac{\alpha}{4}}(\log|t|)^{-2}, & \ \   0 < \alpha \le 2.
	\end{cases}
\end{equation}
	\noindent{\rm(iv)} If zero is \textbf{an eigenvalue}  of \(H\),  the decay depends on whether a \(d\)-wave resonance is present:
\vskip0.1cm
\begin{itemize}
	\item
	If a \(d\)-wave resonance is present, then
	the estimate \eqref{sharp-3rd} holds for every 
	\(0\leq\alpha\leq2\).
	\vskip0.2cm
	\item
	If no \(d\)-wave resonance is present, then
	the estimate \eqref{log-cos-sin} holds for every 
	\(-2<\alpha\leq2\).
\end{itemize}
\end{theorem}
\begin{remark}\label{rem:comparison-LSY}
{\rm Theorem~\ref{main_theorem-1} substantially extends the
	two-dimensional dispersive results of Li--Soffer--Yao
	\cite{LSY23}, especially in the presence of stronger threshold
	singularities. These improvements rely crucially on a different approach to the
	perturbed resolvent expansions from that in \cite{LSY23}, together
	with more refined oscillatory integral estimates. The comparison is as follows.
	
	\begin{itemize}

		\item
		For a regular point or a first-kind resonance, Li et al.~\cite{LSY23}
		proved the optimal \(O(|t|^{-1/2})\) decay for \(\alpha=0\) and
		observed that their argument should extend to \(0\leq\alpha\leq2\).
		Theorem ~\ref{main_theorem-1} (i) enlarges the range to
		\(-2<\alpha\leq2\). The negative-order regime \(-2<\alpha<0\) is
		new, and the perturbed evolution therefore attains the same full
		spectral range and decay rates as in the free case.
		\vskip0.1cm
		\item
		For a second-kind resonance, the stronger zero-energy singularity
		prevents Li et al. ~\cite{LSY23} from directly estimating the low-energy
		part of \(e^{-itH}P_{\mathrm{ac}}(H)\). They instead introduced
		the regularizing factor \(H^{1/2}\) and proved
		\[
		\left\|
		H^{1/2}e^{-itH}P_{\mathrm{ac}}(H)\chi_1(H)
		\right\|_{L^1\to L^\infty}
		+
		\left\|
		e^{-itH}\chi_2(H)
		\right\|_{L^1\to L^\infty}
		\lesssim |t|^{-1/2}.
		\]
		Theorem~\ref{main_theorem-1} (ii) directly includes the 
		case \(\alpha=0\) and, more generally, the full range
		\(-2<\alpha\leq2\). In particular, for \(\alpha=0\)
		we recover the free polynomial decay rate
		\(|t|^{-1/2}\), with only a logarithmic loss, while at the same regularity level
		\(H^{1/2}\) (\(\alpha=2\)) used in \cite{LSY23}, the polynomial
		decay is improved from \(|t|^{-1/2}\) to \(|t|^{-1}\), again up
		to the logarithmic loss.
			\vskip0.1cm
		\item
		For a third-kind resonance or a zero eigenvalue, the threshold
		singularity is even stronger, and Li et al. ~\cite{LSY23} again did not
		directly estimate the low-energy evolution. They introduced
		\(H^{3/2}\) and proved
		\[
		\left\|
		H^{3/2}e^{-itH}P_{\mathrm{ac}}(H)\chi_1(H)
		\right\|_{L^1\to L^\infty}
		+
		\left\|
		e^{-itH}\chi_2(H)
		\right\|_{L^1\to L^\infty}
		\lesssim |t|^{-1/2}.
		\]
		Theorem~\ref{main_theorem-1} (iii)-(iv) determines instead the decay of the
		evolution itself.  The sharp estimate
		\eqref{sharp-3rd} shows that the presence of a \(d\)-wave resonance (i.e.,
		the third-kind case or zero eigenvalue with  \(d\)-wave) reduces  to the logarithmic rate \((\log|t|)^{-1}\). By contrast, if zero is an eigenvalue without
		a \(d\)-wave resonance, the stronger second-kind estimate
		\eqref{log-cos-sin} holds. Thus the \(d\)-wave component is the
		decisive obstruction to dispersive decay.
	\end{itemize}
	
}
\end{remark}

We next complement Theorem~\ref{main_theorem-1} with a weighted
refinement which is not in \cite{LSY23}. 
\begin{theorem}\label{main2-weight}
	Assume $H=\Delta^2+V$ satisfy the hypotheses of Theorem~\ref{main_theorem-1}, and suppose
	that zero is a regular point or a first-kind resonance of \(H\). Then, for every fixed  \(s>0\) and \(-2<\alpha\leq2\),
	\begin{equation}\label{regular_first_weight}
		\left\|
\omega^{-s}
H^{\frac{\alpha}{4}}e^{-itH}P_{\mathrm{ac}}(H)\omega^{-s}f
\right\|_{L^\infty(\mathbb{R}^2)}
\lesssim
\frac{1}
{|t|^{\frac{2+\alpha}{4}}(\log|t|)^s}
\left\|f\right\|_{L^1(\mathbb{R}^2)},
\qquad
|t|\geq2,
	\end{equation}
    where $\omega(x) = \log(2+|x|)$ is the logarithmic weight function.
\end{theorem}

\begin{remark}\label{remark:regular-first}
	{\rm
		The logarithmic gain in \eqref{regular_first_weight} is a genuinely
		perturbative effect. Indeed, the corresponding free 
		evolution has the expansion for every  $-2<\alpha\leq2$ and $|t|\geq2$,
		\[
		\begin{aligned}
			&\bigl[
			(-\Delta)^{\alpha/2}e^{-it\Delta^2}
			\bigr](x,y)
	=
			\frac{1}{8\pi}
			\Gamma\left(\frac{2+\alpha}{4}\right)
			e^{-i\frac{(2+\alpha)\pi}{8}\operatorname{sgn}(t)}
			|t|^{-\frac{2+\alpha}{4}}
			+
			O\left(
			\frac{\bigl(\omega(x)\omega(y)\bigr)^s}
			{|t|^{\frac{2+\alpha}{4}}(\log|t|)^s}
			\right),
		\end{aligned}
		\]
        where \(\Gamma\) denotes the Gamma function,
		see Proposition~\ref{free-energy-decomposition}.
		Thus the free evolution contains a nonzero leading term of order
		\(|t|^{-(2+\alpha)/4}\), and no logarithmic improvement is available
		at this level:
\begin{align*}
	\left\|
\omega^{-s}
(-\Delta)^{\alpha/2}e^{-it\Delta^2}\omega^{-s}
\right\|_{L^1\to L^\infty}
\sim
|t|^{-\frac{2+\alpha}{4}},
\qquad
|t|\geq2.
      \end{align*}  
By contrast, in the regular and first-kind resonance cases, the
		perturbed resolvent expansion cancels this leading threshold
		contribution, yielding \eqref{regular_first_weight}. Hence, although the unweighted
		decay in Theorem~\ref{main_theorem-1}(i) agrees with the free decay,
		the weighted evolution of the perturbed operator decays strictly
		faster by a logarithmic factor.
	}
\end{remark}

\begin{remark}\label{remark:threshold-log-gain}
	{\rm
		A similar logarithmic improvement already appears for the
		two-dimensional Schr\"odinger operator \(H=-\Delta+V\) where zero is a first-kind resonance of $-\Delta.$
		Erdo\u{g}an--Green \cite{ErdoganGreenWeighted13} proved that,
		when zero is a regular point  of \(H\),
		\[
\left\|\omega^{-2}e^{itH}P_{\mathrm{ac}}(H) \omega^{-2}f\right\|_{L^\infty(\mathbb{R}^2)}
		\lesssim
		\frac{1}{|t|(\log|t|)^2}\|f\|_{L^1(\mathbb{R}^2)},
\ \ \omega(x) = \log(2+|x|).	\]
		For the two-dimensional bi-Laplacian $\Delta^2$, zero is a second-kind resonance
		of the free operator, while Theorem~\ref{main2-weight} gives an
		additional  logarithmic improvement when zero is regular or a
		first-kind resonance of $H=\Delta^2+V$.
		Related weighted improvements for fourth-order operators are also
	known when
		zero is  a regular  point in dimensions $n=3,4$, see Goldberg--Green
\cite{GoldbergGreen22} and Green--Toprak \cite{GreenToprak19}.
		
		These results suggest that additional weighted time decay may occur when the
		perturbation weakens  the natural threshold resonance of
		the free operator.  Accordingly,
	in dimension $n=1$, where zero is a second-kind resonance of $\Delta^2$, we expect analogous weighted improvements in the regular
	and first-kind cases. In contrast, for \(n\geq5\), where  zero is a regular point
	of $\Delta^2$, such an improvement is not expected
	in general.
	}
\end{remark}

\subsection{The ideas of the proof}
In this subsection, we outline the main ideas underlying the proofs of our main theorems.
The starting point is the Stone formulas:
\begin{align*}
H^{\frac{\alpha}{4}}e^{-itH} P_{\mathrm{ac}}(H) & =\frac{2}{\pi i} \int_0^{\infty} \lambda^{3+\alpha} e^{-it\lambda^4}\left[R_V^{+}(\lambda^4)-R_V^{-}(\lambda^4)\right]  d \lambda,
\end{align*}
where $-2<\alpha\leq2.$ To make these identities precise, we recall that $R_0(z)=(\Delta^2-z)^{-1}$ and $R_V(z)=(H-z)^{-1}$ denote the resolvents of the free operator $\Delta^2$ and the perturbed operator $H=\Delta^2+V$, respectively, for $z\in \mathbb{C}\setminus[0,\infty)$. Their boundary values (limiting resolvents) on $(0,\infty)$ are defined as
\begin{align*}
	R^\pm_0(\lambda)=\lim_{\epsilon \searrow 0}R_0(\lambda\pm i\epsilon ),\qquad  R_V^\pm(\lambda)=\lim_{\epsilon  \searrow 0}R_V(\lambda\pm i\epsilon ),\quad \lambda>0.
\end{align*}
The existence of $R^\pm_0(\lambda)$ as bounded operators from $L^{2,s}(\mathbb{R}^2)$ to $L^{2,-s}(\mathbb{R}^2)$ for any $s>1/2$ follows from the limiting absorption principle for the resolvent $(-\Delta-z)^{-1}$ (see, e.g., Agmon \cite{Agmon}) and the splitting identity
$$
R_0(z)=\frac{1}{2\sqrt z}\left[(-\Delta-\sqrt z)^{-1}-(-\Delta+\sqrt z)^{-1}\right],\quad z\in \mathbb{C}\setminus[0,\infty), \ \Im\sqrt z>0.
$$
The existence of $R^\pm_V(\lambda)$ under suitable potential decay conditions is established in the existing literature (see, e.g., \cite{FSY18}).

To evaluate the Stone formulas, we decompose the integrals into a low-energy part ($0 < \lambda \ll 1$) and a high-energy part ($\lambda \gtrsim 1$). The high-energy part is straightforward to treat since the free resolvent (see \eqref{reso-big} and \eqref{reso_small}) possesses no singularities for $\lambda \gtrsim 1$. The primary analytical difficulty lies in the low-energy part, which requires a delicate analysis of the singularities of $R_V^\pm(\lambda^4)$ near $\lambda=0$. 

Setting $v(x) = \sqrt{|V(x)|}$ and $U(x) = \operatorname{sgn}(V(x))$, we define the Birman-Schwinger operator $M^\pm(\lambda) = U + vR_0^\pm(\lambda^4)v$. The symmetric resolvent identity,
\begin{equation}\label{resolvent identity}
	R_V^\pm(\lambda^4) = R_0^\pm(\lambda^4) - R_0^\pm(\lambda^4) v \big(M^\pm(\lambda)\big)^{-1} v R_0^\pm(\lambda^4),
\end{equation}
reduces the problem to analyzing the asymptotic expansion of $(M^{\pm}(\lambda))^{-1}$ as $\lambda \to 0^+$. For convenience, we use $\mathbf{k} \in \{0,1,2,3,4\}$ to denote zero energy types below: $\mathbf{k}=0$ (regular), $\mathbf{k}=1,2,3$ (resonances), $\mathbf{k}=4$ (eigenvalue).

If zero is a resonance of the $\mathbf{k}$-th kind ($0\leq\mathbf{k}\leq 4$), Theorem~\ref{thm:M_inverse} guarantees that $(M^\pm(\lambda))^{-1}$ admits the following expansion:
\begin{equation}\label{asymptotic expansion}
	\bigl(M^{\pm}(\lambda)\bigr)^{-1}= \sum_{0\leq j,l\leq\mathbf{k}+2}
	\lambda^{2-k_j-k_l} Q_j \mathcal{M}_{j,l}^{\pm}(\lambda) Q_l,
\end{equation}
where the singularity exponents $k_j $ are given by
	\begin{equation*}
	\begin{array}{c|ccccccc}  
		j & 0 & 1 & 2 & 3 & 4 & 5 & 6 \\  
		\hline  
		k_j & 0 & 1 & 1 & 1 & 2 & 3 & 3  
	\end{array}  
\end{equation*}   The operators $\mathcal{M}_{j,l}^{\pm}(\lambda)\in\mathbb{B}(L^2)$ exhibit additional decay as $\lambda \to 0^+$ for certain indices $j, l$. The orthogonal projections $Q_j$ satisfy the important cancellation properties (see Remark~\ref{cancellationQ}), which are applied to  establish Lemma~\ref{lemma_projection} to analyze the expansion for  $Q_j v R_0^\pm(\lambda^4)$. Lemma~\ref{lemma_projection} demonstrates how the projections extract compensatory positive powers of $\lambda$, effectively annihilating the zero-energy singularities of the free resolvent.

Substituting the identities \eqref{resolvent identity} and \eqref{asymptotic expansion} into the low-energy part of the Stone formulas, we reduce the analysis to bounding a finite number of kernel differences
\begin{equation*}
	\big(\mathcal{I}_{B}^+ - \mathcal{I}_{B}^-\big)(\alpha; t,x,y),
\end{equation*}
where the component kernels are given by
\begin{align*}
	\mathcal{I}_{B}^\pm(\alpha; t,x,y) &= \frac{2}{\pi i}\int_0^{\infty} \cos(t \lambda^2)\lambda \widetilde{\chi}_1(\lambda) \left[\lambda^{4-k_j-k_l} R_0^\pm(\lambda^{4})v Q_j B^\pm(\lambda)Q_lvR_0^\pm(\lambda^{4})\right](x,y) \, d \lambda,
\end{align*}
with $B^\pm(\lambda) \in \{ \mathcal{M}_{j, l}^{\pm}(\lambda) \mid 0 \le j,l \le \mathbf{k}+2 \}$. By applying the expansion of $(Q_j v R_0^\pm(\lambda^4))(x,y)$ established in Lemma~\ref{lemma_projection}, the intricate integral $\mathcal{I}_{B}^\pm$  can be reduced to canonical oscillatory forms systematically studied in Lemma~\ref{quartic-oscillatory}.
 Consequently, the desired temporal decay estimates are ultimately derived by invoking the  bounds from Lemma~\ref{quartic-oscillatory}.

To illustrate our general strategy for controlling these oscillatory integrals, let us examine a representative term where $B^\pm(\lambda)=\mathcal{M}_{1,1}^{\pm}(\lambda)$.
Using the Taylor expansion for $F_\pm$ and the cancellation property $Q_1 v = 0$, we extract an additional positive power of $\lambda$ (see Lemma~\ref{lemma_projection}):
$$
\bigl(Q_1 v R_0^\pm(\lambda^4)\bigr)(x,y)
= -\lambda^{-1} Q_1\Bigl(|\cdot| v \int_0^1 F_\pm'(\lambda|\theta\cdot-y|)\cos\xi\,d\theta\Bigr)(x)
=: \lambda^{-1} \Omega_{1,\pm}(\lambda,x,y),
$$
where \(\xi\) denotes the angle between the vectors
\(\cdot\) and \(\theta\cdot-y\). Moreover,
$R^\pm_0(\lambda^4)(x,y)=\lambda^{-2}F_\pm(\lambda|x-y|)$ (see \eqref{def:F_pm}). Hence, the term $Q_1 v R_0^\pm(\lambda^4)$ is less singular than the free resolvent. Furthermore, we derive 
\[
\bigl\|\partial_\lambda^\ell\bigl(e^{\mp i\lambda|y|} \Omega_{1,\pm}(\lambda,\cdot,y)\bigr)\bigr\|_{L^2}
\lesssim \langle \lambda y\rangle^{-1/2} \lambda^{-\ell}, \qquad \ell=0,1,2.
\]
With these properties, the kernels can be recast into the canonical forms appearing in Lemma~\ref{quartic-oscillatory}:
\begin{align*}
	\mathcal{I}_B^\pm(\alpha; t,x,y) &= \frac{2}{\pi i}\int_0^{\infty} e^{-it\lambda^4}e^{\pm i\lambda r} \lambda \mathcal{E}_B^\pm(\alpha; \lambda, x,y) \, d\lambda, 
\end{align*}
with  $r(x,y) = |x| + |y|$ and the effective amplitude factor $$\mathcal{E}_{\mathcal{M}_{1,1}}^\pm(\alpha; \lambda,x,y) = e^{\mp i\lambda r}\lambda^\alpha \widetilde{\chi}_1(\lambda) \langle \mathcal{M}_{1,1}^{\pm}(\lambda) \Omega_{1,\pm}(\cdot, y), \Omega_{1,\mp}(\cdot, x) \rangle.$$ Applying Leibniz's rule along with the bound $\|\partial_\lambda^{\ell} \mathcal{M}_{1,1}^{\pm}(\lambda)\|_{\mathbb{B}(L^2)}\lesssim\lambda^{-\ell}$ yields
\[
\bigl|\partial_\lambda^{\ell} \mathcal{E}_{\mathcal{M}_{1,1}}^\pm(\alpha; \lambda,x,y) \bigr| \lesssim \bigl\langle \lambda (|x|+|y|)\bigr\rangle^{-1/2} \lambda^{\alpha-\ell} \widetilde{\chi}_1(\lambda/2),  \qquad  \ell=0,1,2.
\]
Then $\mathcal{E}_{\mathcal{M}_{1,1}}^\pm$ satisfies condition \eqref{quartic-cond-2} of Lemma~\ref{quartic-oscillatory} (ii) with parameters $(\sigma, \nu)=(\alpha,0)$ for $-2<\alpha\leq2$, yielding  the bound $|\mathcal{I}_{\mathcal{M}_{1,1}}^\pm| \lesssim \langle t \rangle^{-(2+\alpha)/4}$ uniformly in $x, y$.

To establish the spatially weighted estimates for $|t| \ge 2$, we utilize the following expansion:
\begin{align*}
	R^\pm_0(\lambda^4)(x,y) &= \frac{a_0^\pm}{\lambda^2} + \bigl(a_1^\pm+b_0\log\lambda\bigr)|x-y|^2 + b_0|x-y|^2\log|x-y| + {O}_2\big(\lambda^2|x-y|^4\big) \nonumber \\
	&=: \frac{a_0^\pm}{\lambda^2} + E_1^\pm(\lambda, x, y).
\end{align*}
The cancellation property $Q_1 v = 0$ annihilates the highly singular leading term $\lambda^{-2}$, meaning $\bigl( Q_1 v R_0^{\pm}(\lambda^4) \bigr) = \bigl( Q_1 v E_1^{\pm} \bigr)$. This cancellation mitigates the low-energy singularity by a factor of $\lambda^{2-}$ relative to $R_0^\pm$ but introduces  spatial growth.  Provided that the potential $V$ decays sufficiently fast, we deduce the  bound:
\[
\|\partial_\lambda^\ell  \bigl(Q_1 v R_0^{\pm}(\lambda^4)\bigr)(\cdot,y)\|_{L^2}
\lesssim \|v \partial_\lambda^\ell E_1^\pm(\lambda,\cdot,y)\|_{L^2}
\lesssim \langle y\rangle^4\lambda^{-\ell-}, \quad \ell = 0, 1, 2.
\]
In this weighted scenario, we align the phase to $r(x,y)=0$ and redefine the amplitude factor as
$$
\mathcal{E}_{\mathcal{M}_{1,1}}^\pm(\alpha; \lambda,x,y)
= \widetilde{\chi}_1(\lambda) \lambda^{2+\alpha}
\bigl\langle \mathcal{M}_{1,1}^\pm(\lambda) \bigl( Q_1 v E_1^{\pm} \bigr)(\cdot,y),
\bigl( Q_1 v E_1^{\mp} \bigr)(\cdot,x) \bigr\rangle.
$$
Consequently, we derive the weighted bound
\[
\bigl| \partial_\lambda^\ell \mathcal{E}_{\mathcal{M}_{1,1}}^\pm(\alpha; \lambda,x,y) \bigr|
\lesssim \lambda^{1/2+\alpha-\ell}\widetilde{\chi}_1(\lambda/2) \langle x\rangle^{4} \langle y\rangle^{4}, \quad   \ell=0,1,2.
\]
An application of Lemma~\ref{quartic-oscillatory}(ii) with parameters $(\sigma, \nu)=(\alpha+1/2,0)$ and  $h(x,y)=\langle x\rangle^{4} \langle y\rangle^{4}$ yields the improved temporal decay $|\mathcal{I}_{\mathcal{M}_{1,1}}^\pm|\lesssim 	|t|^{-\frac{5+2\alpha}{8}}
\langle x\rangle^4
\langle y\rangle^4$. Combining this with the uniform $O(|t|^{-(2+\alpha)/4})$ bound via the  inequality \( \min(1, a/b) \lesssim (\log a/\log b)^s \) (valid for \( a,b \ge 2 \) and fixed $s>0$), we  obtain the logarithmically weighted bound for every fixed $s>0$ and $-2<\alpha\leq2$:
$$
	|t|^{-\frac{2+\alpha}{4}}
\min\left\{
1,
\frac{\langle x\rangle^4\langle y\rangle^4}{|t|^{1/8}}
\right\}
\lesssim
\frac{\bigl(\omega(x)\omega(y)\bigr)^s}
{|t|^{\frac{2+\alpha}{4}}(\log|t|)^s},
\qquad |t|\geq2.
$$
which gives the desired result $\bigl(\omega(x)\omega(y)\bigr)^s|t|^{-\frac{2+\alpha}{4}}(\log|t|)^{-s}$.

For the remaining operators $B^\pm(\lambda) \in \{\mathcal{M}_{j,l}^{\pm}(\lambda) \mid 0 \le j,l \le \mathbf{k}+2\}$, the essential strategy remains unchanged. By exploiting the structural bounds of $Q_j v R_0^{\pm}(\lambda^4)$ in conjunction with the operator norm bounds for $\mathcal{M}_{j, l}^{\pm}(\lambda)$, the analysis is reduced to estimating canonical oscillatory integrals in Lemma~\ref{quartic-oscillatory}.
However, the stronger
low-energy singularities and the interaction among the corresponding
projection blocks require a more delicate term-by-term analysis when higher-order resonances occur ($\mathbf{k} \ge 2$).   Despite these  technical challenges, our unified framework  resolves all threshold singularities, thereby providing a complete picture of the time-decay behavior for  bi-Sch\"odinger  group $H^{\alpha/4}e^{-itH}$   across all resonance scenarios.

\section{Resolvent expansions}\label{sec:reso-expan}
\subsection{Free resolvent  }\label{subsec:free_resolvent}
In this subsection, we derive the asymptotic expansions of the free resolvent $R_0^\pm(\lambda^4)$ for the bi-Laplace operator $\Delta^2$.

Recall the expression of the free resolvent for the Laplacian (see, e.g., \cite{ErdoganGreen13}):
\begin{equation}\label{resolvent-R}
	R^\pm (-\Delta;\lambda^2)(x,y) = \pm\frac{i}{4}H_0^\pm(\lambda|x-y|),
\end{equation}
where $H_0^\pm (z)= J_0(z)\pm i Y_0(z) $ are Hankel functions of order zero. From the series expansions for Bessel functions
(see, e.g., \cite{AS64}), we have
\begin{gather*}
    J_0(z)=1-\frac14z^2+\frac1{64}z^4-\frac1{2304}z^6+ \frac{1}{147456}z^8+O(z^{10}),\\
    Y_0(z)=\frac{2}{\pi}\bigl(\log\frac{z}{2}+\gamma\bigr)J_0(z)
          +\frac{2}{\pi}\Bigl(\frac14z^2-\frac3{128}z^4+\frac{11}{13824}z^6-\frac{25}{1769472}z^8+O(z^{10})\Bigr),
\end{gather*}
where $\gamma $ is Euler's constant. In addition, for $|z| \gtrsim 1$,
\begin{equation}\label{id-H0big}
H^\pm_0(z)=e^{\pm iz}w_\pm(z),\qquad |w_\pm^{(\ell)}(z)|\lesssim(1+|z|)^{-1/2-\ell}, \
\ell= 0,1,2,\cdots.
\end{equation}

Utilizing  the identity
\begin{equation*}
    R^\pm_0(\lambda^4) = \frac{1}{2\lambda^2} \left( R^\pm(-\Delta; \lambda^2) - R(-\Delta; -\lambda^2) \right), \quad \lambda > 0.
\end{equation*}
Together with  \eqref{resolvent-R} and \eqref{id-H0big}, it follows that
\begin{equation}\label{reso-big}
    R_0^\pm(\lambda^4)(x,y) = \frac{i}{8\lambda^2} \left[ \pm e^{\pm i\lambda |x - y|} w_\pm(\lambda|x - y|) - e^{-\lambda |x - y|} w_+(i\lambda|x - y|) \right],\ \ \  \lambda|x - y| \gtrsim 1.
\end{equation}
 For $\lambda|x-y| \ll 1 $,  the series expansions of Bessel functions give
   \begin{equation}\label{reso_small}
   \begin{split}
    R^\pm_0(\lambda^4)(x,y)=&\frac{a_0^\pm}{\lambda^2}G_0(x,y)+g_0^\pm(\lambda)G_2(x,y)+b_0G_2^0(x,y)+a_2^\pm\lambda^2G_4(x,y)\\
&+g_1^\pm(\lambda)\lambda^4G_6(x,y)+b_1\lambda^4G_6^0(x,y)+a_4^\pm\lambda^6G_8(x,y)
    +O_2\bigl(\lambda^{8-\varepsilon}|x-y|^{10-\varepsilon}\bigr),
    \end{split}
    \end{equation}
where
$ b_0 = \frac{1}{8\pi}, \ b_1 = \frac{1}{4608\pi},\
a_0^\pm = \pm \frac{i}{8}, \  
a_2^\pm = \pm \frac{i}{512}, \      
a_4^\pm=\pm \frac{i}{1179648},$ and
$$
g_0^\pm(\lambda)=b_0\log\lambda+a_1^\pm, \ g_1^\pm(\lambda)=b_1\log\lambda+a_3^\pm,\ \ a_1^\pm, a_3^\pm\in\mathbb{C}\setminus\mathbb{R}, $$
\begin{equation}\label{def:G_k}
   G_{k}(x,y)= |x-y|^k, \,
   G_k^0(x,y)= |x-y|^k\log|x-y|.
\end{equation}
Combining \eqref{reso-big} and \eqref{reso_small}, we obtain the following expansion of $R^\pm_0(\lambda^4)(x,y)$ for any $\lambda > 0$:
\begin{lemma}[{\cite[Lemma 2.2]{LSY23}}]\label{lem-reso}
For any $\lambda> 0 $, we have the following expansions:
\begin{equation}\label{free-R0pmlambda4}
   \begin{split}
	  R^\pm_0(\lambda^4)(x,y)=&\frac{a_0^\pm}{\lambda^2}G_0(x,y)+g_0^\pm(\lambda)G_2(x,y)+b_0G_2^0(x,y)+a_2^\pm\lambda^2G_4(x,y)\\
&+g_1^\pm(\lambda)\lambda^4G_6(x,y)+b_1\lambda^4G_6^0(x,y)+O_2\bigl(\lambda^{6}|x-y|^{8}\bigr).\nonumber
		\end{split}
	\end{equation}
\end{lemma}
For convenience in the subsequent analysis, we define
\begin{equation}\label{def:F_pm}
R^\pm_0(\lambda^4)(x,y):=\lambda^{-2}F_\pm(\lambda|x-y|),\quad
\widetilde{R}^\pm(\lambda|x-y|):=\lambda^{2}e^{\mp i\lambda|x-y|}R^\pm_0(\lambda^4)(x,y).
\end{equation}
Then $F_\pm(p)=e^{\pm ip}\widetilde{R}^\pm(p)$. From \eqref{reso-big} and \eqref{reso_small},    we have
\begin{equation}\label{free-widetilde_R}
\begin{aligned}
\widetilde{R}^\pm(p)=&\,e^{\mp ip}\Bigl[a_0^\pm+b_0p^2\log p+a_1^\pm p^2
                     +a_2^\pm p^4+b_1p^6\log p+a_3^\pm p^6+O_2\bigl(p^{8}\bigr)\Bigr]\chi_1(p)\\
                     &+\frac{i}{8}\Bigl(\pm w_\pm(p)-e^{\mp ip}e^{-p}w_+(ip)\Bigr)\chi_2(p).
\end{aligned}
\end{equation}

In the sequel,  we denote by $G_k^0$ (resp. $G_k$) the integral operator with kernel $G_k^0(x,y)$ (resp. $G_k(x,y)$) defined in \eqref{def:G_k}.
In particular, $b_0 G_2^0$ is a fundamental solution of $\Delta^2$, i.e., $\Delta^2 (b_0 G_2^0) = \delta$ in the sense of distributions. Here \(\delta\) denotes the Dirac
delta distribution.

\subsection{A characterization of zero resonances} In this subsection, we recall the projection-theoretic characterization
of the zero-energy spectral types introduced in
Definition~\ref{definition1}, established in \cite{LSY23,CCWY}.  
\begin{definition}\label{defS}
	Let $x = (x_1, x_2) \in \mathbb{R}^2$, $v=\sqrt{|V|}$.  Define $P = \|V\|_{L^1(\mathbb{R}^2)}^{-1} \langle \cdot, v\rangle v$ and $T_0 : = U + b_0 v G_2^0 v $. Set  $S_{-1} = \Id$, $S_6 = 0$, and for $j=0,\dots,5$, let $S_j$ be the orthogonal projection onto the subspace $S_jL^2$ defined as follows:
	\begin{itemize}
			\item $S_0L^2 = \{ f \in L^2 \mid \langle  v, f \rangle = 0\}$.
		\item $S_1L^2 = \{ f \in S_0L^2 \mid \langle x_j v, f \rangle = 0,\ j = 1,2 \}$.
		\vskip0.1cm
		\item $S_2L^2 = \{ f \in S_1L^2 \mid S_1T_0f = 0 \}$.
		\vskip0.1cm
		\item $S_3L^2 = \{ f \in S_2L^2 \mid S_0T_0f = 0 \}$.
		\vskip0.1cm
		\item $S_4^0L^2 = \{ f \in S_3L^2 \mid \langle |x|^2 v, f \rangle = 0 \}$.
        \vskip0.1cm
		\item $S_4L^2 = \{ f \in S_3L^2 \mid \langle x_ix_j v, f \rangle = 0,\ i,j = 1,2,\ \text{and}\ PT_0f = 0 \}$.
		\vskip0.1cm
		\item $S_5L^2 = \{ f \in S_4L^2 \mid \langle x_ix_jx_k v, f \rangle = 0,\ i,j,k = 1,2 \}$.
	\end{itemize}
\end{definition}

Note that these projections $S_j$ are well-defined due to the decay condition \eqref{condition} on $V$.  Since  $vG_2^0 v$ is a Hilbert-Schmidt operator and
$$
 T_0  = U + b_0 v G_2^0 v
$$
  is a compact perturbation of $U$,  Fredholm alternative then ensures that $S_2$ is a finite-rank projection. Moreover, observe that $P+S_0=\Id$, then it follows that $S_jT_0=T_0S_j=0$ for $j=4,5.$

From the definitions above, we have the chain of inclusions
\[
S_0 L^2 \supseteq S_1 L^2 \supseteq S_2 L^2 \supseteq S_3 L^2 \supseteq S_4^0 L^2   \supseteq S_4 L^2 \supseteq S_5 L^2,
\]
and the projections $S_j$ (for $j=2,\cdots,5$) and $S_4^0$ are finite-rank operators.

We now introduce the orthogonal projections $Q_j$ associated with the
successive subspaces \(S_jL^2\).

\begin{definition}\label{defQ}
Define
\[
Q_j:=S_{j-1}-S_j,\qquad j=0,1,\ldots,6,
\]
and decompose \(Q_4\) as
\[
Q_4=Q_4^0+Q_4^1,\qquad
Q_4^0:=S_3-S_4^0,\qquad
Q_4^1:=S_4^0-S_4.
\]
\end{definition}
In particular, since \(S_{-1}=\Id\) and \(S_0=\Id-P\), we have
$
Q_0=P.
$
Moreover, the nested structure of the subspaces \(S_jL^2\) implies
that the projections \(Q_j\), \(0\leq j\leq6\), are mutually
orthogonal:
\[
Q_jQ_l=0,\qquad j\neq l.
\]
Since \(S_4L^2\subseteq S_4^0L^2\subseteq S_3L^2\), we also have
\[
Q_4Q_4^0=Q_4^0Q_4=Q_4^0,\qquad
Q_4Q_4^1=Q_4^1Q_4=Q_4^1,\qquad
Q_4^0Q_4^1=Q_4^1Q_4^0=0.
\]
  Definition \ref{defS} and Definition \ref{defQ}  immediately yield the
 following moment cancellations.
\begin{remark}\label{cancellationQ}
	The orthogonal projections $Q_4^0$, $Q_4^1$, and $Q_\beta$ (for $\beta \in \mathbb{N}_0$ with $1 \le \beta \le 6$) satisfy:
	\begin{align*}
		& Q_{1}v = 0 \quad (\text{i.e., } \langle v, Q_{1}f \rangle = 0,\ \forall f \in L^2), \\
		& Q_{\beta}v = Q_{\beta}(x_jv) =Q_{4}^{0}v = Q_{4}^{0}(x_jv) =  0 \quad (\beta = 2,3,4), \\
		& Q_{4}^{1}v = Q_{4}^{1}(x_jv )= Q_{4}^{1}(|x|^2v )= 0,    \\
        & Q_{5}v = Q_{5}(x_jv) = Q_{5}(x_ix_jv) = 0,\quad Q_5T_0=T_0Q_5=0,\\
		& Q_{6}v = Q_{6}(x_jv) = Q_{6}(x_ix_jv) = Q_{6}(x_ix_jx_kv) = 0,\ Q_6T_0=T_0Q_6=0.
	\end{align*}
\end{remark}

		The following proposition relates the successive subspaces
		\(S_2,S_3,S_4,S_5\) to progressively stronger weighted
		\(L^2\)-integrability conditions on the corresponding distributional
		solutions of \(H\phi=0\).
		\begin{proposition}\label{characterizations-1}\cite[Proposition 8.1]{CCWY}
			Assume that $H = \Delta^2 + V$ and $|V(x)| \lesssim \langle x \rangle^{-\mu}$ with $\mu > 0$.
			\begin{itemize}
				\item[(i)] If $\mu > 11$, then
				$\psi \in S_{2} L^2$ if and only if $\psi = Uv\phi$ with $\phi \in W_{-2}(\mathbb{R}^2)$ satisfying $H\phi = 0$ and
				\begin{equation}\label{phi-phi0}
					\phi(x) = \phi_0(x) + \sum_{|\gamma|=1} C_\gamma x^\gamma \quad \text{for some } \phi_0 \in W_{-1}(\mathbb{R}^2).
				\end{equation}
				\item[(ii)] If $\mu > 14$, then
				$\psi \in S_{3} L^2$ if and only if $\psi = Uv\phi$ with $\phi \in W_{-1}(\mathbb{R}^2)$ satisfying $H\phi = 0$.
				\item[(iii)] If $\mu > 18$, then
				$\psi \in S_{4} L^2$ if and only if $\psi = Uv\phi$ with $\phi \in W_{0}(\mathbb{R}^2)$ satisfying $H\phi = 0$.
				\item[(iv)] If $\mu > 18$, then
				$\psi \in S_{5} L^2$ if and only if $\psi = Uv\phi$ with $\phi \in L^2(\mathbb{R}^2)$ satisfying $H\phi = 0$.
			\end{itemize}
		\end{proposition}
		
In particular, Proposition \ref{characterizations-1} \textup{(iii)} and~\textup{(iv)} imply that
\(S_4\neq S_5\) precisely when \(H\phi=0\) admits a nonzero solution
$
\phi\in W_0(\mathbb R^2)\setminus L^2(\mathbb R^2).
$
This is exactly the \(d\)-wave resonance introduced in
Subsection~\ref{main result}. Thus,
\begin{align}\label{Q5-d}
S_4\neq S_5\quad\Longleftrightarrow\quad Q_5\neq0
\quad\Longleftrightarrow\quad
H \text{ has a \(d\)-wave resonance at zero}.
\end{align}

      Recall that \(\mathbf{k}=0\) corresponds to a regular point,
\(\mathbf{k}=1,2,3\) to resonances of the first, second, and third
kind, respectively, and \(\mathbf{k}=4\) to a zero eigenvalue.
Combining Definition~\ref{definition1} with
Proposition~\ref{characterizations-1}, we obtain the following
characterization of each zero-energy spectral type in terms of the
filtration \(\{S_j\}\).

\begin{proposition}\label{characterizations}
Assume that $\left|V(x)\right| \lesssim \left \langle  x\right \rangle^{-\mu}$ for some $\mu>0$
satisfying the condition \eqref{condition}.
Then 
\begin{itemize}
\item[(i)] Zero  is a regular point of $H$  $( \text{i.e.,}\  \mathbf{k}=0 )$  if and only if $S_{2}=0$;
\item[(ii)] Zero is a $\mathbf{k}$-th kind resonance with $1\leq\mathbf{k}\le 4$ if and only if  $S_{\mathbf{k}+1} \neq 0$ and $S_{\mathbf{k}+2}=0$.
\end{itemize}
\end{proposition}

The relationship between the orthogonal projections $Q_j$ and $S_j$, as given in Definition \ref{defQ}, yields the following equivalent characterizations in terms of these $\{Q_j\}$:

\begin{proposition}\label{characterizations Q}
	Assume that $|V(x)| \lesssim \langle x \rangle^{-\mu}$ for some $\mu > 0$ satisfying the condition \eqref{condition}. Then
	\begin{itemize}
		\item[(i)] Zero is a regular point of $H$ $(i.e., \mathbf{k} = 0) $ if and only if
		\[
		Q_0 + Q_1 + Q_2 = \Id.
		\]
		\item[(ii)]  Zero is a $\mathbf{k}$-th kind resonance of $H$ with $1 \le \mathbf{k} \le 4$, if and only if
		\[
		\sum_{j=0}^{\mathbf{k}+1} Q_j \neq \Id \quad \text{and} \quad \sum_{j=0}^{\mathbf{k}+2} Q_j = \Id.
		\]
	\end{itemize}
\end{proposition}

\subsection{Asymptotic expansion of $R_V^\pm(\lambda^4)$ near zero}\label{subsec:Q} 
We turn to provide the asymptotic expansions of $R^\pm_V(\lambda^4)$ near $\lambda=0.$
Let $U(x)=\hbox{sign}\big(V(x)\big)$ and $v(x)=|V(x)|^{1/2}$, then we have $ V=Uv^2$ and
the following symmetric resolvent identity:
 \begin{equation}\label{id-RV}
R^\pm_V(\lambda^4) = R_0^\pm(\lambda^4) -R^\pm_0(\lambda^4)v(M^\pm(\lambda))^{-1}vR^\pm_0(\lambda^4),
\end{equation}
where $M^{\pm}(\lambda)=U+ vR^\pm_0(\lambda^4)v$. Hence, the low-energy expansion of \(R_V^\pm(\lambda^4)\) reduces to
that of \((M^\pm(\lambda))^{-1}\), which was established in
\cite{CCWY}. For simplicity, we do not further distinguish the subcases of the
second-kind resonance considered in \cite{CCWY}. All of these
subcases  are covered by the formulation below, Theorem~\ref{thm:M_inverse}
is consistent with the results of \cite{CCWY}.
\begin{theorem}[{\cite[Theorem 2.7]{CCWY}}]\label{thm:M_inverse}
	Assume that zero is a $\mathbf{k}$-th kind resonance with $0 \leq \mathbf{k} \leq 4$, and $|V(x)| \lesssim \left \langle  x\right \rangle^{-\mu}$ for some $\mu>0$ satisfying the condition \eqref{condition}.
	Then there exists $0 < \lambda_0 \ll 1$ such that for all $0 < \lambda \le (2\lambda_0)^{1/4}$, the operator $\left(M^{\pm}(\lambda)\right)^{-1}$ on $L^2(\mathbb{R}^2)$ admits the asymptotic expansion:
	\begin{equation}\label{eq:M_inverse}  
		\left(M^{\pm}(\lambda)\right)^{-1} = \sum_{0 \leq j, l \leq \mathbf{k}+2} \lambda^{2-k_j-k_l} Q_j \mathcal{M}_{j, l}^{\pm}(\lambda) Q_l ,  
	\end{equation}  
	where the values of the parameter $k_j$ are assigned as follows:  
	\begin{equation}\label{tab:k_alpha}  
		\begin{array}{c|ccccccc}  
			j& 0 & 1 & 2 & 3 & 4 & 5 & 6 \\  
			\hline  
			k_j & 0 & 1 & 1 & 1 & 2 & 3 & 3  
		\end{array}  
	\end{equation}  
	Moreover, the  operators $\mathcal{M}_{j, l}^{\pm}(\lambda)$ possess the following properties:  
	
	\vskip0.3cm  
	\textbf{{\rm (I)}\  If zero is a regular point or a first-kind resonance of $H$} $(\text{i.e.,}\  \mathbf{k}=0,1 )$, then 
	\begin{align}\label{M0-1}
		\Big\| \partial_\lambda^\ell \mathcal{M}_{j,l}^\pm(\lambda) \Big\|_{\mathbb{B}(L^2)} \lesssim
		\begin{cases}
			\lambda^{-\ell}, & \text{for } ( j,l) \neq (3,3), \\
			|\log \lambda| \lambda^{-\ell}, &  \text{for } (j,l) = (3,3),
		\end{cases}
	\end{align}
	for  all $0 \le j, l\le 3$ and $\ell=0,1,2$.	In particular, for  $(j, l) = (0,0)$,
	\[
	\mathcal{M}_{0,0}^\pm(\lambda) = (a_0^\pm)^{-1} \|V\|_{L^1}^{-1}Q_0  + \lambda (-\log \lambda)^{3/2}\Lambda^\pm(\lambda).
	\]	  
	Here and hereafter,  $\Lambda^\pm(\lambda)$ denotes a generic operator in $\mathbb{B}(L^2)$ (possibly varying at each occurrence) satisfying
	$$
	\left\|\partial_\lambda^\ell \Lambda^\pm(\lambda)\right\|_{\mathbb{B}(L^2)} \lesssim \lambda^{-\ell},\quad \ell = 0, 1, 2.$$
	\vskip0.1cm
	\textbf{{\rm (II)}\ If  zero is a  second-kind resonance of $H$} $(\text{i.e.,}\  \mathbf{k}=2 )$,  then
	for all $0 \le j, l\le 3$ and $\ell=0,1,2$, the estimates \eqref{M0-1} remain valid. Moreover,
	for $j \in \{1,2,3\}$,   
	\[  
	\big\|\partial_\lambda^\ell \mathcal{M}_{4,j}^\pm(\lambda)\big\|_{\mathbb{B}(L^2)} + \big\|\partial_\lambda^\ell \mathcal{M}_{j,4}^\pm(\lambda)\big\|_{\mathbb{B}(L^2)} \lesssim \lambda^{1-\ell}|\log \lambda|^{4}, \quad \ell = 0,1,2.  
	\]  
	It remains to specify the case
	$
	(j,l)
	\in
	\{(0,4),(4,0),(4,4)\}
	,$
	which are given as follows:
		\begin{equation*}  
			\begin{aligned}  
				\mathcal{M}_{0,4}^{\pm}(\lambda) &= (\log \lambda)^{-1} Q_0 \Lambda^\pm(\lambda) Q_{4}^0 + Q_0 \Lambda^\pm(\lambda) Q_{4}^1, \\  
				\mathcal{M}_{4,0}^{\pm}(\lambda) &= (\log \lambda)^{-1} Q_{4}^0 \Lambda^\pm(\lambda) Q_0 + Q_{4}^1 \Lambda^\pm(\lambda) Q_0, \\  
				\mathcal{M}_{4,4}^{\pm}(\lambda) &= Q_{4} \Lambda^\pm(\lambda) Q_{4} + (\log\lambda)\bigl(Q_{4} \Lambda^\pm(\lambda) Q_{4}^1+Q_{4}^1 \Lambda^\pm(\lambda) Q_{4}\bigr) + (\log\lambda)^2 Q_{4}^1 \Lambda^\pm(\lambda) Q_{4}^1.
			\end{aligned}  
		\end{equation*}

	\vskip0.1cm
	\textbf{{\rm (III)} If  zero  is a  third-kind resonance or an eigenvalue of $H$} $(\text{i.e.,}\  \mathbf{k}=3,4 )$,  then
	for all $0 \le j, l \le 4$ and $\ell=0,1,2$, all the conclusions established in the second resonance case   remain valid. Moreover,
	for $j \in \{5,6\}$ and $0 \le l \le 4$,
	\[  
	\big\|\partial_\lambda^{\ell} \mathcal{M}_{j,l}^\pm(\lambda)\big\|_{\mathbb{B}\left(L^2\right)}  
	+ \big\|\partial_\lambda^{\ell} \mathcal{M}_{l,j}^\pm(\lambda)\big\|_{\mathbb{B}\left(L^2\right)}  
	\lesssim\lambda^{-\ell},\quad \ell=0,1,2.
	\]  
	For the remaining principal pair $(j,l) \in \{(5,5), (5,6), (6,5), (6,6)\},$  
	\[  
	\mathcal{M}_{j,l}^\pm(\lambda) = 
	\mathcal{A}_{j,l}(\lambda) + (\log \lambda)^{-2}\Gamma_{j,l}^{\pm}(\lambda), 
	\]  
	Here, the principal terms $\mathcal{A}_{j,l}(\lambda)$ are independent of the sign $\pm$ and 
	all $\mathcal{A}_{j,l}(\lambda), \Gamma_{j,l}^{\pm}(\lambda)\in \mathbb{B}(L^2)$  satisfy
	\[
	\big\|\partial_\lambda^\ell \mathcal{A}_{j,l}(\lambda)\big\|_{\mathbb{B}(L^2)} + \big\|\partial_\lambda^\ell \Gamma_{j,l}^{\pm}(\lambda)\big\| _{\mathbb{B}(L^2)} \lesssim \lambda^{-\ell}, \quad \ell = 0,1,2.
	\]
\end{theorem}

To clarify the singular hierarchy in Theorem~\ref{thm:M_inverse},
Table~\ref{tab:Q-leading-order} records the zero-threshold types,
their \(S\)-projection characterizations, and the corresponding
worst leading orders of \((M^\pm(\lambda))^{-1}\) as
\(\lambda\to0^+\).

\vskip-0.3cm
\begin{table}[H]
	\centering
	\caption{Zero-threshold types and leading orders of
		\((M^\pm(\lambda))^{-1}\).}
	\label{tab:Q-leading-order}
	\vskip0.1cm
	\renewcommand{\arraystretch}{1.45}
	\setlength{\tabcolsep}{4pt}
	\footnotesize
	
	\newcommand{\Qcell}[1]{\makecell[c]{\rule[-1.8ex]{0pt}{6.0ex}#1}}
	
	\begin{tabular}{|
			>{\centering\arraybackslash}m{0.20\textwidth}|
			>{\centering\arraybackslash}m{0.34\textwidth}|
			>{\centering\arraybackslash}m{0.36\textwidth}|}
		\hline
		\textbf{Zero-threshold type}
		&
		\textbf{\(S\)-characterization}
		&
		\textbf{Leading order of \((M^\pm(\lambda))^{-1}\)}
		\\[1mm]
		\hline
		
		\Qcell{Regular point\\ \((\mathbf k=0)\)}
		&
		\Qcell{\(S_2=0\)}
		&
		\Qcell{\(O(1)\)}
		\\
		\hline
		
		\Qcell{1st-kind resonance\\ \((\mathbf k=1)\)}
		&
		\Qcell{
			\(\displaystyle S_2\neq0\),
			\ \ \(\displaystyle S_3=0\)}
		&
		\Qcell{\(O(|\log\lambda|)\)}
		\\
		\hline
		
		\Qcell{2nd-kind \\ \((\mathbf k=2)\)}
		&
		\Qcell{
			\(\displaystyle S_3\neq0\),
			\ \ \(\displaystyle S_4=0\)}
		&
		\Qcell{\(O\!\left(\lambda^{-2}|\log\lambda|^2\right)\)}
		\\
		\hline
		
		\Qcell{3rd-kind \\ \((\mathbf k=3)\)}
		&
		\Qcell{
			\(\displaystyle S_4\neq0\),
			\ \ \(\displaystyle S_5=0\)}
		&
		\Qcell{\(O\!\left(\lambda^{-4}|\log\lambda|^{-1}\right)\)}
		\\
		\hline
		
		\Qcell{Zero eigenvalue\\ \((\mathbf k=4)\)}
		&
		\Qcell{\(\displaystyle S_5\neq0\)}
		&
		\Qcell{\(O(\lambda^{-4})\)}
		\\
		\hline
	\end{tabular} 
\end{table}

The cancellation properties of the projections \(Q_j\) listed
in Remark \ref{cancellationQ}, eliminate the corresponding 
moment terms in the free  expansion of \(R_0^\pm(\lambda^4)\).
Consequently, the projected kernels \(Q_j vR_0^\pm(\lambda^4)\) exhibit improved \(\lambda\)-decay than the unprojected free resolvent $R_0^\pm(\lambda^4)$. This improvement partially
compensates for the singularity of \((M^\pm(\lambda))^{-1}\). 
The following Lemma \ref{lemma_projection}  captures this compensation mechanism through structural decompositions and \(L^2\)-bounds for \(Q_j vR_0^\pm(\lambda^4)\).

\subsection{Two fundamental lemmas}\label{subsec:osci_integral}
This subsection presents two key lemmas: the first analyzes the kernel of $Q_jv R_0^\pm(\lambda^4)$, and the second establishes oscillatory integral estimates.

We begin with the definition of the operator $Q_j v R_0^\pm(\lambda^4)$:
\[
\bigl(Q_j v R_0^\pm(\lambda^4) f\bigr)(x) = Q_j\!\Bigl( v(\cdot) \int_{\mathbb{R}^2} R_0^\pm(\lambda^4)(\cdot, y) f(y) \, dy \Bigr)(x).
\]
Its integral kernel is given explicitly by
\(
\bigl(Q_jv R_0^\pm(\lambda^4)\bigr)(x, y) = Q_j\!\left( v\, R_0^\pm(\lambda^4)(\cdot, y) \right)(x).
\)
The structural decompositions and $L^2$-estimates of this integral kernel are summarized in the following Lemma \ref{lemma_projection}.

\begin{lemma}[{\cite[Lemma 2.8]{CCWY}}]\label{lemma_projection}
	Let $Q_j$ $(j\in\mathbb{N}_0, 0\le j\le6)$ and $Q_4^0, Q_4^1$ be defined by \eqref{defQ}. Assume $\left|V(x)\right| \lesssim \left \langle  x\right \rangle^{-\mu}$ for some $\mu>0$ satisfying the condition \eqref{condition}. Then, for $0<\lambda\ll1$, the integral kernels admit the following decompositions:
	\begin{enumerate}
		\item[(i)] For $0\le j\le6$,
		\[
		\bigl(Q_j v R_0^\pm(\lambda^4)\bigr)(x,y)=\lambda^{-1-\delta_{j0}}\,\Omega_{j,\pm}(\lambda,x,y),
		\]
		where $\delta_{j0}=1$ if $j=0$ and $0$ otherwise.
		
		\item[(ii)] For $2\le j\le6$, 
		\begin{align*}
		\bigl(Q_jv R_0^\pm(\lambda^4)\bigr)(x,y)&=\mathcal{J}_{j, \pm} (\lambda,x,y).
		\end{align*}
		\item[(iii)] For $2\le j\le6$,
		setting $m_j=0$ for $j=2,3,4$, $m_5=1$, and $m_6=2$:
		$$
			\bigl(Q_j v R_0^\pm(\lambda^4)\bigr)(x,y) =(b_0Q_j v G_2^0)(x,y)+b_0(\log\lambda)\,Q_j(v|\cdot|^2)(x)+ \lambda^{m_j} (\mathcal{T}_{j, \pm} + \mathcal{T}_j)(\lambda, x, y),
        $$
		where the term $b_0 (\log\lambda) \, Q_j(v|\cdot|^2)(x)$ vanishes
        for $j=5,6$.
		In particular,
		\begin{align*}
\bigl(Q_4^0 v R_0^\pm(\lambda^4)\bigr)(x,y) &=(b_0Q_4^0 v G_2^0)(x,y)+b_0(\log\lambda)\,Q_4^0(v|\cdot|^2)(x)+(\mathcal{T}_{4,\pm}^0+\mathcal{T}_4^0)(\lambda,x,y),\\
\bigl(Q_4^1 v R_0^\pm(\lambda^4)\bigr)(x,y) &=(b_0Q_4^1 v G_2^0)(x,y)+(\mathcal{T}_{4,\pm}^1+\mathcal{T}_4^1)(\lambda,x,y).
		\end{align*}
	\end{enumerate}
	Furthermore, for $\ell=0,1,2$, these integral kernels satisfy the following bounds:
	\begin{align}\label{J}
		&\|\partial_\lambda^\ell\big(e^{\mp i\lambda|y|}\mathcal{J}_{\pm}(\lambda,\cdot,y)\big)\|_{L^2} \lesssim\lambda^{-\ell}|\log\lambda|\langle\lambda y\rangle^{-1/2},\  \text{for } \mathcal{J}_{\pm} \in \{\mathcal{J}_{j,\pm}|_{2\le j\le6}\},\\
&\bigl\|(b_0Q v G_2^0)(\cdot,y)\bigr\|_{L^2} \lesssim
		\begin{cases}
		\log(2+|y|), & \text{for } Q \in \{Q_j|_{2 \le j \le 4}, Q_4^0\}, \\[2pt]
		1, & \text{for } Q \in \{Q_4^1, Q_5, Q_6\},
		\end{cases} \label{lemma_projection_G} \\
&\|\partial_\lambda^\ell\mathcal{T}(\lambda,\cdot,y)\|_{L^2} \lesssim
		\begin{cases}
		\lambda^{-\ell}\log(2+|y|), & \text{for } \mathcal{T} \in \{\mathcal{T}_j|_{2\le j\le4}, \, \mathcal{T}_4^0\}, \\[2pt]
		\lambda^{-\ell}, & \text{for } \mathcal{T} \in \{\mathcal{T}_4^{1},  \mathcal{T}_5, \mathcal{T}_6\},
		\end{cases} \label{lemma_projection_T} \\   &\|\partial_\lambda^\ell\big(e^{\mp i\lambda|y|}K_{\pm}(\lambda,\cdot,y)\big)\|_{L^2} \lesssim\lambda^{-\ell}\langle\lambda y\rangle^{-1/2}, \label{lemma_projection_1}
	\end{align}
	where $K_{\pm}$ generically denotes any  term among $\Omega_{j,\pm} \, (0\le j\le6)$, $\mathcal{T}_{j,\pm} \, (2\le j\le6)$, $\mathcal{T}_{4,\pm}^0$, or $\mathcal{T}_{4,\pm}^{1}$.
\end{lemma}

Next, we state Lemma \ref{quartic-oscillatory} which  serves as a fundamental tool for the oscillatory integrals encountered in the subsequent decay analysis. Its proof is  given in Section \ref{section8}.

\begin{lemma}\label{quartic-oscillatory}
	Let \(h(x,y)\) and \(r=r(x,y)\) be real-valued functions on
	\(\mathbb R^2\times\mathbb R^2\), with \(h(x,y)>0\). Define
	\[
	\mathcal I^\pm(t,x,y)
	:=
	\int_0^\infty
	e^{-it\lambda^4}
	e^{\pm i\lambda r}
	\lambda
	\mathcal E^\pm(\lambda,x,y)\,d\lambda.
	\]
	\begin{enumerate}
		\item[(i)]
		Let \(-2<\gamma\leq2\). Suppose that for
		\(\ell=0,1,2\),
		\begin{equation}\label{quartic-cond-1}
			\left|
			\partial_\lambda^\ell
			\mathcal E^\pm(\lambda,x,y)
			\right|
			\lesssim
			h(x,y)
			\langle\lambda r\rangle^{-1/2}
			\lambda^{\gamma-\ell}.
		\end{equation}
		Then, for $t\neq0,$
		\begin{equation*}
			\left|
			\mathcal I^\pm(t,x,y)
			\right|
			\lesssim
			h(x,y)|t|^{-\frac{2+\gamma}{4}}.
		\end{equation*}
			Moreover, 	if \(r=0\), the estimate above  remains valid in the
		larger range
		$
		-2<\gamma<6$.
	
		\item[(ii)]
		Suppose that  for \(\ell=0,1,2\),
		\begin{equation}\label{quartic-cond-2}
			\left|
			\partial_\lambda^\ell
			\mathcal E^\pm(\lambda,x,y)
			\right|
			\lesssim
			h(x,y)
			\langle\lambda r\rangle^{-1/2}
			|\log\lambda|^{-\nu}
			\lambda^{\sigma-\ell}
			\widetilde\chi_1(\lambda/2).
		\end{equation}
		Then the following estimates hold.
		
		\begin{itemize}
			\item If
			$
			-2<\sigma<2$ with
			$ \nu\in\mathbb R$, 
			or
			$
			\sigma=2$ with
			$ \nu\leq0,$
			then
			\begin{equation*}
				\left|
				\mathcal I^\pm(t,x,y)
				\right|
				\lesssim
				\frac{h(x,y)}
				{\langle t\rangle^{\frac{2+\sigma}{4}}
					\bigl(\log(2+|t|)\bigr)^\nu}.
			\end{equation*}
			Moreover, 	if \(r=0\), the estimate above  remains valid in the
			larger range
			$
			-2<\sigma<6$ and 
			$\nu\in\mathbb R.
			$
			
			\item If
			$
			\sigma=-2$ and
			$\nu>1,$
			then
			\begin{equation*}
				\left|
				\mathcal I^\pm(t,x,y)
				\right|
				\lesssim
				\frac{h(x,y)}
				{\bigl(\log(2+|t|)\bigr)^{\nu-1}}.
			\end{equation*}
		\end{itemize}
	\end{enumerate}
\end{lemma}

\section{Free dispersive estimates}\label{sec:free_case}

In this section, we establish sharp dispersive estimates for the
free propagator
\[
(-\Delta)^{\alpha/2}e^{-it\Delta^2},
\qquad -2<\alpha\leq2,
\]
on \(\mathbb R^2\), both in the unweighted setting $L^1\to L^\infty$
and with
logarithmic spatial weights
\(\omega(x)=\log(2+|x|)\).
We begin with the following Stone formula: 
\begin{equation}\label{free-stone}
\begin{aligned}
\bigl[(-\Delta)^{\alpha/2}e^{-it\Delta^2}\bigr](x,y)
&=
\frac{2}{\pi i}
\int_0^\infty
e^{-it\lambda^4}
\lambda^{3+\alpha}
\bigl[
R_0^+(\lambda^4)-R_0^-(\lambda^4)
\bigr](x,y)\,d\lambda =:K(\alpha;t,x,y).
\end{aligned}
\end{equation}

Based on this representation, we obtain the following sharp decay estimates.

\begin{proposition}\label{free_case}
	Let \(-2<\alpha\leq2\) and  $s>0$. Then, for every \(t\neq0\),
	\begin{equation}\label{free-unweighted-sharp}
		\bigl\|
		(-\Delta)^{\alpha/2}e^{-it\Delta^2}
		\bigr\|_{L^1\to L^\infty}
		\sim
		|t|^{-\frac{2+\alpha}{4}},
\quad
		\bigl\|
\omega^{-s}
(-\Delta)^{\alpha/2}e^{-it\Delta^2}\omega^{-s}
\bigr\|_{L^1\to L^\infty}
\sim
|t|^{-\frac{2+\alpha}{4}}.
	\end{equation}
		Moreover, the interval
	$
	-2<\alpha\leq2
	$
	is optimal for the uniform unweighted \(L^1\to L^\infty\) estimate, namely
	\[
	(-\Delta)^{\alpha/2}e^{-it\Delta^2}
	\notin \mathcal B(L^1(\mathbb R^2),L^\infty(\mathbb R^2)),
	\]
	whenever
	$
	\alpha\leq-2$ 
or $
	\alpha>2.$
\end{proposition}

\begin{proof}
	We divide this proof into three steps.
	
		\medskip
	\noindent
	\textit{Step 1. The sharp unweighted estimate.}	By \eqref{def:F_pm} and \eqref{free-widetilde_R}, the free resolvent
	kernel can be written as
	\[
	R_0^\pm(\lambda^4)(x,y)
	=
	\lambda^{-2}e^{\pm i\lambda |x-y|}
	\widetilde R^\pm(\lambda |x-y|).
	\]
	Combining this with \eqref{free-stone}, we have
	\begin{equation}\label{free-I-pm}
	K(\alpha;t,x,y)
		=
		\mathcal I^+(\alpha;t,x,y)
		-
		\mathcal I^-(\alpha; t, x,y),
	\end{equation}
	where
	\begin{equation}\label{def:free-I-pm}
		\mathcal I^\pm(\alpha; t,x,y)
		:=
		\frac{2}{\pi i}
		\int_0^\infty
		e^{-it\lambda^4}
		e^{\pm i\lambda |x-y|}
		\lambda^{1+\alpha}
		\widetilde R^\pm(\lambda |x-y|)\,d\lambda .
	\end{equation}
	By \eqref{free-widetilde_R}, for \(\ell=0,1,2\),
	\begin{equation}\label{free-amplitude-symbol}
		\left|
		\partial_\lambda^\ell
		\Big(
		\lambda^\alpha
		\widetilde R^\pm(\lambda |x-y|)
		\Big)
		\right|
		\lesssim
		\langle\lambda |x-y|\rangle^{-1/2}
		\lambda^{\alpha-\ell}.
	\end{equation}
	Thus Lemma~\ref{quartic-oscillatory} {(i)}, with
	$ r=|x-y|$ and $
	\gamma=\alpha$
	yields
	\[
	|\mathcal I^\pm(\alpha; t,x,y)|
	\lesssim
	|t|^{-\frac{2+\alpha}{4}},
	\]
	uniformly in \(x,y\in\mathbb R^2\). It follows from
	\eqref{free-I-pm} that
	\[
	\bigl\|
	(-\Delta)^{\alpha/2}e^{-it\Delta^2}
	\bigr\|_{L^1\to L^\infty}=\sup_{x,y\in\mathbb{R}^2}|	K(\alpha; t,x,y)|
	\lesssim
	|t|^{-\frac{2+\alpha}{4}}.
	\]
    
	We next prove the matching lower bound. Since
	$
	\widetilde R^\pm(0)=a_0^\pm=\pm{i}/{8},
	$
	we obtain for every \(x\in\mathbb R^2\),
	\begin{align}
		K(\alpha; t,x,x)
		&=
		\frac{1}{2\pi}
		\int_0^\infty
		e^{-it\lambda^4}
		\lambda^{1+\alpha}\,d\lambda
	=
		\frac{1}{8\pi}
		\Gamma\left(\frac{2+\alpha}{4}\right)
		e^{-i\frac{(2+\alpha)\pi}{8}\operatorname{sgn}(t)}
		|t|^{-\frac{2+\alpha}{4}}.
		\label{free-diagonal}
	\end{align}
	Here we used the following standard oscillatory Gamma integral  understood in the
	Abel-regularized sense:
	\begin{align}\label{Gamma-inte}
	\int_0^\infty
	e^{-it\lambda^4}\lambda^{1+\alpha}\,d\lambda
	=
	\frac14
	\Gamma\left(\frac{2+\alpha}{4}\right)
	e^{-i\frac{(2+\alpha)\pi}{8}\operatorname{sgn}(t)}
	|t|^{-\frac{2+\alpha}{4}}.
	\end{align}
	Therefore,
	\[
	\bigl\|
	(-\Delta)^{\alpha/2}e^{-it\Delta^2}
	\bigr\|_{L^1\to L^\infty}
	\geq
	|K(\alpha; t, x,x)|
	\gtrsim
	|t|^{-\frac{2+\alpha}{4}}.
	\]
Hence, we obtain the sharp estimate $
|t|^{-\frac{2+\alpha}{4}}$ in the unweighted space $L^1\to L^\infty$,  for all $-2<\alpha\leq2.$

		\medskip
	\noindent
\textit{Step 2. The weighted estimate.}
Since the weight function
\(\omega^{-s}\) is bounded on \(\mathbb R^2\) for every \(s>0\).
Therefore, by the unweighted estimate,
\begin{align*}
\bigl\|
\omega^{-s}
(-\Delta)^{\alpha/2}e^{-it\Delta^2}
\omega^{-s}
\bigr\|_{L^1\to L^\infty}
&\lesssim
\bigl\|
(-\Delta)^{\alpha/2}e^{-it\Delta^2}
\bigr\|_{L^1\to L^\infty}  \lesssim
|t|^{-\frac{2+\alpha}{4}}.
\end{align*}
	On the other hand, evaluating at \(x=y=0\) and using
	\eqref{free-diagonal}, we obtain
	\[
	\bigl\|
\omega^{-s}
(-\Delta)^{\alpha/2}e^{-it\Delta^2}
\omega^{-s}
\bigr\|_{L^1\to L^\infty}
	\geq
	\frac{|K(\alpha; t,0,0)|}{(\log2)^{2s}}
	\gtrsim
	|t|^{-\frac{2+\alpha}{4}}.
	\]
	This proves the sharp weighted estimate $
	|t|^{-\frac{2+\alpha}{4}}$.

\medskip
\noindent
\textit{Step 3. Optimality of the range \(-2<\alpha\leq2\).}
We show that, for every fixed \(t\neq0\), the free evolution fails to
define a bounded operator from \(L^1(\mathbb R^2)\) to
\(L^\infty(\mathbb R^2)\) whenever \(\alpha\leq-2\) or
\(\alpha>2\).

\medskip
\noindent
\textit{Failure for \(\alpha\leq-2\).}
Let \(\chi\in C_c^\infty(\mathbb R^2)\) be radial and nonincreasing,
such that
\[
0\leq\chi\leq1,\qquad
\chi(\xi)=1\quad\text{for }|\xi|\leq1,
\qquad
\operatorname{supp}\chi\subset\{|\xi|\leq2\}.
\]
Fix \(t\neq0\). Choose \(0<\delta\ll1\) so that
\[
\Re\bigl(e^{-it|\xi|^4}\bigr)\geq\frac12
\qquad\text{for }|\xi|\leq2\delta.
\]
For \(0<\varepsilon<1/4\), define
\[
\widehat f_\varepsilon(\xi)
=
\chi(\xi/\delta)
-
\chi(\xi/(\varepsilon\delta)).
\]
Since \(\chi\) is radial and nonincreasing,
\(\widehat f_\varepsilon\geq0\). Moreover,
\[
\widehat f_\varepsilon(\xi)=1
\qquad\text{whenever}\qquad
2\varepsilon\delta\leq|\xi|\leq\delta.
\]
Notice also that \(\widehat f_\varepsilon\) vanishes in a
neighborhood of \(\xi=0\), so the multiplier
$
|\xi|^\alpha e^{-it|\xi|^4}
$
is well defined on \(f_\varepsilon\), even when \(\alpha\leq-2\).
By scaling,
\begin{equation}\label{f-unifor}
\sup_{0<\varepsilon<1/4}
\|f_\varepsilon\|_{L^1(\mathbb R^2)}
\lesssim1.
\end{equation}
Since \(\widehat f_\varepsilon\geq0\) and
	\(\widehat f_\varepsilon(\xi)=1\) whenever
	\(2\varepsilon\delta\leq|\xi|\leq\delta\), we obtain
	\[
	\begin{aligned}
		\Re\Big[
		(-\Delta)^{\alpha/2}e^{-it\Delta^2}
		f_\varepsilon
		\Big](0)
		&=
		\frac{1}{(2\pi)^2}
		\int_{\mathbb R^2}
		|\xi|^\alpha
		\Re\bigl(e^{-it|\xi|^4}\bigr)
		\widehat f_\varepsilon(\xi)\,d\xi\\
		&\gtrsim
		\int_{2\varepsilon\delta}^{\delta}
		r^{1+\alpha}\,dr	\longrightarrow\infty, \qquad\text{as }\varepsilon\to0,
	\end{aligned}
	\]
	provided $\alpha\leq-2.$
Together with \eqref{f-unifor}, this shows that the Fourier multiplier
with symbol
$
|\xi|^\alpha e^{-it|\xi|^4},
$
initially defined on functions whose Fourier transforms are supported
away from the origin, admits no bounded extension from
\(L^1(\mathbb R^2)\) to \(L^\infty(\mathbb R^2)\). Hence
\begin{equation}\label{free-failure-low}
(-\Delta)^{\alpha/2}e^{-it\Delta^2}
\notin
\mathcal B\bigl(L^1(\mathbb R^2),L^\infty(\mathbb R^2)\bigr),
\qquad
\alpha\leq-2.
\end{equation}
    
	\medskip
	Next,  we prove the \textit{failure for \(\alpha>2\).}
	Let \(e_1=(1,0)\), and choose
	\(\psi\in C_c^\infty(\mathbb R^2)\) supported in a sufficiently small
	neighborhood of \(e_1=(1,0)\), with \(\psi(e_1)\neq0\). For \(\rho\gg1\),
	define
	$
	\widehat f_\rho(\xi)=\psi(\xi/\rho).
	$
	By scaling,
	\[
	\|f_\rho\|_{L^1}
	=
	C\|\check\psi\|_{L^1},
	\]
	uniformly in \(\rho\). Set
	$
	x_\rho=4t\rho^3e_1.
	$
	Then
	\[
	\begin{aligned}
		\Big[
		(-\Delta)^{\alpha/2}e^{-it\Delta^2}
		f_\rho
		\Big](x_\rho)
		&=
		\frac{\rho^{\alpha+2}}{(2\pi)^2}
		\int_{\mathbb R^2}
		e^{it\rho^4\Phi(\eta)}
		|\eta|^\alpha\psi(\eta)\,d\eta,
	\end{aligned}
	\]
	where
	$
	\Phi(\eta)=4\eta_1-|\eta|^4.
	$
	The phase \(\Phi\) has a nondegenerate critical point at
	\(\eta=e_1\), since
	\[
	\nabla\Phi(e_1)=0,
	\qquad 
\det	\nabla^2\Phi(e_1)
	=\det
	\begin{pmatrix}
		-12&0\\
		0&-4
	\end{pmatrix}=48.
	\]
By the standard stationary phase theorem
(see, e.g., \cite[Section~7.7]{HormanderI}), as \(\rho\to\infty\),
\[
\begin{aligned}
	\int_{\mathbb R^2}
	e^{it\rho^4\Phi(\eta)}
	|\eta|^\alpha\psi(\eta)\,d\eta
	&=
	\frac{2\pi}{|t|\rho^4}
	\frac{\psi(e_1)}{\sqrt{48}}\,
	e^{-i\frac{\pi}{2}\operatorname{sgn}(t)}
	e^{i3t\rho^4}
	+
	O
	\bigl(|t|^{-2}\rho^{-8}\bigr).
\end{aligned}
\]
Consequently,
\begin{equation}\label{free-stationary-lower}
	\begin{aligned}
		\Bigl[
		(-\Delta)^{\alpha/2}e^{-it\Delta^2}
		f_\rho
		\Bigr](x_\rho)
		&=
		c_0(t,\psi)\,
		\rho^{\alpha-2}
		e^{i3t\rho^4}
+
		O
		\bigl(|t|^{-2}\rho^{\alpha-6}\bigr),
	\end{aligned}
\end{equation}
	where
\[
c_0(t,\psi)
=
\frac{\psi(e_1)}
{2\pi\sqrt{48}}\,
|t|^{-1}
e^{-i\frac{\pi}{2}\operatorname{sgn}(t)}
\neq0.
\]
Since the error term in \eqref{free-stationary-lower} is smaller than
the leading term by a factor
\(O(|t|^{-1}\rho^{-4})\), for every fixed
\(t\neq0\) and sufficiently large \(\rho\),
\[
\left|
\Bigl[
(-\Delta)^{\alpha/2}e^{-it\Delta^2}
f_\rho
\Bigr](x_\rho)
\right|
\gtrsim_t
\rho^{\alpha-2}.
\]
If \(\alpha>2\), the right-hand side tends to infinity as
\(\rho\to\infty\), whereas \(\|f_\rho\|_{L^1}\) remains independent
of \(\rho\). Hence
\begin{equation}\label{free-failure-high}
	(-\Delta)^{\alpha/2}e^{-it\Delta^2}
	\notin
	\mathcal B(L^1,L^\infty),
	\qquad
	\alpha>2.
\end{equation}

Combining \eqref{free-failure-low} and
\eqref{free-failure-high}, we conclude that
$
-2<\alpha\leq2
$
is exactly the optimal range for the uniform
\(L^1(\mathbb R^2)\to L^\infty(\mathbb R^2)\) estimate of the free
evolution.

\end{proof}

We next separate the free propagator into its low- and high-energy
components by the identity $1=\chi_1+\chi_2$, where  $
\chi_1\in C_c^\infty(\mathbb{R})$
defined in Subsection \ref{Notations}
and
\[
\chi_2:=1-\chi_1,
\qquad
\widetilde\chi_j(\lambda):=\chi_j(\lambda^4),
\qquad j=1,2.
\]

\begin{proposition}\label{free-energy-decomposition}
	Let \(-2<\alpha\leq2\) and $s>0.$
	
	\begin{enumerate}
		\item[\textup{(i)}]
		For \(j=1,2\) and \(t\neq0\),
		\begin{equation}\label{free-cutoff-unweighted}
			\bigl\|
			(-\Delta)^{\alpha/2}e^{-it\Delta^2}
			\chi_j(\Delta^2)
			\bigr\|_{L^1\to L^\infty}
			\lesssim
			|t|^{-\frac{2+\alpha}{4}}.
		\end{equation}
		
		\item[\textup{(ii)}]
		For \(|t|\geq2\), the low-energy kernel satisfies
		\begin{align}
			&\bigl[
			(-\Delta)^{\alpha/2}e^{-it\Delta^2}
			\chi_1(\Delta^2)
			\bigr](x,y)
		=
			\frac{1}{8\pi}
			\Gamma\left(\frac{2+\alpha}{4}\right)
			e^{-i\frac{(2+\alpha)\pi}{8}\operatorname{sgn}(t)}
			|t|^{-\frac{2+\alpha}{4}}
	+
			O\left(
			\frac{\bigl(\omega(x)\omega(y)\bigr)^s}
			{|t|^{\frac{2+\alpha}{4}}
				(\log|t|)^s}
			\right),
			\label{free-low-asymptotic}
		\end{align}
where $\Gamma$
denotes the Gamma function.
 	Moreover, the high-energy component satisfies
		\begin{equation}\label{free-high-weighted}
			\bigl\|
\omega^{-s}
(-\Delta)^{\alpha/2}e^{-it\Delta^2}\chi_2(\Delta^2)\omega^{-s}
\bigr\|_{L^1\to L^\infty}
			\lesssim
			{|t|^{-\frac{2+\alpha}{4}}}
			{(\log|t|)^{-s}}.
		\end{equation}
	\end{enumerate}
\end{proposition}

\begin{proof}
	We first prove \textup{(i)}. Inserting the smooth cutoff
	\(\widetilde\chi_j(\lambda)=\chi_j(\lambda^4)\) into
	\eqref{def:free-I-pm}, the resulting amplitude factor
	$
	\lambda^\alpha\widetilde\chi_j(\lambda)
	\widetilde R^\pm(\lambda|x-y|)
	$
	satisfies the same bounds as in \eqref{free-amplitude-symbol}.
	Hence, Lemma~\ref{quartic-oscillatory} \textup{(i)} gives
	\[
	\bigl\|
	(-\Delta)^{\alpha/2}e^{-it\Delta^2}\chi_j(\Delta^2)
	\bigr\|_{L^1\to L^\infty}
	\lesssim
	|t|^{-\frac{2+\alpha}{4}},
	\qquad j=1,2.
	\]
	
	We turn to \textup{(ii)}. By  \eqref{reso_small} and \eqref{reso-big}, we have the following expansion for any $\lambda> 0$:
	\begin{align}
		\label{R+-R-}
		\bigl[R_0^+(\lambda^4)-R_0^-(\lambda^4)\bigr](x,y)=\frac{i}{4}\lambda^{-2}+E(\lambda,x,y),
	\end{align}
	where $E(\lambda,x,y)$ satisfies
	\begin{equation}\label{eq:bound_E}
		|\partial_\lambda^\ell E(\lambda,x,y)|\lesssim\lambda^{-\ell}\langle x\rangle^2\langle y\rangle^2,\quad\ell=0,1,2.
	\end{equation}
		Denote by \(K_{j}(\alpha; t, x,y)\) the kernel of
	$
	(-\Delta)^{\alpha/2}e^{-it\Delta^2}\chi_j(\Delta^2).
	$
	Substituting \eqref{R+-R-} into Stone's formula gives
	\begin{equation}\label{free-leading-remainder-split}
	K_{j}(\alpha; t, x,y)
		=
		L_{j}(\alpha; t)
		+
		A_{j}(\alpha; t, x,y),\quad j=1,2,
	\end{equation}
	where
	\begin{align*}
		L_{j}(\alpha; t)
		&:=
		\frac{1}{2\pi}
		\int_0^\infty
		e^{-it\lambda^4}
		\lambda^{1+\alpha}
		\widetilde\chi_j(\lambda)\,d\lambda,\\
		 A_{j}(\alpha; t, x,y)
	&:=
		\frac{2}{\pi i}
		\int_0^\infty
		e^{-it\lambda^4}
		\lambda^{3+\alpha}
		\widetilde\chi_j(\lambda)
		E(\lambda;x,y)\,d\lambda.
	\end{align*}
	By \eqref{eq:bound_E}, one has for $j=1,2,$
	\[
	\left|
	\partial_\lambda^\ell
	\left[
	\lambda^{2+\alpha}
	\widetilde\chi_j(\lambda)
	E(\lambda;x,y)
	\right]
	\right|
	\lesssim
	\lambda^{2+\alpha-\ell}
	\langle x\rangle^2\langle y\rangle^2,
	\qquad
	\ell=0,1,2.
	\]
	 We  apply the extended
	range in Lemma~\ref{quartic-oscillatory}\textup{(i)} with \(r=0\) and
	$
	\gamma=2+\alpha\in(0,4].
	$
Then
	\begin{equation}\label{free-E-polynomial-weight}
		|A_{j}( \alpha; t, x,y)|
		\lesssim
		|t|^{-\frac{4+\alpha}{4}}
		\langle x\rangle^2\langle y\rangle^2.
	\end{equation}
	
	On the other hand, \eqref{free-cutoff-unweighted} and the elementary
	bound
	$
	|L_{j}(\alpha; t)|
	\lesssim
	|t|^{-\frac{2+\alpha}{4}},
	$
	which follows again from
	Lemma~\ref{quartic-oscillatory}\textup{(i)} with \(r=0\) and
	\(\gamma=\alpha\), imply
	\begin{equation}\label{free-E-unweighted}
		|A_{j}( \alpha; t, x,y)|
		\lesssim
		|t|^{-\frac{2+\alpha}{4}}.
	\end{equation}
	Combining \eqref{free-E-polynomial-weight}, \eqref{free-E-unweighted} with  the inequality $\min(1,a/b) \lesssim (\log a/\log b)^s$ (valid for $a,b\ge 2$ and fixed $s>0$),
we obtain  for  $|t|\geq2$ and every fixed   $s>0$,
	\begin{align}
		|A_{j}( \alpha; t, x,y)|
		&\lesssim
		|t|^{-\frac{2+\alpha}{4}}
		\min\left\{
		1,
		\frac{\langle x\rangle^2\langle y\rangle^2}{|t|^{1/2}}
		\right\}	\lesssim\frac{\bigl(\omega(x)\omega(y)\bigr)^s}{|t|^{\frac{2+\alpha}{4}}(\log|t|)^s},\quad j=1,2.
		\label{free-E-min}
	\end{align}

	It remains to analyze the scalar oscillatory integrals
	\(L_{j}(\alpha; t)\).
	 First suppose that \(-2<\alpha<2\). Since
	\(\widetilde\chi_2\) vanishes near zero and
	$
	\lambda^{\alpha-2}\to0$
as $\lambda\to\infty,$
	repeated integration by parts, 
	gives
\begin{align}\label{L-alpha}
		L_{2}(\alpha;t)
		=O
		(|t|^{-N}),\quad -2<\alpha<2,
\end{align}
	for any \(N>0\).
    Therefore combining \eqref{Gamma-inte} and \eqref{L-alpha} for \(-2<\alpha<2\), we have
	\begin{align}
		L_{1}(\alpha; t)
		&=
		\frac{1}{2\pi}
		\int_0^\infty
		e^{-it\lambda^4}\lambda^{1+\alpha}\,d\lambda
		-
		L_{2}(\alpha; t)
		=
		\frac{1}{8\pi}
		\Gamma\left(\frac{2+\alpha}{4}\right)
		e^{-i\frac{(2+\alpha)\pi}{8}\operatorname{sgn}(t)}
		|t|^{-\frac{2+\alpha}{4}}
		+
		O(|t|^{-N}).
		\label{low-scalar-alpha-less-two}
	\end{align}

	For the endpoint \(\alpha=2\), integration by parts in the compactly
	supported low-energy integral yields
	\begin{align}
		L_{1}(2;t)
		&=
		\frac{1}{2\pi}
		\int_0^\infty
		e^{-it\lambda^4}
		\lambda^3\widetilde\chi_1(\lambda)\,d\lambda
	=
		\frac{1}{8\pi i t}
		+
		\frac{1}{8\pi i t}
		\int_0^\infty
		e^{-it\lambda^4}
		\widetilde\chi_1'(\lambda)\,d\lambda
		=
		\frac{1}{8\pi i t}
		+
		O(|t|^{-N}).
		\label{low-scalar-alpha-two}
	\end{align}
By the  standard oscillatory Gamma integral, understood in the
Abel-regularized sense,
	\[
	\frac{1}{2\pi}
	\int_0^\infty
	e^{-it\lambda^4}\lambda^3\,d\lambda
	=
	\frac{1}{8\pi i t},
	\]
	we also have
	\begin{align}\label{L2}
	L_{2}(2; t)=	\frac{1}{2\pi}
	\int_0^\infty
	e^{-it\lambda^4}\lambda^3\,d\lambda-	L_{1}(2;t)=O(|t|^{-N}).
	\end{align}
	Combining \eqref{free-leading-remainder-split}, \eqref{free-E-min}, \eqref{low-scalar-alpha-less-two}, and \eqref{low-scalar-alpha-two} yields the low-energy expansion \eqref{free-low-asymptotic}. Similarly, \eqref{free-leading-remainder-split}, \eqref{free-E-min}, \eqref{L-alpha}, and \eqref{L2} prove the desired estimate \eqref{free-high-weighted} for the high-energy part.
\end{proof}

\section{High-energy decay estimates}\label{sec:high-energy}

In this section, we establish the high-energy dispersive estimates
for
\[
H^{\alpha/4}e^{-itH}P_{\mathrm{ac}}(H)\chi_2(H),
\qquad -2<\alpha\leq2,
\]
on \(\mathbb R^2\), both in the unweighted space
\(L^1\to L^\infty\) and 
and with
logarithmic spatial weights
\(\omega(x)=\log(2+|x|)\).

We begin
by Stone's formula,
\begin{equation}\label{stone-high-schrodinger}
	\begin{aligned}
		&H^{\alpha/4}e^{-itH}
		P_{\mathrm{ac}}(H)\chi_2(H)
=
		\frac{2}{\pi i}
		\int_0^\infty
		e^{-it\lambda^4}
		\lambda^{3+\alpha}
		\widetilde\chi_2(\lambda)
		\bigl[
		R_V^+(\lambda^4)-R_V^-(\lambda^4)
		\bigr]\,d\lambda,
	\end{aligned}
\end{equation}
where $\widetilde\chi_2(\lambda)=\chi_2(\lambda^4)$ is supported in $|\lambda|\geq\lambda_0^{1/4}.$
The main result of this section is stated as  the following Theorem  \ref{main-theorem-high-schrodinger}.

\begin{theorem}\label{main-theorem-high-schrodinger}
	Assume that
	$
	|V(x)|\lesssim\langle x\rangle^{-4-}
	$
	and that \(H\) has no positive embedded eigenvalues. Then, for every
	\(-2<\alpha\leq2\) and \(t\neq0\),
	\begin{equation*}
		\left\|
		H^{\alpha/4}e^{-itH}
		P_{\mathrm{ac}}(H)\chi_2(H)
		\right\|_{L^1\to L^\infty}
		\lesssim
		|t|^{-\frac{2+\alpha}{4}}.
	\end{equation*}
	Moreover, for \(|t|\geq2\) and $s>0,$
	\begin{equation*}
		\left\|\omega^{-s}
		H^{\alpha/4}e^{-itH}
		P_{\mathrm{ac}}(H)\chi_2(H)\omega^{-s}
		\right\|_{L^1\to L^\infty}
		\lesssim
		\frac{1}
		{|t|^{\frac{2+\alpha}{4}}(\log|t|)^s}.
	\end{equation*}
\end{theorem}

We use the resolvent identity
\begin{equation}\label{high-resolvent-identity}
	\begin{aligned}
		R_V^\pm(\lambda^4)
		={}&
		R_0^\pm(\lambda^4)
		-
		R_0^\pm(\lambda^4)
		VR_0^\pm(\lambda^4)+
		R_0^\pm(\lambda^4)
		VR_V^\pm(\lambda^4)
		VR_0^\pm(\lambda^4).
	\end{aligned}
\end{equation}
Substituting \eqref{high-resolvent-identity} into
\eqref{stone-high-schrodinger}, and using
Proposition~\ref{free-energy-decomposition} for the free term, it
remains to estimate the first Born term and the iterated resolvent
term, whose kernels are given by
\begin{align}
	\mathcal L_1^\pm(\alpha;t,x,y)
	:={}&
	\int_0^\infty
	e^{-it\lambda^4}
	\lambda^{3+\alpha}
	\widetilde\chi_2(\lambda)
	\bigl[
	R_0^\pm(\lambda^4)
	VR_0^\pm(\lambda^4)
	\bigr](x,y)\,d\lambda,
	\label{def:T1-high}
	\\
	\mathcal L_2^\pm(\alpha;t,x,y)
	:={}&
	\int_0^\infty
	e^{-it\lambda^4}
	\lambda^{3+\alpha}
	\widetilde\chi_2(\lambda)
	\bigl[
	R_0^\pm(\lambda^4)
	VR_V^\pm(\lambda^4)
	VR_0^\pm(\lambda^4)
	\bigr](x,y)\,d\lambda.
	\label{def:T2-high}
\end{align}

We first consider the first Born correction.

\begin{proposition}\label{high-first-born}
	Assume that
	$
	|V(x)|\lesssim\langle x\rangle^{-4-}.
	$
	Then, for every \(-2<\alpha\leq2\),
	\begin{equation*}
		\sup_{x,y\in\mathbb R^2}
		|\mathcal L_{1}^\pm(\alpha;t,x,y)|
		\lesssim
		|t|^{-\frac{2+\alpha}{4}},
		\quad
		|\mathcal L_1^\pm(\alpha;t,x,y)|
		\lesssim
		|t|^{-2}
		\langle x\rangle^2\langle y\rangle^2.
	\end{equation*}
\end{proposition}

\begin{proof}
	Recall that (see \eqref{def:F_pm})
	\[
	R_0^\pm(\lambda^4)(x,y)
	=
	\lambda^{-2}
	e^{\pm i\lambda|x-y|}
	\widetilde R^\pm(\lambda|x-y|).
	\]
	Therefore,
	\begin{equation}\label{T1-spatial-representation}
		\mathcal L_1^\pm(\alpha;t,x,y)
		=
		\int_{\mathbb R^2}
		\mathcal K_1^\pm(\alpha;t,x,y,z)
		V(z)\,dz.
	\end{equation}
Here
	\begin{equation}\label{def:K1-high}
		\mathcal K_1^\pm(\alpha;t,x,y,z)
		=
		\int_0^\infty
		e^{-it\lambda^4}
		e^{\pm i\lambda r_1}
		\lambda
		\mathcal E_1^\pm(\alpha;\lambda, x,y,z)\,d\lambda,
	\end{equation}
	with
	$
	r_1=r_1(x,y,z)
	:=
	|x-z|+|y-z|
$
	and the amplitude factor $	\mathcal E_1^\pm$ defined by
	\begin{equation}\label{def:E1-high}
		\mathcal E_1^\pm(\alpha; \lambda, x,y,z)
		:=
		\lambda^{\alpha-2}
		\widetilde\chi_2(\lambda)
		\widetilde R^\pm(\lambda|x-z|)
		\widetilde R^\pm(\lambda|y-z|).
	\end{equation}
	
	By \eqref{free-widetilde_R}, the Leibniz rule, and the fact that
\(\widetilde\chi_2\) is supported away from zero, for
\(\ell=0,1,2\),
	\begin{align}
		\left|
		\partial_\lambda^\ell
		\mathcal E_1^\pm(\alpha;\lambda,x,y,z)
		\right|
		&\lesssim
		\lambda^{\alpha-2-\ell}
		\langle\lambda|x-z|\rangle^{-1/2}
		\langle\lambda|y-z|\rangle^{-1/2}
		\nonumber\\
		&\lesssim
		\lambda^{\alpha-2-\ell}
		\langle\lambda r_1\rangle^{-1/2}\lesssim
		\lambda^{\alpha-\ell}
		\langle\lambda r_1\rangle^{-1/2},\quad \ell=0,1,2.
		\label{E1-sharp-amplitude}
	\end{align}
	Applying Lemma~\ref{quartic-oscillatory} \textup{(i)} with
	\(\gamma=\alpha\), we obtain
	\[
	|\mathcal K_1^\pm(\alpha;t,x,y,z)|
	\lesssim
	|t|^{-\frac{2+\alpha}{4}},
	\]
	uniformly in \(x,y,z\). Since \(V\in L^1(\mathbb R^2)\),
	\eqref{T1-spatial-representation} gives
$	|\mathcal L_1^\pm(\alpha;t,x,y)|
\lesssim
|t|^{-\frac{2+\alpha}{4}}$ 	uniformly in \(x,y\).
	
	We next prove the spatially weighted pointwise estimate for $\mathcal L_1^\pm$. For
	\(\ell=0,1,2\), note that
	\begin{equation}\label{E1-full-phase-derivative}
	\Bigl|\partial_\lambda^\ell \bigl( e^{\pm i\lambda r} \mathcal{E}_1^\pm(\alpha;\lambda,x,y,z)\bigr)\Bigr|
	\lesssim\sum_{\ell_1+\ell_2+\ell_3 \leq\ell}\lambda^{\alpha-2-\ell_1}|x-z|^{\ell_2}|y-z|^{\ell_3}
	\lesssim\lambda^{\alpha-2}\langle x\rangle^\ell\langle y\rangle^\ell\langle z \rangle^\ell.
	\end{equation}
Integrating by parts twice in \eqref{def:K1-high}, we obtain
	\begin{align}
		\mathcal K_1^\pm(\alpha;t,x,y,z)
		=
		\frac{1}{(4it)^2}
		\int_0^\infty
		e^{-it\lambda^4}
		\partial_\lambda
		\left(
		\lambda^{-3}
		\partial_\lambda
		\left(
		\lambda^{-2}
		e^{\pm i\lambda r_1}
		\mathcal E_1^\pm(\alpha;\lambda, x,y,z)
		\right)
		\right)
		\,d\lambda.
		\label{K1-twice-IBP}
	\end{align}
	The boundary terms vanish because
	\(\widetilde\chi_2\) vanishes near zero and
	\(\alpha\leq2\).
	Furthermore, 	by \eqref{E1-full-phase-derivative},
	\[
	\left|
	\partial_\lambda
	\left(
	\lambda^{-3}
	\partial_\lambda
	\left(
	\lambda^{-2}
	e^{\pm i\lambda r_1}
	\mathcal E_1^\pm
	\right)
	\right)
	\right|
	\lesssim
	\lambda^{\alpha-9}
	\langle x\rangle^2
	\langle y\rangle^2
	\langle z\rangle^2.
	\]
Combining the support of $\widetilde{\chi}_2$ with  $\alpha-9<-1$, one has
	\[
	|\mathcal K_1^\pm(\alpha;t,x,y,z)|
	\lesssim
	|t|^{-2}
	\langle x\rangle^2
	\langle y\rangle^2
	\langle z\rangle^2	\int_{\lambda_0}^\infty
	\lambda^{\alpha-9}\,d\lambda\lesssim 	|t|^{-2}
	\langle x\rangle^2
	\langle y\rangle^2
	\langle z\rangle^2.
	\]
	Finally,  since \(|V(x)|\lesssim \langle x\rangle^{-4-}\) implies \(\langle z\rangle^2 |V(z)|\in L^1(\mathbb{R}^2)\), we obtain
$	|\mathcal L_1^\pm(\alpha;t,x,y)|
	\lesssim
	|t|^{-2}
	\langle x\rangle^2\langle y\rangle^2.$
\end{proof}

We next estimate the term containing the perturbed resolvent. We use
the following high-energy limiting absorption estimate.
\begin{lemma}[{\cite[Theorem 2.23]{FSY18}}]\label{RRR}
	Assume that $|V(x)| \lesssim \langle x \rangle^{-\mu}$ for some $\mu > k+1$ with $k \in \mathbb{N}_0$, and that $H = \Delta^2 + V$ has no positive eigenvalues. Then, for any $s>k+\frac{1}{2}, R_V^{ \pm}(\lambda) \in \mathbb{B}\left(L^{2,s}(\mathbb{R}^2), L^{2,-s}(\mathbb{R}^2)\right)$ are $C^k$-continuous for all $\lambda>0$. Furthermore,
	$$
	\left\|\partial_\lambda^{k} R_V^{ \pm}(\lambda)\right\|_{L^{2,s}(\mathbb{R}^2) \rightarrow L^{2,-s}(\mathbb{R}^2)}=O\left(|\lambda|^{-\frac{3(k+1)}{4}}\right),~\lambda \rightarrow+\infty.
	$$
\end{lemma}

\begin{proposition}\label{high-iterated-term}
	Assume that
$
|V(x)|\lesssim\langle x\rangle^{-4-}
$. Then for every
	\(-2<\alpha\leq2\),
	\begin{equation*}
		\sup_{x,y\in\mathbb R^2}
		|\mathcal L_2^\pm(\alpha;t,x,y)|
		\lesssim
		|t|^{-\frac{2+\alpha}{4}},
		\quad
		|\mathcal L_2^\pm(\alpha;t,x,y)|
		\lesssim
		|t|^{-2}
		\langle x\rangle^2\langle y\rangle^2.
	\end{equation*}
\end{proposition}

\begin{proof}
	For \(x, y\in\mathbb R^2\), we have
	\begin{align*}
		&
		\bigl[
		R_0^\pm(\lambda^4)
		VR_V^\pm(\lambda^4)
		VR_0^\pm(\lambda^4)
		\bigr](x,y)
	=
		\lambda^{-4}
		e^{\pm i\lambda(|x|+|y|)}
		\mathcal{K}_2^\pm(\lambda;x,y),
		\label{T2-phase-factorization}
	\end{align*}
	where
\begin{equation}\label{Kpm-2}
	\mathcal{K}_2^{\pm}(\lambda;x,y):=	\left\langle
	R_V^\pm(\lambda^4)
	V\Phi_y^\pm(\lambda,\cdot),
	V\Phi_x^\mp(\lambda,\cdot)
	\right\rangle \  \text{with}\ 
	\Phi_x^\pm(\lambda,z)
	:=
	e^{\pm i\lambda(|x-z|-|x|)}
	\widetilde {R}^{\pm}(\lambda|x-z|)
	\end{equation}
	Thus $	\mathcal L_2^\pm$ can be expressed as 
	\begin{equation}\label{T2-oscillatory-representation}
		\mathcal L_2^\pm(\alpha;t,x,y)
		=
		\int_0^\infty
		e^{-it\lambda^4}
		e^{\pm i\lambda r_2}
		\lambda
		\mathcal E_2^\pm(\alpha; \lambda, x,y)\,d\lambda.
	\end{equation}
	Here
	$
	r_2:=|x|+|y|
	$
	and
	\begin{equation}\label{def:E2-high}
		\mathcal E_2^\pm(\alpha;\lambda, x,y)
		:=
		\lambda^{\alpha-2}
		\widetilde\chi_2(\lambda)
		\mathcal K_2^\pm(\lambda;x,y).
	\end{equation}
	
	For \(\ell=0,1,2\), we have
	\begin{equation}\label{phase-derivative-high}
		\left|
		\partial_\lambda^\ell
		e^{\pm i\lambda(|x-z|-|x|)}
		\right|
		\lesssim
		\langle z\rangle^\ell\quad\text{and}\quad
		\left|
		\partial_\lambda^\ell
		\widetilde R^\pm(\lambda|x-z|)
		\right|
		\lesssim
		\langle\lambda x\rangle^{-1/2}
		\langle z\rangle^{1/2}
		\lambda^{1/2-\ell}.
	\end{equation}
	where the second bound utilizes the inequality  $\langle \lambda(x-z) \rangle^{-1/2} \lesssim \langle \lambda x \rangle^{-1/2} \langle  z \rangle^{1/2} \lambda^{1/2}$ for $\lambda \gtrsim 1$.
	By Lemma \ref{RRR}, for $s>\ell+\frac{1}{2}$ with $\ell=0,1,2$:
	\begin{equation}\label{eq:inequality_2}
		\left\|\partial_\lambda^{\ell} R_V^{ \pm}(\lambda^4)\right\|_{L^{2,s} \rightarrow L^{2,-s}} \lesssim \lambda^{-3}.
	\end{equation}
	Given $|V(x)| \lesssim \langle x \rangle^{-4-}$, 	combining \eqref{phase-derivative-high},\eqref{Kpm-2},\eqref{eq:inequality_2}  with H\"older's inequality yields for  $\ell=0,1,2$,
	$$
	\begin{aligned}
		\left|\partial_\lambda^{\ell} \mathcal K_2^{\pm}(\lambda,x,y)\right|
		& \lesssim \frac{1}{\left\langle \lambda x\right\rangle^{\frac{1}{2}}\left\langle \lambda y\right\rangle^{\frac{1}{2}}}\sum_{\ell_1+\ell_2 \le \ell}
		\lambda \left\|V(\cdot)\langle\cdot\rangle^{\ell_1+s+\frac{1}{2}}\right\|_{L^2}^2 \cdot\left\|\partial_\lambda^{\ell_2} R_V^{\pm}(\lambda^{4} )\right\|_{L^{2,s} \rightarrow L^{2,-s}}  \\
		& \lesssim \lambda^{-2}  {{\left\langle \lambda x\right\rangle}^{-\frac{1}{2}}{\left\langle \lambda y\right\rangle}^{-\frac{1}{2}}}
		\lesssim \left\langle \lambda (|x|+|y|)\right\rangle^{-\frac{1}{2}}\lambda^{-2},\quad \lambda \gtrsim1,
	\end{aligned}
	$$
	with $s > \ell_2 + \frac{1}{2}$.
		Hence, by the definition \eqref{def:E2-high} of the amplitude factor $\mathcal{E}_2^{\pm},$
	it follows that for $\lambda \gtrsim1,$
	\begin{align}
		\left|
		\partial_\lambda^\ell
		\mathcal E_2^\pm(\alpha; \lambda, x,y)
		\right|
		&\lesssim
		\lambda^{\alpha-4-\ell}
		\langle\lambda r_2\rangle^{-1/2}
	\lesssim
		\lambda^{\alpha-\ell}
		\langle\lambda r_2\rangle^{-1/2},
		\qquad
		\ell=0,1,2,
		\label{E2-quartic-bound}
	\end{align}
	Lemma~\ref{quartic-oscillatory}\textup{(i)} with
	\(\gamma=\alpha\), now gives
	$
	|\mathcal L_2^\pm(\alpha;t,x,y)|
	\lesssim
	|t|^{-\frac{2+\alpha}{4}},
	$
	uniformly in \(x,y\).
	
	To prove the pointwise estimate, one obtains from
	\eqref{E2-quartic-bound}, for \(\ell=0,1,2\),
	\begin{equation}\label{E2-full-phase-derivative}
		\left|
		\partial_\lambda^\ell
		\left[
		e^{\pm i\lambda r_2}
		\mathcal E_2^\pm(\alpha;\lambda, x,y)
		\right]
		\right|
		\lesssim
		\lambda^{\alpha-4}
		\langle x\rangle^2
		\langle y\rangle^2.
	\end{equation}
	Integrating by parts twice in
	\eqref{T2-oscillatory-representation}, as in
	\eqref{K1-twice-IBP}, yields
	\begin{align*}
		|\mathcal L_2^\pm(\alpha;t,x,y)|
		&\lesssim
		|t|^{-2}
		\langle x\rangle^2\langle y\rangle^2
		\int_{\lambda_0}^\infty
		\lambda^{\alpha-11}\,d\lambda\lesssim
		|t|^{-2}
		\langle x\rangle^2\langle y\rangle^2.
	\end{align*}
\end{proof}

\begin{proof}[Proof of Theorem~\ref{main-theorem-high-schrodinger}]
	By Propositions~\ref{high-first-born} and
	\ref{high-iterated-term}, for \(j=1,2\),
	\begin{equation*}
	\sup_{x,y\in\mathbb{R}^2}	|\mathcal L_j^\pm(\alpha;t,x,y)|
		\lesssim
		|t|^{-\frac{2+\alpha}{4}},
\quad
		|\mathcal L_j^\pm(\alpha;t,x,y)|
		\lesssim
		|t|^{-2}
		\langle x\rangle^2\langle y\rangle^2,
	\end{equation*}
	together with  the inequality $\min(1,a/b) \lesssim (\log a/\log b)^s$ (valid for $a,b\ge 2$ and fixed $s>0$),
	it follows that for $|t|\ge 2$ and every fixed $s>0$,
	\[|\mathcal L_j^{\pm}(\alpha; t,x,y)|\lesssim
	|t|^{-\frac{2+\alpha}{4}}	\min\left(1,\frac{\langle x\rangle^{2}\langle y\rangle^{2}}{|t|^{\frac{6-\alpha}{4}}}\right)
	\lesssim\frac{\bigl(\omega(x)\omega(y)\bigr)^s}{	|t|^{\frac{2+\alpha}{4}}\bigl(\log|t|\bigr)^s}.
	\]
	Combine these with
	the free high-energy estimate (see \eqref{free-cutoff-unweighted} and \eqref{free-high-weighted} in Proposition \ref{free-energy-decomposition}), we complete the proof of  Theorem~\ref{main-theorem-high-schrodinger}.
\end{proof}

\section{Low-energy estimates in the regular and first-kind resonance cases}
\label{sec:regular_first_sch}

Since   the
high-energy estimates have been  established in Section \ref{sec:high-energy}, it remains to study the
low-energy part 
$
H^{\alpha/4}e^{-itH}P_{\mathrm{ac}}(H)\chi_1(H).
$
We begin with a
general strategy, valid for all threshold resonances, and then
specialize in this section to the cases where zero is either a
regular point or a first-kind resonance.

Recall Stone's formula
\begin{equation}\label{stone-low-sch}
	H^{\alpha/4}e^{-itH}P_{\mathrm{ac}}(H)\chi_1(H)
	=
	\frac{2}{\pi i}
	\int_0^\infty
	e^{-it\lambda^4}
	\lambda^{\alpha+3}
	\widetilde\chi_1(\lambda)
	\bigl[
	R_V^+(\lambda^4)-R_V^-(\lambda^4)
	\bigr]\,d\lambda,
\end{equation}
where
$
\widetilde\chi_1(\lambda)
:=
\chi_1(\lambda^4)
$
is supported on $ |\lambda|\leq(2\lambda_0)^{1/4}$.
When zero is of spectral type
\(\mathbf{k}\in\{0,1,2,3,4\}\), substituting
\eqref{eq:M_inverse} into the symmetric resolvent identity
\eqref{id-RV} gives
\[
R_V^{\pm}(\lambda^4)
=
R_0^\pm(\lambda^{4})
-
R_0^\pm(\lambda^{4})v
\Big(
\sum_{0\leq j,l\leq \mathbf{k}+2}
\lambda^{2-k_j-k_l}
Q_j\mathcal{M}_{j,l}^{\pm}(\lambda)Q_l
\Big)
vR_0^\pm(\lambda^{4}).
\]
Inserting this representation back into \eqref{stone-low-sch}, it follows that for fixed $\alpha\in(-2,2],$
\begin{align}\label{Stone-operator}
	H^{\alpha/4}e^{-itH}P_{\mathrm{ac}}(H)\chi_1(H)
	=	(-\Delta)^{\alpha/2}e^{-it\Delta^2}
	\chi_1(\Delta^2)-\sum_{B^\pm(\lambda)\in\Theta_{\mathbf{k}}}\left(	T_{\mathcal I_{B}^+}-T_{	\mathcal I_{B}^-}\right),
\end{align}
where $T_{\mathcal I_{B}^\pm}$
is the integral operator with the kernel $\mathcal I_{B}^\pm$ defined by 
\begin{align}
	\mathcal I_{B}^\pm(\alpha; t,x,y)
	:={}&
	\frac{2}{\pi i}
	\int_0^\infty
	e^{-it\lambda^4}
	\lambda\widetilde\chi_1(\lambda)
	\lambda^{\alpha+4-k_j-k_l}
	\left\langle
	B^\pm(\lambda)
	Q_lvR_0^\pm(\lambda^4)(\cdot,y),
	Q_jvR_0^\mp(\lambda^4)(\cdot,x)
	\right\rangle
	\,d\lambda,
	\label{def:I-B-alpha}
\end{align}
and  $B^\pm(\lambda) \in\Theta_{\mathbf{k}} :=\{ \mathcal{M}_{j, l}^{\pm}(\lambda) \mid j,l\in\mathbb{Z},\ 0 \le j,l \le \mathbf{k}+2 \}.$ Considering the free low-energy estimate established  in Proposition \ref{free-energy-decomposition}, it reduces to investigating the kernel differences:
$$\mathcal I_{B}^+(\alpha; t,x,y)-\mathcal I_{B}^-(\alpha; t,x,y).$$

\vskip0.3cm
We now specialize to the regular and first-kind resonance cases,
i.e. \(\mathbf{k}\in\{0,1\}\). The aim of this section is to prove
Theorem~\ref{main-theorem-low-sch}. Combined with the high-energy
estimates established in Section~\ref{sec:high-energy}, this yields
Theorem~\ref{main_theorem-1} \textup{(i)} and
Theorem~\ref{main2-weight}.
\begin{theorem}\label{main-theorem-low-sch}
	Let
	$
	H=\Delta^2+V
	$
with $
	|V(x)|
	\lesssim
	\langle x\rangle^{-11-}.
	$
	Suppose that \(H\) has no positive embedded eigenvalues and that zero
	is either a regular point or a first-kind resonance of \(H\). Then,
	for every
	$
	-2<\alpha\leq2,
	$
	\begin{equation}\label{low-sch-unweighted}
		\left\|
		H^{\frac{\alpha}{4}}e^{-itH}
		P_{\mathrm{ac}}(H)\chi_1(H)
		\right\|_{L^1\to L^\infty}
		\lesssim
		\langle t\rangle^{-\frac{2+\alpha}{4}}.
	\end{equation}
	Moreover,  for $|t|\geq2$ and  $s>0,$
	\begin{equation}\label{low-sch-weighted}
		\left\|
		\omega^{-s}H^{\frac{\alpha}{4}}e^{-itH}
		P_{\mathrm{ac}}(H)\chi_1(H)\omega^{-s}
	\right\|_{L^1\to L^\infty}
		\lesssim
		\frac{1}
		{|t|^{\frac{2+\alpha}{4}}(\log|t|)^s},
	\end{equation}
     where $\omega(x) = \log(2+|x|)$ is the logarithmic weight function.
\end{theorem}

The proof is divided into the unweighted and weighted estimates.

\subsection{The unweighted estimate}
\label{subsec:regular_first_sch_unweighted}

Recall from Theorem~\ref{thm:M_inverse} that, for
\(0\leq j,l\leq3\) and \(\ell=0,1,2\),
\begin{equation}\label{M-jk-regular-first}
	\left\|
	\partial_\lambda^\ell
	\mathcal M_{j,l}^\pm(\lambda)
	\right\|_{\mathbb B(L^2)}
	\lesssim
	\begin{cases}
		\lambda^{-\ell},
		& (j,l)\neq(3,3),\\[2pt]
		|\log\lambda|\lambda^{-\ell},
		& (j,l)=(3,3).
	\end{cases}
\end{equation}

\begin{proposition}\label{prop:I-alpha-unweighted}
	Let
	$
	B^\pm(\lambda)=\mathcal M_{j,l}^\pm(\lambda),
	$ with $0\leq j,l\leq3.$
	Then, for every \(-2<\alpha\leq2\),
	\[
	\sup_{x,y\in\mathbb R^2}
	\left|
	\mathcal I_{B}^\pm(\alpha; t,x,y)
	\right|
	\lesssim
	\langle t\rangle^{-\frac{2+\alpha}{4}}.
	\]
\end{proposition}

\begin{proof}
	By Lemma~\ref{lemma_projection} (i) and noting that $-1-\delta_{\beta0}=k_\beta-2$ for all \(0\leq \beta\leq3\), it follows that
	\begin{equation}\label{Omega-representation-sch}
		\bigl(
		Q_\beta vR_0^\pm(\lambda^4)
		\bigr)(\cdot,z)
		=
		\lambda^{k_\beta-2}
		\Omega_{\beta,\pm}(\lambda,\cdot,z),
	\end{equation}
	where, for \(\ell=0,1,2\) and $0\leq \beta\leq3,$
	\begin{equation}\label{Omega-estimate-sch}
		\left\|
		\partial_\lambda^\ell
		\left(
		e^{\mp i\lambda|z|}
		\Omega_{\beta,\pm}(\lambda,\cdot,z)
		\right)
		\right\|_{L^2}
		\lesssim
		\lambda^{-\ell}
		\langle\lambda z\rangle^{-1/2}.
	\end{equation}
	For \(\beta=3\), we shall also use the improved representation
	\begin{equation}\label{J3-representation-sch}
		\bigl(
		Q_3vR_0^\pm(\lambda^4)
		\bigr)(x,z)
		=
		\mathcal J_{3,\pm}(\lambda,x,z),
	\end{equation}
	together with
	\begin{equation}\label{J3-estimate-sch}
		\left\|
		\partial_\lambda^\ell
		\left(
		e^{\mp i\lambda|z|}
		\mathcal J_{3,\pm}(\lambda,\cdot,z)
		\right)
		\right\|_{L^2}
		\lesssim
		\lambda^{-\ell}
		|\log\lambda|
		\langle\lambda z\rangle^{-1/2}.
	\end{equation}
	
	We distinguish two cases.
	
	\medskip
	
	\noindent
	\textit{Case 1: \((j,l)\neq(3,3)\).}
	Using \eqref{Omega-representation-sch}, we have
	\[
	\begin{aligned}
		&
		\left\langle
		B^\pm(\lambda)
		Q_lvR_0^\pm(\lambda^4)(\cdot,y),
		Q_jvR_0^\mp(\lambda^4)(\cdot,x)
		\right\rangle
	=
		\lambda^{k_j+k_l-4}
		\left\langle
		B^\pm(\lambda)
		\Omega_{l,\pm}(\lambda,\cdot,y),
		\Omega_{j,\mp}(\lambda,\cdot,x)
		\right\rangle.
	\end{aligned}
	\]
	Substituting this identity into \eqref{def:I-B-alpha} yields
	\begin{equation}\label{canonical-I-alpha}
		\mathcal I_{B}^\pm(\alpha; t,x,y)
		=
		\frac{2}{\pi i}
		\int_0^\infty
		e^{-it\lambda^4}
		e^{\pm i\lambda r}
		\lambda
		\mathcal E_{B}^\pm(\alpha; \lambda,x,y)
		\,d\lambda,
	\end{equation}
	where
	$
	r=|x|+|y|
	$
	and
	\begin{align}
		\mathcal E_{B}^\pm(\alpha; \lambda,x,y)
		:={}&
		\lambda^\alpha
		\widetilde\chi_1(\lambda)
		\left\langle
		B^\pm(\lambda)
\big(e^{\mp i\lambda |y|}	\Omega_{l,\pm})(\lambda,\cdot,y),
		\big(e^{\pm i\lambda|x|}\Omega_{j,\mp}\big)(\lambda,\cdot,x)
		\right\rangle.
		\label{def:E-B-alpha}
	\end{align}
	Applying Leibniz's rule together with H\"older's inequality gives, for $\ell = 0,1,2$,
	\[
	\begin{aligned}
		\bigl|\partial_\lambda^{\ell} 	\mathcal E_{B}^\pm(\alpha; \lambda,x,y) \bigr|
		\lesssim \sum_{\ell_1+\ell_2+\ell_3 =\ell}
		&\bigl\|\partial_\lambda^{\ell_1}\left( \lambda^\alpha\widetilde{\chi}_1(\lambda) B^\pm (\lambda)\right) \bigr\|_{\mathbb{B}(L^2)} \\
		&\times \bigl\|\partial_\lambda^{\ell_2} \bigl(e^{\mp i \lambda|y|}\Omega_{l\pm}(\lambda,\cdot,y)\bigr) \bigr\|_{L^2}
		\bigl\|\partial_\lambda^{\ell_3} \bigl(e^{\pm i \lambda |x|}\Omega_{j,\mp}(\lambda,\cdot,x)\bigr) \bigr\|_{L^2}.
	\end{aligned}
	\]
	Combining this with the estimate \eqref{M-jk-regular-first}, and
	\eqref{Omega-estimate-sch} yields the amplitude factor bound
\[
\begin{aligned}
	\left|
	\partial_\lambda^\ell
	\mathcal E_{B}^\pm(\alpha; \lambda,x,y)
	\right|
	\lesssim{}&
	\lambda^{\alpha-\ell}
	\langle\lambda x\rangle^{-1/2}
	\langle\lambda y\rangle^{-1/2}
	\widetilde\chi_1(\lambda/2)
	\lesssim
	\lambda^{\alpha-\ell}
	\langle\lambda r\rangle^{-1/2}
	\widetilde\chi_1(\lambda/2).
\end{aligned}
\]	Lemma~\ref{quartic-oscillatory}\textup{(ii)} with $(\sigma,\nu)=(\alpha, 0)$  therefore gives
	$
	|
	\mathcal I_{B}^\pm(\alpha; t,x,y)
	|
	\lesssim
	\langle t\rangle^{-\frac{2+\alpha}{4}},
	$ uniformly in $x, y.$
	
	\medskip
	
	\noindent
	\textit{Case 2: \((j,l)=(3,3)\).}
	Here
	$
	4-k_3-k_3=2.
	$
	Using \eqref{J3-representation-sch}, we again obtain the canonical
	form \eqref{canonical-I-alpha}  with
	\begin{align}
		\mathcal E_{B}^\pm(\alpha; \lambda,x,y)
		:={}&
		\lambda^{\alpha+2}
		e^{\mp i\lambda r}
		\widetilde\chi_1(\lambda)
		\left\langle
		B^\pm(\lambda)
		\mathcal J_{3,\pm}(\lambda,\cdot,y),
		\mathcal J_{3,\mp}(\lambda,\cdot,x)
		\right\rangle,\quad 	r=|x|+|y|.
		\label{def:E-33-alpha}
	\end{align}
	By \eqref{M-jk-regular-first} and
	\eqref{J3-estimate-sch}, Leibniz's rule and H\"older's inequality, 
	\[
	\begin{aligned}
		\left|
		\partial_\lambda^\ell
		\mathcal E_{B}^\pm(\alpha; \lambda,x,y)
		\right|
		&\lesssim
		\lambda^{\alpha+2-\ell}
		|\log\lambda|^3
		\langle\lambda x\rangle^{-1/2}
		\langle\lambda y\rangle^{-1/2}
		\widetilde\chi_1(\lambda/2)
		\\
		&\lesssim
		\lambda^{\alpha-\ell}
		\langle\lambda r\rangle^{-1/2}
		\widetilde\chi_1(\lambda/2),\quad 0<\lambda\ll1.
	\end{aligned}
	\]
	A second application of
Lemma~\ref{quartic-oscillatory}\textup{(ii)} with
\((\sigma,\nu)=(\alpha,0)\) proves the desired estimate.
\end{proof}

\subsection{The weighted  estimate}
\label{subsec:regular_first_sch_weighted}

We now prove the logarithmically improved estimate
\eqref{low-sch-weighted}. Recall that \(Q_0=P\) is the orthogonal
projection onto \(\operatorname{span}\{v\}\), and hence
$
Q_0v=v.
$
In particular, the \(Q_0\)-component does not possess a cancellation
against the leading constant term in the expansion of the
free resolvent $R_0^\pm(\lambda^4)$. By contrast, if \((j,l)\neq(0,0)\), then at least one
of the indices \(j\) and \(l\) is positive. The corresponding identity
$
Q_jv=0$
or
$Q_lv=0$
eliminates the most singular \(\lambda^{-2}\)-term from at least one
of the two projected free resolvent.

Accordingly, we divide the analysis into the following two cases:
\begin{itemize}
	\item[(I)]
	\(B^\pm(\lambda)=\mathcal M_{j,l}^\pm(\lambda)\) with
	\((j,l)\neq(0,0)\); see
	Proposition~\ref{prop:I-alpha-weighted-nonzero};
	    \vskip0.1cm
	\item[(II)]
	\(B^\pm(\lambda)=\mathcal M_{0,0}^\pm(\lambda)\); see
	Proposition~\ref{prop:I-alpha-weighted-00}.
\end{itemize}
By Theorem~\ref{thm:M_inverse} (I), for \((j,l)\neq(0,0)\), we have 
\begin{equation}\label{M-nonzero-weighted-sch}
	\big\|
	\partial_\lambda^\ell
	\mathcal M_{j,l}^\pm(\lambda)
	\big\|_{\mathbb B(L^2)}
	\lesssim
	|\log\lambda|\lambda^{-\ell},
	\qquad \ell=0,1,2,
\end{equation}
while
\begin{equation*}\label{M00-expansion-sch}
	\mathcal M_{0,0}^\pm(\lambda)
	=
	(a_0^\pm)^{-1}
	\|V\|_{L^1}^{-1}Q_0
	+
	\lambda(-\log\lambda)^{3/2}
	\Lambda^\pm(\lambda),
\end{equation*}
where $\Lambda^\pm(\lambda)$ denotes a generic operator in $\mathbb{B}(L^2)$ satisfying
$
	\|
	\partial_\lambda^\ell
	\Lambda^\pm(\lambda)
	\|_{\mathbb B(L^2)}
	\lesssim
	\lambda^{-\ell},$ for
	$\ell=0,1,2.$

We also use the free resolvent expansion obtained by \eqref{reso-big} and \eqref{reso_small}:
\begin{equation}\label{R0-E1-sch}
	R_0^\pm(\lambda^4)(x,y)
	=
	a_0^\pm\lambda^{-2}G_0(x,y)
	+
	E_1^\pm(\lambda,x,y),
	\qquad
	G_0(x,y)\equiv1,
\end{equation}
where, for \(\ell=0,1,2\),
\begin{equation}\label{E1-derivative-sch}
	\left|
	\partial_\lambda^\ell
	E_1^\pm(\lambda,x,y)
	\right|
	\lesssim
	\lambda^{-\ell-}
	\langle x\rangle^4
	\langle y\rangle^4,\quad \quad 0<\lambda\ll1.
\end{equation}
For this weighted analysis, define
\begin{align}
	\mathcal E_{B}^\pm(\alpha; \lambda,x,y)
	:={}&
	\widetilde\chi_1(\lambda)
	\lambda^{\alpha+4-k_j-k_l}
	\left\langle
	B^\pm(\lambda)
	Q_lvR_0^\pm(\lambda^4)(\cdot,y),
	Q_jvR_0^\mp(\lambda^4)(\cdot,x)
	\right\rangle .
	\label{def:W-B-alpha}
\end{align}
Then
\begin{equation}\label{I-weighted-canonical}
	\mathcal I_{B}^\pm(\alpha; t,x,y)
	=
	\frac{2}{\pi i}
	\int_0^\infty
	e^{-it\lambda^4}
	\lambda
	\mathcal E_{B}^\pm(\alpha;\lambda,x,y)
	\,d\lambda.
\end{equation}

\begin{proposition}\label{prop:I-alpha-weighted-nonzero}
 	Let
	$
	B^\pm(\lambda)=\mathcal M_{j,l}^\pm(\lambda),$ for
	$(j,l)\neq(0,0).
	$
	Then, for every    \(-2<\alpha\leq2\) and $s>0$,
	\[
	\left|
	\mathcal I_{B}^\pm(\alpha;t,x,y)
	\right|
	\lesssim
	\frac{\bigl(\omega(x)\omega(y)\bigr)^s}
	{|t|^{\frac{2+\alpha}{4}}(\log|t|)^s},\quad |t|\geq2.
	\]
\end{proposition}

\begin{proof}
	If \(1\leq\beta\leq3\), then \(Q_\beta v=0\). It follows from
\eqref{R0-E1-sch} that
\[
Q_\beta vR_0^\pm(\lambda^4)
=
Q_\beta vE_1^\pm(\lambda).
\]
	Consequently, by \eqref{E1-derivative-sch}, for $1\leq\beta\leq3$ and $\ell=0,1,2,$
	\begin{equation}\label{Qj-R0-weighted}
		\left\|
		\partial_\lambda^\ell
		\bigl(
		Q_\beta vR_0^\pm(\lambda^4)
		\bigr)(\cdot,y)
		\right\|_{L^2}
		\lesssim
		\lambda^{-\ell-}
		\langle y\rangle^4.
	\end{equation}
	For \(\beta=0\), no such cancellation is available, and
	\begin{equation}\label{Q0-R0-weighted}
		\left\|
		\partial_\lambda^\ell
		\bigl(
		Q_0vR_0^\pm(\lambda^4)
		\bigr)(\cdot,y)
		\right\|_{L^2}
		\lesssim
		\lambda^{-2-\ell}
		\langle y\rangle^4.
	\end{equation}
	
	Since \((j,l)\neq(0,0)\), at least one of the two projected resolvent
factors satisfies \eqref{Qj-R0-weighted}. Using
\eqref{M-nonzero-weighted-sch},
\eqref{Qj-R0-weighted}, and \eqref{Q0-R0-weighted}, together with
Leibniz's rule and H\"older's  inequality, we obtain
	\begin{equation}
		\label{W-nonzero-estimate}
		\left|
		\partial_\lambda^\ell
		\mathcal E_{B}^\pm(\alpha; \lambda,x,y)
		\right|
		\lesssim
		\lambda^{\alpha+\frac12-\ell}
		\widetilde\chi_1(\lambda/2)
		\langle x\rangle^4
		\langle y\rangle^4,
		\qquad \ell=0,1,2.
	\end{equation}
	Applying Lemma~\ref{quartic-oscillatory}\textup{(ii)} with
	$
	r=0,
	(\sigma,\nu)=(\alpha+1/2, 0)$ (here $\alpha+1/2\in(-3/2, 5/2]\subset(-2, 6)$), we have
	\begin{equation}\label{weighted-strong-decay}
		\left|
		\mathcal I_{B}^\pm(\alpha; t,x,y)
		\right|
		\lesssim
		|t|^{-\frac{5+2\alpha}{8}}
		\langle x\rangle^4
		\langle y\rangle^4.
	\end{equation}
	On the other hand, Proposition~\ref{prop:I-alpha-unweighted} gives the uniform bound 
	$
	|
	\mathcal I_{B}^\pm|
	\lesssim
	|t|^{-\frac{2+\alpha}{4}}.
	$
	Therefore, by the inequality \( \min(1, a/b) \lesssim (\log a/\log b)^s \) (\( a,b \ge 2 \) and fixed $s>0$),
	we have for every fixed $s>0$ and $-2<\alpha\leq2,$
\begin{align}\label{weight-0-1}
	\left|
	\mathcal I_{B}^\pm(\alpha; t,x,y)
\right	|
	&\lesssim
	|t|^{-\frac{2+\alpha}{4}}
	\min\left\{
	1,
	\frac{\langle x\rangle^4\langle y\rangle^4}{|t|^{1/8}}
	\right\}
	\lesssim
	\frac{\bigl(\omega(x)\omega(y)\bigr)^s}
	{|t|^{\frac{2+\alpha}{4}}(\log|t|)^s},
	\qquad |t|\geq2.
\end{align}
\end{proof}

\begin{proposition}\label{prop:I-alpha-weighted-00}
	Let
	$
	B^\pm(\lambda)=\mathcal M_{0,0}^\pm(\lambda).
	$
	Then, for every  \(-2<\alpha\leq2\) and $s>0$,
	\begin{equation}\label{I00-alpha-expansion}
		\begin{aligned}
			&
			\bigl(
			\mathcal I_{B}^+
			-
			\mathcal I_{B}^-
			\bigr)(\alpha; t,x,y)
		=
				\frac{1}{8\pi}
			\Gamma\left(\frac{2+\alpha}{4}\right)
			e^{-i\frac{(2+\alpha)\pi}{8}\operatorname{sgn}(t)}
			|t|^{-\frac{2+\alpha}{4}}
			+
			O\bigg(
			\frac{\bigl(\omega(x)\omega(y)\bigr)^s}
			{|t|^{\frac{2+\alpha}{4}}(\log|t|)^s}
			\bigg).
		\end{aligned}
	\end{equation}
	Here, \(\Gamma\) denotes the Gamma function.
\end{proposition}

\begin{proof}
	Substituting  $R_0^{\pm}(\lambda^4)(x,y) = a_0^{\pm} \lambda^{-2}G_0(x,y) + E_1^{\pm}(\lambda,x,y)$ into \eqref{def:W-B-alpha} and observing that $(Q_0vG_0)(\cdot,y)=(Q_0v)(\cdot)=v(\cdot)$, we separate an $x,y$-independent principal part and a remainder:
	\[
	\mathcal E_{B}^\pm(\alpha; \lambda,x,y) =\lambda^{\alpha} \widetilde{\chi}_1(\lambda) (a_0^\pm)^2 \bigl\langle \mathcal{M}_{0,0}^\pm(\lambda) v, v \bigr\rangle + \mathcal{E}_{\mathrm{rem}}^\pm(\lambda,x,y).
	\]
	Since $E_1^\pm$ satisfies the bounds \eqref{E1-derivative-sch}, $\mathcal{E}_{\mathrm{rem}}^\pm$ obeys the estimate \eqref{W-nonzero-estimate}, ensuring that its contribution to the kernel $\mathcal{I}_B^\pm$ is bounded by $	|t|^{-\frac{5+2\alpha}{8}}
	\langle x\rangle^4
	\langle y\rangle^4$.
	
	For the principal part, using the asymptotic expansion $\mathcal M_{0,0}^\pm(\lambda)
	=
	(a_0^\pm)^{-1}
	\|V\|_{L^1}^{-1}Q_0
	+
	\lambda(-\log\lambda)^{3/2}
	\Lambda^\pm(\lambda)$ together with $\langle Q_0 v, v \rangle=\|V\|_{L^1},$ we obtain that
	\[
\lambda^\alpha	\widetilde{\chi}_1(\lambda) (a_0^\pm)^2 \bigl\langle \mathcal{M}_{0,0}^\pm(\lambda) v, v \bigr\rangle = a_0^\pm \lambda^\alpha \widetilde{\chi}_1(\lambda) + \lambda^{1+\alpha} (-\log \lambda)^{3/2} \widetilde{\chi}_1(\lambda) (a_0^\pm)^2 \bigl\langle \Lambda^{\pm}(\lambda) v, v \bigr\rangle.
	\]
	The second term on the right at least satisfies the bounds \eqref{W-nonzero-estimate}, yielding that its contribution to the kernel $\mathcal{I}_B^\pm$  is bounded by $	|t|^{-\frac{5+2\alpha}{8}}
	\langle x\rangle^4
	\langle y\rangle^4$.
	
	It remains to evaluate the following integral:
	\begin{align*}
\frac{2}{\pi i}	\bigl(a_0^+-a_0^-\bigr)\int_0^\infty e^{-it\lambda^4}\lambda^{1+\alpha}\widetilde{\chi}_1(\lambda)d\lambda.
	\end{align*}
	By $a_0^+-a_0^-=i/4$ together with  \eqref{low-scalar-alpha-less-two}
	and \eqref{low-scalar-alpha-two},
	for all $-2<\alpha\leq2,$ we have
		\begin{align*}
		\frac{2}{\pi i}	\bigl(a_0^+-a_0^-\bigr)\int_0^\infty e^{-it\lambda^4}\lambda^{1+\alpha}\widetilde{\chi}_1(\lambda)d\lambda=	\frac{1}{8\pi}
		\Gamma\left(\frac{2+\alpha}{4}\right)
		e^{-i\frac{(2+\alpha)\pi}{8}\operatorname{sgn}(t)}
		|t|^{-\frac{2+\alpha}{4}}
		+
		O(|t|^{-N}),\quad \forall N>0.
	\end{align*}
Hence, 	
	\begin{align}\label{mathcal I}
			\bigl(
		\mathcal I_{B}^+
		-
		\mathcal I_{B}^-
		\bigr)(\alpha; t,x,y)
		=	\frac{1}{8\pi}
		\Gamma\left(\frac{2+\alpha}{4}\right)
		e^{-i\frac{(2+\alpha)\pi}{8}\operatorname{sgn}(t)}
		|t|^{-\frac{2+\alpha}{4}}+O\bigg(\frac{\langle x\rangle^4
			\langle y\rangle^4}{|t|^{\frac{5+2\alpha}{8}}}\bigg).
	\end{align}

	By \eqref{mathcal I} and $|\mathcal{I}_{B}^\pm|\lesssim| t|^{-\frac{2+\alpha}{4}}$ uniformly in $x,y$ (see Proposition \ref{prop:I-alpha-unweighted}), one has the remainder term in \eqref{mathcal I}  also satisfies
	the uniform bound $| t|^{-\frac{2+\alpha}{4}}$. Then by \eqref{weight-0-1}, we  obtain  the desired asymptotic expansion.
\end{proof}

\begin{proof}[Proof of Theorem~\ref{main-theorem-low-sch}]
	The unweighted estimate \eqref{low-sch-unweighted} follows from  Proposition~\ref{prop:I-alpha-unweighted}, the free
	low-energy estimate in Proposition~\ref{free-energy-decomposition}
	and the decomposition \eqref{Stone-operator}.

	We now prove \eqref{low-sch-weighted}. 
		For the pair \((j,l)=(0,0)\),
	Proposition~\ref{prop:I-alpha-weighted-00} gives
	\[
	\bigl(
	\mathcal I_{\mathcal M_{0,0}}^+
	-
	\mathcal I_{\mathcal M_{0,0}}^-
	\bigr)(\alpha; t,x,y)
	=
	\frac{1}{8\pi}
	\Gamma\left(\frac{2+\alpha}{4}\right)
	e^{-i\frac{(2+\alpha)\pi}{8}\operatorname{sgn}(t)}
	|t|^{-\frac{2+\alpha}{4}}
	+
	O\bigg(
	\frac{\bigl(\omega(x)\omega(y)\bigr)^s}
	{|t|^{\frac{2+\alpha}{4}}(\log|t|)^s}
	\bigg),\quad |t|\geq2.
	\]
	The leading term cancels exactly with that of the free kernel in
	\eqref{free-low-asymptotic}. For all remaining pairs
	\((j,l)\neq(0,0)\), the desired weighted bound follows from
	Proposition~\ref{prop:I-alpha-weighted-nonzero}. Combining these
	estimates with \eqref{Stone-operator} yields
	\eqref{low-sch-weighted}, completing the proof.
\end{proof}

\section{The second kind resonance case}\label{sec:second}
This section  is aim to show Theorem 
\ref{main_theorem-1} (ii).
Together with   the high-energy estimates in
Theorem \ref{main-theorem-high-schrodinger}, it remains to prove the following
Theorem \ref{main_theorem_low_2}

\begin{theorem}\label{main_theorem_low_2}
Let $H = \Delta^2 + V$ with $|V(x)| \lesssim \langle x \rangle^{-14-}$.
Assume that $H$ has no positive embedded eigenvalues and that zero is a second-kind resonance of $H$. Then,
for every
$
-2<\alpha\leq2,
$
\begin{align*}
		\left\|
	H^{\frac{\alpha}{4}}e^{-itH}
	P_{\mathrm{ac}}(H)\chi_1(H)
	\right\|_{L^1\to L^\infty}
	\lesssim
	\langle t\rangle^{-\frac{2+\alpha}{4}}\bigl(\log(2+|t|)\bigr)^2.
	\end{align*}
      \end{theorem}

Building upon the $L^1\to L^\infty$ dispersive bounds already established for the regular and first-kind resonance cases in Subsection \ref{subsec:regular_first_sch_unweighted}, we only need to bound the kernel differences $(\mathcal{I}_{B}^+-\mathcal{I}_{B}^-)(t,x,y)$ defined in \eqref{def:I-B-alpha}, restricting our attention to the operators $B^\pm(\lambda)=\mathcal{M}_{j, l}^{\pm}(\lambda)$ with the index pairs $(j, l)$ satisfying either $j=4$ with $0 \le l \le 4$, or $0 \le l \le 4$ with $j=4$.  

Recall from Theorem \ref{thm:M_inverse} (II) that for $j = 1,2,3$:  
\[
\|\partial_\lambda^\ell \mathcal{M}_{4,j}^\pm(\lambda)\|_{\mathbb{B}(L^2)} + \|\partial_\lambda^\ell \mathcal{M}_{j,4}^\pm(\lambda)\|_{\mathbb{B}(L^2)} \lesssim \lambda^{1-\ell}|\log \lambda|^{4}, \quad \ell = 0,1,2.  
\]
Moreover, for pairs involving indices 0 and 4, we have:  
\begin{align}  
	\mathcal{M}_{0,4}^{\pm}(\lambda) &= (\log \lambda)^{-1} Q_0 \Lambda^\pm(\lambda) Q_{4}^0 + Q_0 \Lambda^\pm(\lambda) Q_{4}^1, \nonumber\\  
	\mathcal{M}_{4,0}^{\pm}(\lambda) &= (\log \lambda)^{-1} Q_{4}^0 \Lambda^\pm(\lambda) Q_0 + Q_{4}^1 \Lambda^\pm(\lambda) Q_0, \nonumber\\ 
	\label{M44-1}
	\mathcal{M}_{4,4}^{\pm}(\lambda) &= Q_{4} \Lambda^\pm(\lambda) Q_{4} + (\log\lambda)\bigl(Q_{4} \Lambda^\pm(\lambda) Q_{4}^1+Q_{4}^1 \Lambda^\pm(\lambda) Q_{4}\bigr) + (\log\lambda)^2 Q_{4}^1 \Lambda^\pm(\lambda) Q_{4}^1.
\end{align} 
Throughout the paper, $\Lambda^\pm(\lambda)$ denotes  a generic operator in $\mathbb{B}(L^2)$, possibly different at each occurrence, satisfying
$
\|\partial_\lambda^\ell \Lambda^\pm(\lambda)\|_{\mathbb{B}(L^2)} \lesssim \lambda^{-\ell}, \  \ell=0,1,2.
$

We categorize the  operators $B^\pm(\lambda)$ into two classes: the operators $\mathcal{M}_{j,4}^\pm, \mathcal{M}_{4,j}^\pm$ for $0\leq j\leq3$ (handled in Proposition \ref{prop_second_2_3}) and the operator $\mathcal{M}_{4,4}^\pm(\lambda)$ which generates the slowest decay (handled in Proposition \ref{prop_second_1}).

\begin{proposition}\label{prop_second_2_3}
	Let $B^\pm(\lambda) \in \{\mathcal{M}_{4,j}^{\pm}(\lambda), \mathcal{M}_{j,4}^{\pm}(\lambda)\}$ with $j\in \{0,1,2,3\}$. Then for every $-2<\alpha\leq2,$
\[	\sup_{x,y\in\mathbb R^2}
\left|
\mathcal I_{B}^\pm(\alpha; t,x,y)
\right|| \lesssim	\langle t\rangle^{-\frac{2+\alpha}{4}}.  
	\]  
\end{proposition}

\begin{proof}
	By Lemma \ref{lemma_projection}, the projected free resolvents relevant
	to the present proposition satisfy
	\begin{align}
		\bigl(Q_\beta vR_0^\pm(\lambda^4)\bigr)(\cdot,z)
		&=
		\lambda^{-1-\delta_{\beta0}}
		\Omega_{\beta,\pm}(\lambda,\cdot,z),
		\label{Qalpha}\\
		\bigl(Q_4vR_0^\pm(\lambda^4)\bigr)(\cdot,z)
		&=
		\mathcal{J}_{4,\pm}(\lambda,\cdot,z),
		\label{De all Q}\\
		\bigl(Q_4^1vR_0^\pm(\lambda^4)\bigr)(\cdot,z)
		&=
		(b_0Q_4^1vG_2^0)(\cdot,z)
		+
		\mathcal{T}_4^1(\lambda,\cdot,z)
		+
		\mathcal{T}_{4,\pm}^1(\lambda,\cdot,z),
		\label{eq:resolvent_decomp-Q4}
	\end{align}
	Moreover, for $\ell=0,1,2$,
	\begin{align}
		&
		\left\|
		\partial_\lambda^\ell
		\Bigl(
		e^{\mp i\lambda|z|}
		\Omega_{\beta,\pm}(\lambda,\cdot,z)
		\Bigr)
		\right\|_{L^2}
		+
		|\log\lambda|^{-1}
		\left\|
		\partial_\lambda^\ell
		\Bigl(
		e^{\mp i\lambda|z|}
		\mathcal{J}_{4,\pm}(\lambda,\cdot,z)
		\Bigr)
		\right\|_{L^2}
		\nonumber\\
		&\qquad
		+
		\left\|
		\partial_\lambda^\ell
		\Bigl(
		e^{\mp i\lambda|z|}
		\mathcal{T}_{4,\pm}^1(\lambda,\cdot,z)
		\Bigr)
		\right\|_{L^2}
		\lesssim
		\lambda^{-\ell}
		\langle\lambda z\rangle^{-1/2},
		\label{eq:E_bound-1}
	\end{align}
	while
	\(
	\|(b_0Q_4^1vG_2^0)(\cdot,z)\|_{L^2}
	\lesssim1,
	\left\|
	\partial_\lambda^\ell
	\mathcal{T}_4^1(\lambda,\cdot,z)
	\right\|_{L^2}
	\lesssim
	\lambda^{-\ell}.
	\)

	\noindent
	\textbf{Case 1: $B^\pm=\mathcal{M}_{0,4}^\pm$ or
		$\mathcal{M}_{4,0}^\pm$.}
	We consider $\mathcal{M}_{0,4}^\pm$  (the symmetric case $\mathcal{M}_{4,0}^{\pm}$ follows identically). Write
	$$
	\mathcal{M}_{0,4}^\pm(\lambda)
	=
	B_0^\pm(\lambda)+B_1^\pm(\lambda),
	$$
	where
	$
	B_0^\pm(\lambda)
	=
	(\log\lambda)^{-1}
	Q_0\Lambda^\pm(\lambda)Q_4^0$
	and $
	B_1^\pm(\lambda)
	=
	Q_0\Lambda^\pm(\lambda)Q_4^1.
	$ 	By linearity, it suffices to estimate the contributions of
	$B_0^\pm$ and $B_1^\pm$ separately.
	
	We first treat $B_0^\pm$. Since
	$
	4-k_0-k_4=2,
	$
	substituting \eqref{Qalpha} and \eqref{De all Q} into  \eqref{def:I-B-alpha} gives
	\begin{align}
		\mathcal{I}_{B_0}^\pm(\alpha; t,x,y)
		={}&
		\frac{2}{\pi i}
		\int_0^\infty
	e^{-it\lambda^4}
		e^{\pm i\lambda r}
		\lambda
		\mathcal{E}^\pm(\alpha; \lambda,x,y)
		\,d\lambda,
		\label{eq:K-B0-standard}
	\end{align}
	where $r:=r(x,y)=|x|+|y|$ and
	\begin{align}
		\mathcal{E}^\pm(\alpha; \lambda,x,y)
		={}&
		\widetilde{\chi}_1(\lambda)\lambda^{\alpha}
		(\log\lambda)^{-1}
		\Bigl\langle
		Q_0\Lambda^\pm(\lambda)Q_4^0
		\Bigl(
		e^{\mp i\lambda|y|}
		\mathcal{J}_{4,\pm}(\lambda,\cdot,y)
		\Bigr),
		e^{\pm i\lambda|x|}
		\Omega_{0,\mp}(\lambda,\cdot,x)
		\Bigr\rangle.
		\label{eq:E-B0}
	\end{align}
	Indeed, the factor $\lambda^2$ arising from
	$\lambda^{4-k_0-k_4}$ exactly cancels the factor $\lambda^{-2}$
	in the $Q_0$-resolvent expansion.
		Since $\|\partial_\lambda^\ell \Lambda^\pm(\lambda)\|_{\mathbb{B}(L^2)} \lesssim \lambda^{-\ell}$ for $\ell=0,1,2$ and the factor $(\log\lambda)^{-1}$ exactly compensates for the logarithmic growth of $\mathcal{J}_{4,\pm}$ in \eqref{eq:E-B0}, it follows from Leibniz's rule, H\"older's inequality, and \eqref{eq:E_bound-1} that
	\begin{equation*}
		\left|
		\partial_\lambda^\ell
		\mathcal{E}^\pm(\alpha; \lambda,x,y)
		\right|
		\lesssim
		\widetilde{\chi}_1(\lambda/2)
		\langle\lambda(|x|+|y|)\rangle^{-1/2}
		\lambda^{\alpha-\ell},
		\qquad
		\ell=0,1,2.
	\end{equation*} 
	Lemma
	  \ref{quartic-oscillatory}(ii) with $(\sigma,\nu)=(\alpha,0)$ therefore yields
	the desired $\langle t\rangle^{-\frac{2+\alpha}{4}}$ bound for 
	$\mathcal{I}_{B_0}^\pm$.

We next consider \(B_1^\pm(\lambda)=Q_0\Lambda^\pm(\lambda)Q_4^1\).
By \eqref{eq:resolvent_decomp-Q4}, the kernel
\(\mathcal I_{B_1}^\pm\) decomposes into three contributions,
corresponding respectively to
$
b_0Q_4^1vG_2^0,
\mathcal T_4^1$ and
$\mathcal T_{4,\pm}^1.
$
Each contribution can be written in the  same canonical form as \eqref{eq:K-B0-standard}
with the phase \(r(x,y)\) and the amplitude factor 
\(\mathcal E^\pm\) specified below.

For the contribution of \(b_0Q_4^1vG_2^0\), we take
\(r(x,y)=|x|\) and
\[
\mathcal E^\pm(\alpha;\lambda,x,y)
=
\lambda^\alpha\widetilde\chi_1(\lambda)
\Big\langle
Q_0\Lambda^\pm(\lambda)Q_4^1
(b_0Q_4^1vG_2^0)(\cdot,y),
e^{\pm i\lambda|x|}
\Omega_{0,\mp}(\lambda,\cdot,x)
\Big\rangle .
\]
The contribution of \(\mathcal T_4^1\) has the same phase
\(r(x,y)=|x|\), with the amplitude factor
\[
\mathcal E^\pm(\alpha;\lambda,x,y)
=
\lambda^\alpha\widetilde\chi_1(\lambda)
\Big\langle
Q_0\Lambda^\pm(\lambda)Q_4^1
\mathcal T_4^1(\lambda,\cdot,y),
e^{\pm i\lambda|x|}
\Omega_{0,\mp}(\lambda,\cdot,x)
\Big\rangle .
\]
Finally, the contribution of \(\mathcal T_{4,\pm}^1\) has phase
\(r(x,y)=|x|+|y|\), with the  amplitude factor
\begin{align*}
	\mathcal E^\pm(\alpha;\lambda,x,y)
	=
	\lambda^\alpha\widetilde\chi_1(\lambda)
	\Big\langle
	Q_0\Lambda^\pm(\lambda)Q_4^1
	\Big(
	e^{\mp i\lambda|y|}
	\mathcal T_{4,\pm}^1(\lambda,\cdot,y)
	\Big),
	e^{\pm i\lambda|x|}
	\Omega_{0,\mp}(\lambda,\cdot,x)
	\Big\rangle .
\end{align*}
By \eqref{eq:E_bound-1}, the derivative bounds for
\(\Lambda^\pm(\lambda)\), Leibniz's rule, and H\"older's  inequality, each  
$
\mathcal E^\pm
$
satisfies
\[
\left|
\partial_\lambda^\ell
\mathcal E^\pm(\alpha;\lambda,x,y)
\right|
\lesssim
\lambda^{\alpha-\ell}
\widetilde\chi_1(\lambda/2)
\langle\lambda r(x,y)\rangle^{-1/2},
\qquad
\ell=0,1,2,
\]
where \(r(x,y)\) denotes the corresponding phase above. Then applying Lemma
\ref{quartic-oscillatory}(ii) with $(\sigma,\nu)=(\alpha,0)$ yields
the desired $\langle t\rangle^{-\frac{2+\alpha}{4}}$ bound for both
$\mathcal{I}_{B_1}^\pm$.

	Combining the estimates for $B_0^\pm$ and $B_1^\pm$ gives
	$$
	\sup_{x,y\in \mathbb{R}^2}
	\bigl|
\mathcal{I}_{\mathcal{M}_{0,4}}^\pm(t,x,y)
	\bigr|
	\lesssim
	\langle t\rangle^{-\frac{2+\alpha}{4}}.
	$$
	
	\noindent
	\textbf{Case 2: $B^\pm=\mathcal{M}_{j,4}^\pm$ or
		$\mathcal{M}_{4,j}^\pm$, $1\leq j\leq3$.}
	We consider $\mathcal{M}_{j,4}^\pm$ (the symmetric case $\mathcal{M}_{4,j}^{\pm}$ follows identically). Since
	$
	4-k_j-k_4=1
	$
	and
	$
	\bigl(
	Q_j vR_0^\mp(\lambda^4)
	\bigr)(\cdot,x)
	=
	\lambda^{-1}
	\Omega_{j,\mp}(\lambda,\cdot,x),
	$
	the explicit factor $\lambda$ arising from
	$\lambda^{4-k_j-k_4}$  in
	\eqref{def:I-B-alpha}   exactly cancels this
	$\lambda^{-1}$ singularity. Repeating the same calculation as in the derivation of \eqref{eq:K-B0-standard}, the kernel \(\mathcal{I}_{\mathcal{M}_{j,4}}^\pm\) takes the same form as \eqref{eq:K-B0-standard} with phase \(r(x,y)=|x|+|y|\) and amplitude factor
	\begin{align*}
		\mathcal{E}^\pm(\alpha; \lambda,x,y)
		={}&
		\widetilde{\chi}_1(\lambda)
		\Bigl\langle
		\mathcal{M}_{j,4}^\pm(\lambda)
		\Bigl(
		e^{\mp i\lambda|y|}
		\mathcal{J}_{4,\pm}(\lambda,\cdot,y)
		\Bigr),
		e^{\pm i\lambda|x|}
		\Omega_{j,\mp}(\lambda,\cdot,x)
		\Bigr\rangle.
	\end{align*}
	Recall that
	$
	\|\partial_\lambda^\ell \mathcal{M}_{j, 4}^{\pm}(\lambda)\|_{\mathbb{B}(L^2)} \lesssim \lambda^{1-\ell}|\log\lambda|^{4}$ for   $\ell=0,1,2.$ 
	Applying Leibniz's rule and H\"older's inequality alongside \eqref{eq:E_bound-1}, we see that the additional $\lambda$ factor from $ \mathcal{M}_{j, 4}^{\pm}(\lambda)$ absorbs the logarithmic singularities:
	\[  
	\bigl| \partial_\lambda^{\ell} \mathcal{E}^\pm(\alpha; \lambda,x,y) \bigr| \lesssim\lambda^{\alpha} \widetilde{\chi}_1(\lambda/2) \langle \lambda r \rangle^{-1/2} \lambda^{1-\ell}|\log\lambda|^5 \lesssim \widetilde{\chi}_1(\lambda/2) \langle \lambda r \rangle^{-1/2} \lambda^{\alpha-\ell}, \quad \ell = 0,1,2. 
	\]   
	Hence Lemma \ref{quartic-oscillatory}(ii) with $(\sigma,\nu) =(\alpha, 0)$ again gives
	$$
	\sup_{x,y\in\mathbb{R}^2}
	\bigl|
	\mathcal{I}_{\mathcal{M}_{j,4}}^\pm(t,x,y)
	\bigr|
	\lesssim
\langle t\rangle^{-\frac{2+\alpha}{4}}.
	$$
	
	Combining the two cases completes the proof.
\end{proof}

\begin{proposition}\label{prop_second_1}  
	Let $B^\pm(\lambda)= \mathcal{M}_{4,4}^{\pm}(\lambda)$. Then  for every $-2<\alpha\leq2,$
	\begin{equation}\label{eq:prop1_unweighted}
	\sup_{x,y\in\mathbb R^2}
	\left|
	\mathcal I_{B}^\pm(\alpha; t,x,y)
	\right|| \lesssim	\langle t\rangle^{-\frac{2+\alpha}{4}} \bigl(\log(2+|t|)\bigr)^2.
	\end{equation} 
\end{proposition}  
\begin{proof} 
	This proof relies on the following representations derived in Lemma \ref{lemma_projection}:
	\begin{equation}   \label{Q4andQ}
		\begin{split}
			\bigl(Q_4 v R_0^\pm(\lambda^4)\bigr)(\cdot,z) &= \mathcal{J}_{4,\pm}(\lambda, \cdot,z), \\
			\bigl(Q_{4}^1v R_0^\pm(\lambda^4)\bigr)(\cdot,z) &= (b_0 Q_{4}^1v G_2^0)(\cdot,z) + \mathcal{T}_{4,\pm}^{1}(\lambda, \cdot,z) + \mathcal{T}_{4}^{1}(\lambda, \cdot,z).
		\end{split}
	\end{equation}
	For $\ell=0,1,2$, these components satisfy the bounds:
	\begin{equation}\label{eq:E_bound-Q4andQ}
		\begin{split}
			\big\| \partial_\lambda^\ell \bigl( e^{\mp i\lambda|z|}\mathcal{J}_{4,\pm}(\lambda, \cdot, z) \bigr) \big\|_{L^2} &\lesssim \lambda^{-\ell}|\log\lambda| \langle \lambda z \rangle^{-1/2}, \quad \|(b_0 Q_{4}^1v G_2^0)(\cdot, z)\|_{L^2} \lesssim 1,\\
			\big\| \partial_\lambda^\ell \bigl( e^{\mp i\lambda|z|}\mathcal{T}_{4,\pm}^{1}(\lambda, \cdot, z) \bigr) \big\|_{L^2}  &\lesssim \lambda^{-\ell} \langle \lambda z \rangle^{-1/2} , \quad \big\| \partial_\lambda^\ell \mathcal{T}_{4}^{1}(\lambda, \cdot, z) \big\|_{L^2} \lesssim \lambda^{-\ell}.
		\end{split}
	\end{equation}
	Note that the expansion for $Q_4v R_0^\pm(\lambda^4)$ contains an $|\log\lambda|$ singularity, whereas the expansion for $Q_{4}^1v R_0^\pm(\lambda^4)$ does not. Considering  $4-k_4-k_4=0$ and  substituting  \eqref{Q4andQ}, \eqref{eq:M44_rearranged} and 
	\begin{align}\label{eq:M44_rearranged}
		\mathcal{M}_{4,4}^{\pm}(\lambda) = Q_{4} \Lambda^\pm(\lambda) Q_{4} + (\log\lambda)\bigl(Q_{4} \Lambda^\pm(\lambda) Q_{4}^1+Q_{4}^1 \Lambda^\pm(\lambda) Q_{4}\bigr) + (\log\lambda)^2 Q_{4}^1 \Lambda(\lambda) Q_{4}^1,
	\end{align}
	into \eqref{def:I-B-alpha}, the kernel $\mathcal{I}_B^\pm$  reduces to finite linear combinations of canonical oscillatory integrals of the form:
	\begin{equation}\label{eq:canonical_oscillatory_prop1}
		\frac{2}{\pi i}	\int_0^\infty \lambda e^{-i t\lambda^4} e^{\pm i\lambda r(x,y)} \mathcal{E}_{f,g}^\pm(\alpha; \lambda, x, y)d\lambda,
	\end{equation}
	where the amplitude factor $\mathcal{E}_{f,g}^\pm$ is defined by
	\begin{equation}\label{eq:amplitude_def_prop1}
		\mathcal{E}_{f,g}^\pm(\alpha; \lambda, x, y) = e^{\mp i\lambda r(x,y)} \lambda^\alpha{\chi}_1(\lambda) \big\langle M_{sub}^\pm(\lambda) f(\lambda, \cdot, y), \, g(\lambda, \cdot, x) \big\rangle.
	\end{equation}
	Here, $M_{sub}^\pm(\lambda)$ is a constituent operator from \eqref{eq:M44_rearranged}, while $f$ and $g$ are corresponding states selected from the expansions in \eqref{Q4andQ}. The phase function $r(x,y) \in \{0, |x|, |y|, |x|+|y|\}$ neutralizes the exponential phases embedded within $f$ and $g$.
	
	To uniformly bound \eqref{eq:canonical_oscillatory_prop1}, we apply H\"older's inequality to \eqref{eq:amplitude_def_prop1}, tracing the $\lambda$-logarithmic singularities inherited from each component:
	\begin{itemize}
		\item For $M_{sub}^\pm(\lambda) = Q_{4} \Lambda^\pm(\lambda) Q_{4}$, it acts between the states $f = \mathcal{J}_{4,\pm}$ and $g = \mathcal{J}_{4,\mp}$. The product of their individual $L^2$-bounds yields an amplitude factor constrained by $ |\log\lambda|^2$.
        \vskip0.1cm
		\item For $M_{sub}^\pm(\lambda) = (\log\lambda) Q_{4} \Lambda^\pm(\lambda) Q_{4}^1$ (the symmetric case $(\log\lambda) Q_{4}^1 \Lambda^\pm(\lambda)Q_{4}$ follows identically),
		the operator pairs $f\in\{b_0 Q_{4}^1 v G_2^0, \mathcal{T}_{4,\pm}^{1}, \mathcal{T}_{4}^{1}\}$ and $g = \mathcal{J}_{4,\mp}(\lambda, \cdot, x)$. The inherent logarithm in the operator multiplies the singularity of $g$, yielding a total $\lambda$-singularity of $ |\log\lambda|^2$.

            \vskip0.1cm
		\item For $M_{sub}^\pm(\lambda) = (\log\lambda)^2 Q_{4}^1 \Lambda(\lambda) Q_{4}^1$, this operator pairs with the bounded, $\log$-free $Q_{4}^1$-expansions on both sides, i.e., $f,g\in\{b_0 Q_{4}^1 v G_2^0, \mathcal{T}_{4,\pm}^{1}, \mathcal{T}_{4}^{1}\}.$  The $\lambda$-singularity arises entirely from the operator itself and is bounded by $|\log\lambda|^2$.
	\end{itemize}
	Thus, for all cross-terms, applying Leibniz's rule and H\"older's inequality to the amplitude factor yields
	\[
	\bigl| \partial_\lambda^{\ell} \mathcal{E}_{f,g}^\pm(\alpha;\lambda,x,y) \bigr| \lesssim \widetilde{\chi}_1(\lambda/2) \langle \lambda r(x,y) \rangle^{-1/2} \lambda^{\alpha-\ell} |\log\lambda|^2, \quad \ell = 0,1,2.
	\]
	An application of Lemma \ref{quartic-oscillatory}(ii) with parameters $(\sigma,\nu)=(\alpha,-2)$ then yields the uniform  bound  \eqref{eq:prop1_unweighted} for the integrals \eqref{eq:canonical_oscillatory_prop1}.
\end{proof}

\section{The third kind resonance and eigenvalue cases}\label{sec:third_fourth}

This section  is aim to show Theorem 
\ref{main_theorem-1} (iii) and (iv).
Together with   the high-energy estimates in
Theorem \ref{main-theorem-high-schrodinger}, it remains to prove the following
Theorems  \ref{main_theorem_low_31} and \ref{main_theorem_low_32}.

\begin{theorem}\label{main_theorem_low_31}
	Let \(H = \Delta^2 + V\) with \(|V(x)| \lesssim \langle x \rangle^{-18-}\).
	Assume that \(H\) has no positive embedded eigenvalues and that zero is either a third-kind resonance or an eigenvalue accompanied by a \(d\)-wave resonance (i.e., \(Q_5 \neq 0\)). Then, for \(0\le \alpha \le 2\) and $|t|\gg1,$
	\begin{equation}\label{optimality_31}
		\left\|
		H^{\frac{\alpha}{4}}e^{-itH}
		P_{\mathrm{ac}}(H)\chi_1(H)
		\right\|_{L^1\to L^\infty} \sim
		\begin{cases}
		\bigl(\log |t|\bigr)^{-1}, & \text{for}\  \alpha = 0,\\[6pt]
		|t|^{-\frac{\alpha}{4}}(\log|t|)^{-2}, & \text{for}\  0 < \alpha \le 2.
		\end{cases}
	\end{equation}
\end{theorem}

\begin{theorem}\label{main_theorem_low_32}
	Let $H = \Delta^2 + V$ with $|V(x)| \lesssim \langle x \rangle^{-18-}$.  
	Assume that $H$ has no positive embedded eigenvalues and that zero is an eigenvalue of $H$ but exhibits no d-wave resonance.  Then for every  $-2<\alpha\leq2.$
		$$  
		\begin{aligned}  
		\left\|
	H^{\frac{\alpha}{4}}e^{-itH}
	P_{\mathrm{ac}}(H)\chi_1(H)
	\right\|_{L^1\to L^\infty} \lesssim\langle t\rangle^{-\frac{2+\alpha}{4}}{(\log (2+|t|))^2}.  
		\end{aligned}  
	$$
\end{theorem}  
Combining  \eqref{Q5-d} with Definition \ref{definition1} and the definition of $d$-wave resonance, we have
\[
\mathbf{k}=3
\ \text{or}\ 
\mathbf{k}=4 \text{ with a \(d\)-wave resonance}
\quad\Longleftrightarrow\quad
\text{a \(d\)-wave resonance is present}
\quad\Longleftrightarrow\quad
Q_5\neq0.
\]
We therefore divide the proofs of
Theorems~\ref{main_theorem_low_31} and
\ref{main_theorem_low_32} into the following Subsections \ref{Q5neq0} and 
 \ref{Q5=0}, respectively.

 \subsection{The estimates for the case \(Q_5\neq0\)}\label{Q5neq0} This subsection is devoted to the proof of
Theorem~\ref{main_theorem_low_31}. 
 Following the strategy from the previous sections, we only need to estimate the integral kernels $\mathcal{I}_{B}^+-\mathcal{I}_{B}^-$ defined in \eqref{def:I-B-alpha}, restricted to the operators $B^\pm(\lambda)=\mathcal{M}_{j, l}^{\pm}(\lambda)$ where either $j \in \{5,6\}$ or $l \in \{5,6\}$.  

Recall from Theorem \ref{thm:M_inverse} (III) that
\[
	 \begin{aligned}
&\mathcal{M}_{j,l}^\pm(\lambda)= \mathcal{A}_{j,l}(\lambda)+(\log \lambda)^{-2}\Gamma_{j,l}^{ \pm}(\lambda), \ (j,l)=(5,5),(5,6),(6,5), (6,6).
	\end{aligned}
    \]
with
$$
\left\|\partial_\lambda^{\ell} \mathcal{A}_{j,l}(\lambda)\right\|_{\mathbb{B}\left(L^2\right)}
+\big\|\partial_\lambda^{\ell} \Gamma_{j, l}^{ \pm}(\lambda) \big\|_{\mathbb{B}\left(L^2\right)} \lesssim \lambda^{-\ell},
\ \ell=0,1,2.
$$
Moreover, for $ j=5,6, \ 0 \le l \le 4$ and $\ell=0,1,2,$
\begin{align*}
\big\|\partial_\lambda^{\ell} \mathcal{M}_{j,l}^\pm(\lambda)\big\|_{\mathbb{B}\left(L^2\right)}
+\big\|\partial_\lambda^{\ell} \mathcal{M}_{l,j}^\pm(\lambda)\big\|_{\mathbb{B}\left(L^2\right)}
\lesssim \lambda^{-\ell}.
\end{align*}
Therefore it suffices to analyze $\mathcal{I}_{B}^+-\mathcal{I}_{B}^-$  for the following cases:
\begin{itemize}
	\item $B^\pm(\lambda)=\mathcal{M}_{j,l}^\pm(\lambda) \ \text{with} \ (j,l)=(5,5),(5,6),(6,5),(6,6)$;
        \vskip0.1cm
	\item $ B^\pm(\lambda)= \mathcal{M}_{j, l}^\pm(\lambda) \ \text{with} \  (j,l)  \in \{ \mathbb{Z}^2 \mid j=5,6,\ 0 \le l \le 4 \ \text{or} \ l=5,6,\ 0 \le j \le 4\}.$
\end{itemize}
The proof of Theorem~\ref{main_theorem_low_31} is divided into
Propositions~\ref{third_merged_1}, \ref{third_merged_2}, and
\ref{sharp-3-4}. The first two propositions establish the required
upper bounds, while Proposition~\ref{sharp-3-4} proves the matching
lower bounds and hence the sharpness of the decay rates.

\begin{proposition}\label{third_merged_1} 
	Let $B^\pm(\lambda)=\mathcal{M}_{j,l}^{\pm}(\lambda)$ with $j,l\in \{5,6\}$. Then   for $0\leq\alpha\leq2,$   
	\[      
	\sup _{x, y \in \mathbb{R}^2}\left|\left(\mathcal{I}_{B}^+-\mathcal{I}_{B}^-\right)(\alpha;t,x,y)\right|  \lesssim 	\begin{cases}
		\bigl(\log(2+ |t|)\bigr)^{-1}, & \alpha = 0,\\[6pt]
		\bigl(\langle t\rangle^{\frac{\alpha}{4}} \log(2+ |t|)\bigr)^{-1}, & 0 < \alpha \le 2.
	\end{cases}    
	\]      
\end{proposition}      
\begin{proof}   
	For $j, l \in \{5,6\}$, the operators $B^\pm(\lambda)$ are decomposed into a $\pm$-independent principal term $\mathcal{A}_{j,l}(\lambda)$  and a $\pm$-dependent remainder: $(\log \lambda)^{-2}\Gamma_{j,l}^{ \pm}(\lambda)$. We analyze the integral kernel $\mathcal{I}_B^\pm$ generated by these two parts separately.  
	Recall from Lemma \ref{lemma_projection} (iii) that  
	\begin{align}\label{ex-Q56}  
		\bigl(Q_\beta v R_0^\pm(\lambda^4)\bigr)(\cdot,z) &= (b_0 Q_\beta v G_2^0)( \cdot,z) +\lambda^{m_\beta}\bigl(\mathcal{T}_{\beta}(\lambda, \cdot,z)+\mathcal{T}_{\beta,\pm}(\lambda, \cdot,z)\bigr), \quad \beta=5,6,
	\end{align}  
	where $m_5=1$ and $m_6=2$. 	By \eqref{lemma_projection_G}--\eqref{lemma_projection_1},  for $j=5,6,$  one has $\|Q_\beta vG_2^0(\cdot, z)\|_{L^2}\lesssim 1$ and
	\begin{equation*}
	\begin{split}       
	\big\| \partial_\lambda^\ell \bigl( e^{\mp i\lambda|z|}\mathcal{T}_{\beta,\pm}(\lambda, \cdot, z) \bigr) \big\|_{L^2}& \lesssim \lambda^{-\ell} \langle \lambda z \rangle^{-1/2}, \ \ 
	\big\| \partial_\lambda^\ell \mathcal{T}_\beta(\lambda, \cdot, z) \big\|_{L^2} \lesssim \lambda^{-\ell}, \ \   \ell=0,1,2.
	\end{split}    
	\end{equation*}        
	By \eqref{ex-Q56}, the kernel $\mathcal{I}_B^\pm$  defined in \eqref{def:I-B-alpha} is decomposed into finite sums of canonical oscillatory integrals of the form:
	\begin{align}\label{3-4res-integral}
	\frac{2}{\pi i}\int_0^\infty \lambda e^{-it\lambda^4}e^{\pm i\lambda r(x,y)} \mathcal{E}_{fg}^\pm(\alpha; \lambda, x, y) \, d\lambda,  
	\end{align}
	where the  amplitude factor $	\mathcal{E}_{fg}^\pm(\alpha; \lambda, x, y) $ is defined by  
\begin{equation}\label{eq:amplitude_def}  
		\mathcal{E}_{fg}^\pm(\alpha;\lambda, x, y) = e^{\mp i\lambda r(x,y)} \lambda^{\gamma_\alpha} \widetilde{\chi}_1(\lambda) \big\langle B^\pm(\lambda) f(\lambda, \cdot, y), g(\lambda, \cdot, x) \big\rangle.    
	\end{equation}    
	Here, $f, g \in\{b_0 Q_\beta v G_2^0, \mathcal{T}_{\beta,\pm}, \mathcal{T}_{\beta}\}$ for $\beta\in\{5,6\}$. Moreover, $r(x,y)$ isolates the inherent spatial oscillations of $f$ and $g$: it contributes $|y|$ (resp. $|x|$) if the state is $\mathcal{T}_{\beta,\pm}(\lambda, \cdot, y)$ (resp. $\mathcal{T}_{\beta,\mp}(\lambda, \cdot, x)$), and contributes $0$ otherwise.   The exponent $\gamma$ is given by $\gamma_\alpha =\alpha+ 4 - k_j - k_l + p=\alpha-2+p$ for $j,l \in \{5,6\}$, where $p \ge 0$ is the total $\lambda$-power contributed by $f$ and $g$. For example if $(f,g)=(b_0Q_l v G_2^0, \mathcal{T}_{j,\mp}),$  then $r=|x|$, $p=m_j.$
	
	\vspace{1.5mm}      
\textbf{Case 1: The $\pm$-independent principal terms $\mathcal{A}_{j,l}(\lambda)$.} \\  
	For these terms, $B^+(\lambda)=B^-(\lambda) := B(\lambda)$. The kernel difference $\mathcal{I}_B^+-\mathcal{I}_B^-$ can be  
    rewritten as a sum of terms involving
    the resolvent difference:  
	\begin{align*} 
	&\bigl[R_0^+ (\lambda^4)v Q_j B(\lambda) Q_l v R_0^+(\lambda^4)\bigr](x,y) - \bigl[R_0^- (\lambda^4)v Q_j B(\lambda) Q_l v R_0^-(\lambda^4)\bigr](x,y)  \\  
	&= \bigl[(R_0^+(\lambda^4)-R_0^-(\lambda^4))v  Q_j B(\lambda) Q_lv R_0^+(\lambda^4) \bigr](x,y)+\big[ R_0^-(\lambda^4) v  Q_j B(\lambda) Q_l v (R_0^+(\lambda^4)-R_0^-(\lambda^4))\bigr](x,y).  
	\end{align*}  
	By \eqref{ex-Q56}, we obtain that for $\beta\in\{5,6\},$
		\begin{align}\label{difference5,6}
			\bigl[Q_\beta v\bigl(R_0^+(\lambda^4)-R_0^-(\lambda^4)\bigr)\bigr](\cdot, z)&=\lambda^{m_\beta}\bigl(\mathcal{T}_{\beta,+}(\lambda, \cdot,z)-\mathcal{T}_{\beta,-}(\lambda,\cdot,z)\bigr).
	\end{align}  
 Consequently, in the amplitude factor \eqref{eq:amplitude_def}, at least one state in the pairing $(f,g)$ must be a $\mathcal{T}_{\beta,\pm}$ remainder with $\beta\in\{5,6\}$, which guarantees a  factor $\lambda^{m_\beta}$ ($m_\beta\ge 1$).   
	Thus, we deduce that $p \ge 1$, yielding $\gamma_\alpha \ge \alpha-1$.  
	
	The most singular amplitude factor \eqref{eq:amplitude_def} occurs  when $p=1$. Since $m_5=1$ and $m_6=2$, this requires $\mathcal{T}_{5,+}-\mathcal{T}_{5,-}$ to be paired with $b_0 Q_\beta v G_2^0$  ($\beta\in\{5,6\}$)  (e.g., $f = \mathcal{T}_{5,\pm}$, $g = b_0 Q_\beta v G_2^0$, yielding $r(x,y) = |y|$ and $B(\lambda)=\mathcal{A}_{5,5}, \mathcal{A}_{6,5} $). Such amplitude factor satisfies:  
	\begin{align*}   
	\bigl|\partial_\lambda^{\ell} \mathcal{E}_{fg}^\pm(\alpha; \lambda,x,y) \bigr| \lesssim \langle \lambda r(x,y)\rangle^{-1/2} \lambda^{\alpha-1-\ell} \widetilde{\chi}_1 (\lambda/2), \quad \ell=0,1,2.      
	\end{align*}    
	Applying Lemma \ref{quartic-oscillatory}(ii) with parameters $(\sigma,\nu)=(\alpha-1,0)$  yields the bound $\langle t\rangle^{-\frac{\alpha+1}{4}}$, which we require $-2<\alpha-1\leq2,$ i.e.,  $-1<\alpha\leq3.$
	
	In particular, when we take $B(\lambda)=\mathcal{A}_{6,6}(\lambda)$, the expansion \eqref{difference5,6} guarantees a factor $\lambda^2$ (note that $m_6=2$), yielding $\gamma \ge \alpha.$ Hence, we  apply Lemma \ref{quartic-oscillatory}(ii) with  $(\sigma,\nu)=(\alpha,0)$  ($-2<\alpha\leq2$) to obtain that 
	\begin{equation}\label{mathcalA66}
	\sup _{x, y \in \mathbb{R}^2}\left|\left(\mathcal{I}_{B}^+-\mathcal{I}_{B}^-\right)(\alpha; t,x,y)\right|  \lesssim \langle t\rangle^{-\frac{\alpha+2}{4}}, \ \text{with} \ B^\pm(\lambda)=\mathcal{A}_{6,6}(\lambda).
	\end{equation}
	
	\vspace{1.5mm}      
	\textbf{Case 2: The $\pm$-dependent remainder terms $(\log\lambda)^{-2}\Gamma_{j,l}^\pm(\lambda)$.} \\  
	Since these operators depend on $\pm$, the algebraic identity \eqref{difference5,6} is inapplicable.  Thus we evaluate the kernels directly using the expansion \eqref{ex-Q56}. The dominant contribution to the decay rate arises from $f = b_0 Q_lv G_2^0$ and $g = b_0 Q_jv G_2^0$.  
	In this scenario, no $\lambda$-powers are gained ($p=0$, hence $\gamma_\alpha = \alpha-2$), and  $r(x,y)=0$. Incorporating the operator norm $\|\partial_\lambda^\ell \Gamma_{j,l}^\pm\| \lesssim \lambda^{-\ell}$, the corresponding  amplitude factor obeys:  
	\[    
	|\partial_\lambda^\ell \mathcal{E}_{fg}^\pm(\alpha; \lambda, x, y)| \lesssim |\log\lambda|^{-2}\lambda^{\alpha-2-\ell} \widetilde{\chi}_1(\lambda/2), \quad \ell=0,1,2.    
	\]    
	Applying Lemma \ref{quartic-oscillatory}(ii) with  $(\sigma,\nu)=(\alpha-2,2)$  dictates the slowest  decay rate
	\begin{equation} \label{Ijl}    
		\begin{cases}
		\bigl(\log(2+ |t|)\bigr)^{-1}, & \alpha = 0,\\[6pt]
		\bigl(\langle t\rangle^{\frac{\alpha}{4}} (\log(2+ |t|))^2\bigr)^{-1}, & 0 < \alpha < 4.
	\end{cases}    
	\end{equation} 
	All other cross-pairings contain remainder powers ($p \ge 1$), which yield the bound  $\langle t\rangle^{-\frac{\alpha+1}{4}}$  ($-1<\alpha\leq3$) by  Lemma \ref{quartic-oscillatory}(ii) with parameters $(\sigma,\nu)=(\alpha-1,0).$   
	
	 Summing the bounds from Part 1 and Part 2 concludes the proof. In particular, we derive
for	$B^\pm(\lambda)=\mathcal{M}_{j, l}^{\pm}(\lambda)$ with $j,l\in \{5,6\}$ and $0\leq\alpha\leq2,$
	 \[
	 	\begin{aligned}
	 \left(\mathcal{I}_{B}^+-\mathcal{I}_{B}^-\right)(\alpha; t,x,y)&=\frac{2}{\pi i} I_{j,l}(\alpha; t,x,y)+O\bigl(\langle t\rangle^{-\frac{\alpha+1}{4}}\bigr),
  \end{aligned}
     \]
	 where 
	 \begin{equation}\label{kn-sharp}
	 \begin{split}
	 	I_{j,l}(\alpha; t,x,y)&=b_0^2\int_0^\infty  e^{-it\lambda^4} \lambda^{\alpha-1}(\log\lambda)^{-2}\widetilde{\chi}_1(\lambda) \big\langle \Gamma_{j,l}^\pm(\lambda) (Q_l v G_2^0)( \cdot, y), (Q_j v G_2^0)( \cdot, x) \big\rangle d\lambda,
	 	\end{split}
	 \end{equation}
	 satisfying the uniform estimate \eqref{Ijl}.
\end{proof}

\begin{proposition}\label{third_merged_2} 
	Let $B^\pm(\lambda)=\mathcal{M}_{j,l}^{\pm}(\lambda), \mathcal{M}_{l,j}^{\pm}(\lambda)$ with $j=5,6$, $0\le l\le4$. Then for every $0\leq\alpha\leq2,$
	\[
\sup_{x,y\in\mathbb{R}^2}\left|\mathcal{I}_{B}^\pm(\alpha; t,x,y)\right|
	\lesssim\langle t\rangle^{-\frac{\alpha+1}{4}} \log(2+|t|).
	\]
\end{proposition}
\begin{proof}
The proof is the same as the argument in Proposition
\ref{third_merged_1}. Indeed, using \eqref{ex-Q56} for the
$Q_5$- and $Q_6$-projected free resolvents, together with the
corresponding estimates for the $Q_l$-projected resolvent
($0\leq l \leq4$) (see Lemma \ref{lemma_projection} (i)-(ii)), one reduces the kernel $\mathcal{I}_{B}^\pm$ to the canonical
oscillatory integrals in \eqref{3-4res-integral}. The resulting
amplitude factors satisfy the bounds required by Lemma
\ref{quartic-oscillatory}(ii) with $(\sigma,\nu)=(\alpha-1,0)$ when
$0\leq l\leq3$, and with $(\sigma,\nu)=(\alpha-1,-1)$ when
$l=4$, the latter logarithmic loss coming from
$\mathcal J_{4,\pm}$. Hence the desired
$\langle t\rangle^{-\frac{\alpha+1}{4}}\log(2+|t|)$ bound follows. The symmetric case is treated in the same way.
\end{proof}

	It remains to establish the sharpness of the bound \eqref{optimality_31}.
	Synthesizing the preceding analysis, it is evident that for $|t|\gg 1$, 
the decay rate is determined by the kernels ${I}_{j,l}(t,x,y)$ defined in \eqref{kn-sharp}.
	To derive the asymptotic expansion, we utilize \cite[(9.25)-(9.28) in Section 9]{CCWY}  to decompose the operator $\Gamma_{j,l}^{\pm}(\lambda)$:    
	\[    
	\Gamma_{j,l}^{\pm}(\lambda) = \mathcal{L}_{j,l}^\pm + (-\log \lambda)^{-1/2} \Lambda^\pm(\lambda), 
	\]    
	where the leading terms are given by  
	\begin{equation}\label{mathcal_L} 
		\begin{aligned}  
			\mathcal{L}_{5,5}^{\pm} &= -a_3^\pm b_1^{-2}(Q_5 v G_6 v Q_5)^{-1}, \ \  
			&&\mathcal{L}_{5,6}^{\pm} = -a_3^\pm \mathcal{D}_{5,5} (Q_5 v G_6^0 v Q_6) \mathcal{D}_{6,6}, \\  
			\mathcal{L}_{6,5}^{\pm} &= -a_3^\pm \mathcal{D}_{6,6} (Q_6 v G_6^0 v Q_5) \mathcal{D}_{5,5}, \ \  
			&&\mathcal{L}_{6,6}^{\pm} = a_3^\pm b_1 \mathcal{D}_{6,6} (Q_6 v G_6^0 v Q_5) \mathcal{D}_{5,5} (Q_5 v G_6^0 v Q_6) \mathcal{D}_{6,6},  
		\end{aligned}  
	\end{equation}  
	with $\mathcal{D}_{5,5} = -(b_1 Q_5 v G_6 v Q_5)^{-1}$ and $\mathcal{D}_{6,6} = (b_1 Q_6 v G_6^0 v Q_6)^{-1}$.
	
	This decomposition isolates the dominant singularity:    
	\begin{itemize}
	\item Replacing $\Gamma_{j,l}^{\pm}(\lambda)$ with $(-\log\lambda)^{-1/2}\Lambda^{\pm}(\lambda)$ in \eqref{kn-sharp} yields an improved decay rate of $(\log(2+|t|))^{-3/2}$ for $\alpha=0$ and $\bigl(\langle t\rangle^{\frac{\alpha}{4}} (\log(2+ |t|))^{5/2}\bigr)^{-1}$ for $0<\alpha\leq2$, obtained via Lemma~\ref{quartic-oscillatory} (ii) with $(\sigma, \nu) = (\alpha-2, 5/2)$.
        \vskip0.1cm
	\item The term $\mathcal{L}_{j,l}^{\pm}$ produces the slowest decay \eqref{Ijl}, given by Lemma~\ref{quartic-oscillatory} (ii) with $(\sigma, \nu) = (\alpha-2, 2)$.
\end{itemize} 
Hence,  we deduce the asymptotic expansion in $\mathbb{B}(L^1, L^\infty)$:      
\[      
	\begin{aligned}      
		H^{\frac{\alpha}{4}}e^{-itH} P_{\mathrm{ac}}(H)\chi_1(H) &= -\frac{2}{\pi i}\sum_{j,l = 5}^{6} T_{\mathcal{I}_{j,l}} + 	\begin{cases}
		O\bigl(\log(2+ |t|)\bigr)^{-3/2}, & \alpha = 0,\\[6pt]
			O\bigl(\langle t\rangle^{\frac{\alpha}{4}} (\log(2+ |t|))^{5/2}\bigr)^{-1}, & 0 < \alpha \leq 2,
		\end{cases} 
	\end{aligned}      
\]    
where each $T_{\mathcal{I}_{j,l}}$  is the integral operator defined by the respective kernel:    
\[
\begin{aligned}      
	\mathcal{I}_{j,l}(\alpha; t,x,y) &= c_{j,l}(x,y)\int_0^{\infty} e^{-it\lambda^4} \lambda^{\alpha-1} \widetilde{\chi}_1(\lambda) (\log \lambda)^{-2} \, d\lambda,   
\end{aligned}
\]
with  $c_{j,l}(x,y) = b_0^2 \big\langle (\mathcal{L}_{j,l}^+ -\mathcal{L}_{j,l}^-) (Q_l v G_2^0)(\cdot, y), (Q_j v G_2^0)(\cdot, x) \big\rangle$.   
Here the summation over $j, l \in \{5,6\}$ accommodates the third-kind resonance scenario, where $Q_6 = 0$ implies $c_{j,l} \equiv 0$ for all pairs $(j,l) \neq (5,5)$.  
  
Thus, establishing the optimal bounds in \eqref{optimality_31} reduces to proving the following Proposition \ref{sharp-3-4}.

\begin{proposition}\label{sharp-3-4}  
	For $|t| \gg 1$, the leading operators satisfy the lower bounds:  
\begin{equation}\label{low_bound_weighted_1}  
		\Big\|\sum_{j,l = 5}^{6}T_{ \mathcal{I}_{j,l}}\Big\|_{L^1 \to L^\infty} \gtrsim \begin{cases}
			\bigl(\log(2+ |t|)\bigr)^{-1}, & \alpha = 0,\\[6pt]
			\bigl(\langle t\rangle^{\frac{\alpha}{4}} (\log(2+ |t|))^{2}\bigr)^{-1}, & 0 < \alpha \leq2.
		\end{cases} 
	\end{equation}  
\end{proposition}  

\begin{proof}
	 By continuous embeddings $L^{2,s}(\mathbb{R}^2) \hookrightarrow L^1(\mathbb{R}^2)$ and $L^\infty(\mathbb{R}^2) \hookrightarrow L^{2,-s}(\mathbb{R}^2)$ for $s > 1$, it is  sufficient to construct a specific test function $\varphi \in L^{2,s}(\mathbb{R}^2)$ such that:  
\begin{equation}\label{eq:lower_bounds}  
		\Big|\Big\langle \sum_{j,l= 5}^{6} T_{\mathcal{I}_{j,l}} \varphi, \varphi \Big\rangle \Big| \gtrsim  \begin{cases}
			\bigl(\log(2+ |t|)\bigr)^{-1}, & \alpha = 0,\\[6pt]
			\bigl(\langle t\rangle^{\frac{\alpha}{4}} (\log(2+ |t|))^{2}\bigr)^{-1}, & 0 < \alpha \leq2.
			\end{cases}
	\end{equation}  
	
	Under the hypothesis that zero is a third-kind resonance or an eigenvalue accompanied by a $d$-wave resonance, the subspace $Q_5 L^2\neq\{0\}$. We fix  $0\neq\psi \in Q_5 L^2$ and explicitly construct our test function as $\varphi := v U \psi = -b_0 |V|G_2^0 v\psi$. 
	 Since
	$
	\langle y^\beta v,\psi\rangle=0
	$ (multi-index $|\beta|\leq2$), 
	a third-order Taylor expansion of
	$
	G_2^0(x,y)=|x-y|^2\log|x-y|
	$
	in the \(y\)-variable yields
	\[
	\bigl|(G_2^0v\psi)(x)\bigr|
	\lesssim
	\langle x\rangle^{-1}
	\bigl(1+\log\langle x\rangle\bigr),
	\qquad x\in\mathbb R^2.
	\]
	 Hence,
	 the estimate $|\varphi(x)| \lesssim \langle x \rangle^{-\mu-1}(1+\log\langle x\rangle)$ guarantees that $\varphi \in L^{2,s}(\mathbb{R}^2)$.

	 Since $Q_5 T_0 = 0$, we have $Q_5(U + b_0 v G_2^0 v)U \psi = 0$, which dictates that  
	\[  
	b_0 Q_5 v G_2^0 \varphi = b_0 Q_5 v G_2^0 v U \psi = -Q_5 \psi = -\psi.  
	\]  
Furthermore, by identical logic, using $Q_6T_0=0$ and the orthogonality  $Q_6 Q_5=0$, we have
	$$
	b_0 Q_6 v G_2^0 \varphi=b_0 Q_6 v G_2^0 v U \psi =  -Q_6 \psi=-Q_6Q_5 \psi=0.
	$$
 Consequently, $\langle c_{j,l} \varphi, \varphi \rangle=b_0^2\big\langle (\mathcal{L}_{j,l}^+ -\mathcal{L}_{j,l}^-) (Q_l v G_2^0)\varphi, (Q_j v G_2^0)\varphi \big\rangle=0$
  for all $(j,l) \neq (5,5)$.  
	The only surviving term corresponds to $(j,l)=(5,5)$. Utilizing the invertibility of $\mathcal{L}_{5,5}^\pm = -a_3^\pm b_1^{-2} (Q_5 v G_6 v Q_5)^{-1}$ on $Q_5L^2$ (defined in \eqref{mathcal_L}) and the fact that $\text{Im}(a_3^\pm) \neq 0$, we obtain   
	\begin{align*}  
		\langle c_{5,5} \varphi, \varphi \rangle &= b_0^2 \big\langle (\mathcal{L}_{5,5}^+ -\mathcal{L}_{5,5}^-) Q_5 v G_2^0 \varphi, Q_5 v G_2^0 \varphi \big\rangle \\  
		&= -(a_3^+ - a_3^-) b_1^{-2}\big\langle (Q_5 v G_6 v Q_5)^{-1} \psi, \psi \big\rangle := c_0 \neq 0.  
	\end{align*}  
Hence, 	
	\begin{equation}\label{T_ijl}
		\Big|\Big\langle \sum_{j,l= 5}^{6} T_{\mathcal{I}_{j,l}} \varphi, \varphi \Big\rangle \Big| =\Big|c_0\int_0^{\infty} e^{-it\lambda^4} \lambda^{\alpha-1} \widetilde{\chi}_1(\lambda) (\log \lambda)^{-2} \, d\lambda\Big|.
	\end{equation}  
	For $\alpha=0,$ by Remark~\ref{remark:osci}, this bound $	\bigl(\log(2+ |t|)\bigr)^{-1}$  comes from low-frequency region $|t|\lambda^4 \ll 1$, while high-frequency region $|t|\lambda^4 \gtrsim 1$ yields the faster $(\log(2+|t|))^{-2}$ decay. Hence, it follows that
		\begin{align*}
		\Big|\Big\langle \sum_{j,l= 5}^{6} T_{\mathcal{I}_{j,l}} \varphi, \varphi \Big\rangle \Big| &\gtrsim\Big|\int_0^{\infty} e^{-it\lambda^4} \lambda^{-1} \widetilde{\chi}_1(\lambda){\chi}_1(|t|\lambda^4) (\log \lambda)^{-2} \, d\lambda\Big|\\
		&\gtrsim\Big|\int_0^{\infty} \cos(t\lambda^4) \lambda^{-1} \widetilde{\chi}_1(\lambda){\chi}_1(|t|\lambda^4) (\log \lambda)^{-2} \, d\lambda\Big|.
	\end{align*}  
Evaluating the temporal integrals for $|t| \gg 1$, we explicitly extract the logarithmic decay:  
	\begin{align}  \label{alpha=0}
	\Big|\int_0^{\infty} \cos(t\lambda^4) \lambda^{-1} \widetilde{\chi}_1(\lambda){\chi}_1(|t|\lambda^4) (\log \lambda)^{-2} \, d\lambda\Big| &\ge \cos 1 \int_0^{|t|^{-1/4} \lambda_0^{1/4}} \lambda^{-1} (\log \lambda)^{-2} \, d\lambda \gtrsim \frac{1}{\log |t|}.  
	\end{align}  

For \(0<\alpha\leq2\), set \(T=|t|\). By the change of variables
\(u=T\lambda^4\) in \eqref{T_ijl}, we obtain
\begin{align}
\Big|
\Big\langle
\sum_{j,l=5}^{6}T_{\mathcal I_{j,l}}\varphi,
\varphi
\Big\rangle
\Big|
=
4|c_0|T^{-\frac{\alpha}{4}}(\log T)^{-2}|I_T|,
\label{T-0,2}
\end{align}
where
\[
I_T
=
\int_0^\infty
e^{-i(\operatorname{sgn}t)u}
u^{\frac{\alpha}{4}-1}
\widetilde\chi_1\left((u/T)^{1/4}\right)
\left(1-\frac{\log u}{\log T}\right)^{-2}
du .
\]

We first prove that
\begin{equation}\label{limit:F-T}
I_T
\longrightarrow
e^{-i\frac{\pi\alpha}{8}\operatorname{sgn}(t)}
\Gamma\left(\frac{\alpha}{4}\right),
\qquad T\to\infty .
\end{equation}

Fix \(R>1\). For sufficiently large \(T\), we have \(R<T^{1/2}\).
Hence, for \(0<u\le R\),
\begin{align}\label{R-T}
\left(1-\frac{\log u}{\log T}\right)^{-2}\le 4.
\end{align}
Moreover,
\[
\widetilde\chi_1((u/T)^{1/4})\to1,
\qquad
\left(1-\frac{\log u}{\log T}\right)^{-2}\to1
\]
for every \(u>0\) as \(T\to\infty\). Since
\(u^{\frac{\alpha}{4}-1}\in L^1(0,R)\), the dominated convergence
theorem yields
\begin{align*}
			&\int_0^R
			e^{-i(\operatorname{sgn}t)u}
			u^{\frac{\alpha}{4}-1}
			\widetilde\chi_1\left((u/T)^{1/4}\right)
			\Big(
			1-\frac{\log u}{\log T}
			\Big)^{-2}\,du
		\longrightarrow
			\int_0^R
			e^{-i(\operatorname{sgn}t)u}
			u^{\frac{\alpha}{4}-1}\,du,\ \ T\to\infty.
		\end{align*}

It remains to control the tail uniformly in \(T\). We decompose
\[
[R,\infty)
=
[R,T^{1/2}]
\cup
[T^{1/2},\infty).
\]

For the interval $[R,T^{1/2}]$, a single integration by parts gives
\[
\bigg|
\int_R^{T^{1/2}}
e^{-i(\operatorname{sgn}t)u}
u^{\frac{\alpha}{4}-1}
\widetilde\chi_1((u/T)^{1/4})
\left(1-\frac{\log u}{\log T}\right)^{-2}du
\bigg|
\lesssim
R^{\frac{\alpha}{4}-1}.
\]
Indeed, on this interval,
we also derive \eqref{R-T}
and the derivatives of the amplitude are bounded by
\(Cu^{\frac{\alpha}{4}-2}\).

For the second interval, using again integration by parts and the
support property
\[
\operatorname{supp}\widetilde\chi_1\cap[0,\infty)
\subset [0,(2\lambda_0)^{1/4}],
\qquad \lambda_0\ll1,
\]
we have
\[
\widetilde\chi_1((u/T)^{1/4})\left(1-\frac{\log u}{\log T}\right)^{-2}
\lesssim (\log T)^2
\]
on the support of the integrand. Consequently,
\begin{align*}
&\bigg|
\int_{T^{1/2}}^\infty
e^{-i(\operatorname{sgn}t)u}
u^{\frac{\alpha}{4}-1}
\widetilde\chi_1((u/T)^{1/4})
\left(1-\frac{\log u}{\log T}\right)^{-2}du
\bigg|
\lesssim
(\log T)^2T^{\frac{\alpha-4}{8}} \longrightarrow0,\quad\ \text{as}\  T\to\infty,
\end{align*}
since \(0<\alpha\le2\) implies \(\alpha-4<0\).

Therefore, for every fixed \(R>1\),
\begin{align}
\limsup_{T\to\infty}
\Big|
I_T-
\int_0^R
e^{-i(\operatorname{sgn}t)u}
u^{\frac{\alpha}{4}-1}du
\Big|
\lesssim
R^{\frac{\alpha}{4}-1}.
\label{limsup-IR}
\end{align}

On the other hand, since \(0<\alpha/4<1\), a single integration by
parts gives
\[
\Big|
\int_R^\infty
e^{-i(\operatorname{sgn}t)u}
u^{\frac{\alpha}{4}-1}du
\Big|
\lesssim
R^{\frac{\alpha}{4}-1}.
\]
Combining this with \eqref{limsup-IR}, we obtain
\[
\limsup_{T\to\infty}
\Big|
I_T-
\int_0^\infty
e^{-i(\operatorname{sgn}t)u}
u^{\frac{\alpha}{4}-1}du
\Big|
\lesssim
R^{\frac{\alpha}{4}-1}.
\]
Since $(\alpha/4)-1<0$,
letting \(R\to\infty\) yields
\[
\lim_{T\to\infty}
\Big|
I_T-
\int_0^\infty
e^{-i(\operatorname{sgn}t)u}
u^{\frac{\alpha}{4}-1}du
\Big|
=0.
\]
Hence
\[
I_T\longrightarrow
\int_0^\infty
e^{-i(\operatorname{sgn}t)u}
u^{\frac{\alpha}{4}-1}du, \quad\ \text{as}\  T\to\infty.
\]

Finally, applying the oscillatory Gamma identity
\[
\int_0^\infty e^{-i\varepsilon u}u^{\beta-1}du
=
e^{-i\varepsilon\pi\beta/2}\Gamma(\beta),
\qquad 0<\beta<1,
\]
with $\beta=\alpha/4$ and $\varepsilon=\operatorname{sgn}(t)$,
we obtain \eqref{limit:F-T}.
	Consequently, \eqref{T-0,2} and \eqref{limit:F-T}  give
		\begin{align}
			\Big|
			\Big\langle
			\sum_{j,l=5}^{6}T_{\mathcal I_{j,l}}\varphi,
			\varphi
			\Big\rangle
			\Big|
			\geq
			2|c_0|
			\Gamma\left(\frac{\alpha}{4}\right)
			|t|^{-\frac{\alpha}{4}}(\log|t|)^{-2}\gtrsim
			|t|^{-\frac{\alpha}{4}}(\log|t|)^{-2}.
			\label{0<alphaleq2}
		\end{align}

		Combining \eqref{alpha=0} and \eqref{0<alphaleq2}, we complete this proof.
\end{proof}

\subsection{The estimates for the case $Q_{5} = 0$}\label{Q5=0}
We next turn to the case where zero is an eigenvalue of $H$ with no $d$-wave resonance, i.e.,  all solutions $\phi \in W_{0}(\mathbb{R}^2)$ to $H\phi = 0$ are in $L^2(\mathbb{R}^2)$, i.e., $Q_5 L^2 = \{0\}$.

The condition \(Q_5L^2=\{0\}\) causes certain terms in
\((M^\pm(\lambda))^{-1}\) to vanish, leading to the improved decay
rate in Theorem~\ref{main_theorem_low_32}. The corresponding
asymptotic expansion was established in \cite[Lemma~7.6]{CCWY}.
Note that the stronger decay assumption
\(|V(x)|\lesssim\langle x\rangle^{-22-}\) imposed in
\cite[Lemma~7.6]{CCWY} is needed for the further analysis of the
case where both the \(d\)- and \(p\)-wave resonances are absent.
Since we do not distinguish this additional case here, the
weaker assumption
\(|V(x)|\lesssim\langle x\rangle^{-18-}\) suffices for our purposes.

\begin{lemma}\cite[Lemma 7.6]{CCWY}\label{M_inverse_optimal}  
	Assume that zero is an eigenvalue of $H$ but exhibits no $d$-wave resonance $(\text{i.e.,}\  Q_5 = 0)$. Suppose that $|V(x)| \lesssim \langle x \rangle^{-18-}$ and $0<\lambda\ll1$. Then 
	\begin{itemize}  
		\item[(i)] For $0 \le j,l\le 4$, the estimates for $\mathcal{M}_{j,l}^\pm(\lambda)$ established in Theorem~\ref{thm:M_inverse} remain valid.
		    \vskip0.1cm
		\item[(ii)] For $1 \le j \le 3$ and $\ell = 0,1,2$,  
		\begin{align}  \label{M6123}
		\bigl\|\partial_\lambda^{\ell} \mathcal{M}_{6,j}^\pm(\lambda)\bigr\|_{\mathbb{B}(L^2)}  
		+ \bigl\|\partial_\lambda^{\ell} \mathcal{M}_{j,6}^\pm(\lambda)\bigr\|_{\mathbb{B}(L^2)}  
		\lesssim \lambda^{1-\ell}.  
		\end{align}
		    \vskip0.1cm
		\item[(iii)] For $(j,l) =(6,0),\,(0,6),\,(6,4),\,(4,6),\,(6,6)$, the operators admit the following expansions:  
		\[  
		\begin{aligned}  
			\mathcal{M}_{6,0}^{\pm}(\lambda) &= \lambda (\log \lambda)Q_6\Lambda^\pm(\lambda)Q_0, \quad 
			\mathcal{M}_{0,6}^{\pm}(\lambda) = \lambda (\log \lambda)Q_0\Lambda^\pm(\lambda)Q_6,  \\  
			\mathcal{M}_{6,4}^{\pm}(\lambda) &=  \lambda (\log \lambda)^2 Q_6 \Lambda^\pm(\lambda)Q_{4}^1+\lambda (\log \lambda) Q_6 \Lambda^\pm(\lambda)Q_4,\\
		\mathcal{M}_{4,6}^{\pm}(\lambda) &= \lambda (\log \lambda)^2 Q_{4}^1 \Lambda^\pm(\lambda)Q_{6} +\lambda (\log \lambda) Q_4 \Lambda^\pm(\lambda)Q_6, \\ 
		\mathcal{M}_{6,6}^{\pm}(\lambda) &= \mathcal{D}_{6,6}
			+  \lambda^2 (\log \lambda)^2Q_6 \Lambda^\pm(\lambda)Q_6,
		\end{aligned}  
		\]  
		where  $\mathcal{D}_{6,6} = (b_1 Q_6 v G_6^0 v Q_6)^{-1}$ and  $\Lambda^\pm(\lambda)$ denotes a generic $\lambda$-dependent operator in $\mathbb{B}(L^2)$ (possibly differing at each occurrence) satisfying
		$\|\partial_\lambda^{\ell} \Lambda^\pm(\lambda)\|_{\mathbb{B}(L^2)} \lesssim \lambda^{-\ell}$ for $ \ell=0,1,2.
		$
	\end{itemize}
\end{lemma}

Based on the expansions provided by Lemma~\ref{M_inverse_optimal}, we now proceed to prove Theorem~\ref{main_theorem_low_32}. In view of the extensive analysis carried out in the previous sections, the proof reduces to the study of the difference kernels
$
\mathcal{I}_{B}^+ - \mathcal{I}_{B}^-$  (defined in \eqref{def:I-B-alpha})
with
\[
B^\pm(\lambda) \in \bigl\{ \mathcal{M}_{j, l}^{\pm}(\lambda) \mid j=6,\ 0 \le l\le 4 \ \text{or} \ 0 \le j \le 4,\ l=6 \ \text{or} \ j=l=6 \bigr\},
\]
 We establish the  kernel estimates for all components involving the index 6 in the following Proposition \ref{prop_part1}.  

\begin{proposition}\label{prop_part1}    
	Assume $Q_5 = 0$. For
	 any operator $B^\pm(\lambda)$ containing the index $6$ and $-2<\alpha\leq2,$ the associated integral kernels satisfy the following estimates:  
	\begin{itemize}
		\item[(i)] If $B^\pm(\lambda) \in \{ \mathcal{M}_{j,6}^{\pm}(\lambda), \mathcal{M}_{6,j}^{\pm}(\lambda) \}$ for $1 \le j \le 3$, then  
		\[  
		\sup_{x,y\in\mathbb R^2}
		\left|
		\mathcal I_{B}^\pm(\alpha; t,x,y)
		\right|| \lesssim	\langle t\rangle^{-\frac{2+\alpha}{4}}.  
		\]  
		\item[(ii)] If $B^\pm(\lambda) \in \{ \mathcal{M}_{0,6}^{\pm}(\lambda), \mathcal{M}_{6,0}^{\pm}(\lambda) \}$, then
		\[
 \sup_{x,y\in\mathbb R^2}
 \left|
 \mathcal I_{B}^\pm(\alpha; t,x,y)
 \right|| \lesssim	\langle t\rangle^{-\frac{2+\alpha}{4}} \log(2+|t|).
	\]
		\item[(iii)] If $B^\pm(\lambda) \in \{ \mathcal{M}_{4,6}^{\pm}(\lambda), \mathcal{M}_{6,4}^{\pm}(\lambda), \mathcal{M}_{6,6}^{\pm}(\lambda) \}$, then  
		\[
\sup_{x,y\in\mathbb R^2}
\left|
\mathcal I_{B}^\pm(\alpha; t,x,y)
\right|| \lesssim	\langle t\rangle^{-\frac{2+\alpha}{4}} (\log(2+|t|))^{2}.
	\]
	\end{itemize}  
\end{proposition}   
\begin{proof}    
 (i) Let $B^\pm(\lambda) \in \{\mathcal{M}_{j,6}^{\pm}(\lambda), \mathcal{M}_{6,j}^{\pm}(\lambda)\}$ for $1 \le j\le 3$.  
	We utilize the following resolvent expansions adjacent to the projections:  
	\begin{align}    
		\bigl(Q_6 v R_0^\pm(\lambda^4)\bigr)(\cdot,z) &= (b_0 Q_6 v G_2^0)( \cdot,z) + \lambda^{2}\bigl(\mathcal{T}_{6}(\lambda, \cdot,z)+\mathcal{T}_{6,\pm}(\lambda, \cdot,z)\bigr), \label{Q666} \\    
		\bigl(Q_j v R_0^\pm(\lambda^4)\bigr)(\cdot,z) &= \lambda^{-1}\Omega_{j,\pm}(\lambda,\cdot,z). \label{Q123}    
	\end{align}     
	When substituting these into \eqref{def:I-B-alpha} with $4-k_6-k_j = 0$, the $\lambda^{-1}$ singularity originating from the $Q_j$-expansion \eqref{Q123} is perfectly absorbed by    $\|\partial_\lambda^{\ell} \mathcal{M}_{6,j}^\pm\|_{\mathbb{B}(L^2)}+\|\partial_\lambda^{\ell} \mathcal{M}_{j,6}^\pm\|_{\mathbb{B}(L^2)} \lesssim \lambda^{1-\ell}$ from Lemma \ref{M_inverse_optimal}(ii). Applying Lemma~\ref{quartic-oscillatory}(ii) with  $(\sigma,\nu) = (\alpha,0)$ immediately secures the uniform unweighted bound $O\bigl(\langle t\rangle^{-\frac{2+\alpha}{4}}\bigr)$.

	\vspace{0.15cm}  
	 (ii) Let $B^\pm(\lambda) \in \{\mathcal{M}_{6,0}^{\pm}(\lambda), \mathcal{M}_{0,6}^{\pm}(\lambda)\}$.  
	We pair the $Q_6$ expansion \eqref{Q666} with the  singular $Q_0$ expansion $\bigl(Q_0 v R_0^\pm(\lambda^4)\bigr) = \lambda^{-2}\Omega_{0,\pm}(\lambda)$. 
The inherent  factor for these indices is $\lambda^{4-k_6-k_0} = \lambda$ in  \eqref{def:I-B-alpha} . According to Lemma \ref{M_inverse_optimal} (iii), the operator explicitly introduces an extra $\lambda$ and a logarithmic singularity: $\mathcal{M}_{6,0}^{\pm}(\lambda) = \lambda (\log \lambda)Q_6\Lambda^\pm(\lambda)Q_0$ (and symmetrically for $\mathcal{M}_{0,6}^\pm$).
	Multiplying these factors together, applying Lemma~\ref{quartic-oscillatory}(ii) with parameters $(\sigma,\nu) = (\alpha,-1)$ validates the uniform bound ${O}\big(\langle t\rangle^{-\frac{2+\alpha}{4}} \log(2+|t|)\big)$.  
	
	\vspace{0.15cm}  
 (iii)  Let $B^\pm(\lambda) \in \{\mathcal{M}_{6,4}^{\pm}(\lambda), \mathcal{M}_{4,6}^{\pm}(\lambda), \mathcal{M}_{6,6}^{\pm}(\lambda)\}$.  
	We employ \eqref{Q666} for $Q_6$ and  $\bigl(Q_4 v R_0^\pm(\lambda^4)\bigr)(\cdot,z) = \mathcal{J}_{4,\pm}(\lambda, \cdot,z)$ for $Q_4$. Recall from \eqref{eq:E_bound-Q4andQ} that $\mathcal{J}_{4,\pm}$  carries a $|\log \lambda|$ growth.
	
	For $\mathcal{M}_{6,4}^{\pm}(\lambda) = \lambda (\log \lambda) Q_6 \Lambda^\pm(\lambda)Q_4$ (and symmetrically for $\mathcal{M}_{4,6}^\pm$), the operator supplies an explicit $\lambda \log \lambda$ and note that $\lambda^{4-k_6-k_4} = \lambda^{-1}$.
	Thus,  applying Lemma \ref{quartic-oscillatory}(ii) with parameters $(\sigma,\nu)=(\alpha,-2)$ yields the required bound ${O}\big(\langle t\rangle^{-\frac{2+\alpha}{4}} (\log(2+|t|))^{2}\big)$.  
	
	Finally, for  $\mathcal{M}_{6,6}^{\pm}(\lambda) = \mathcal{D}_{6,6} + \lambda^2 (\log \lambda)^2Q_6 \Lambda^\pm(\lambda)Q_6$. Here $\lambda^{4-k_6-k_6} = \lambda^{-2}$. Since the leading term $\mathcal{D}_{6,6}$ is independent of the sign $\pm$, the difference $\mathcal{K}_{\mathcal{D}_{6,6}}^+ - \mathcal{K}_{\mathcal{D}_{6,6}}^-$ is bounded by  $\langle t \rangle^{-1}$ (see \eqref{mathcalA66} in Proposition~\ref{third_merged_1} for the detailed justification).
	For the remaining  term involving $\lambda^2(\log\lambda)^2Q_6 \Lambda^\pm(\lambda)Q_6$, applying   Lemma \ref{quartic-oscillatory}(ii) with parameters $(\sigma,\nu)=(\alpha,-2)$ yields the uniform  bound $O\bigl(\langle t\rangle^{-\frac{2+\alpha}{4}}(\log(2+|t|))^{2}\bigr)$. Summing these contributions concludes the proof of (iii).
\end{proof}

\begin{proof}[\textbf{Proof of Theorem \ref{main_theorem_low_32} }]  
	Combining  Proposition \ref{prop_part1} with the  estimates established  in the regular, first-kind, and second-kind resonance cases,
	we conclude the proof of Theorem \ref{main_theorem_low_32}.  
\end{proof}

\section{Oscillatory integral estimates}\label{section8}
In this section we prove Lemma~\ref{quartic-oscillatory}, which encompasses nearly all classes of oscillatory integrals encountered in this paper.

\begin{proof} [Proof of Lemma~\ref{quartic-oscillatory} ]
	Assume $t > 0$ without loss of generality and abbreviate $\mathcal{E}(\lambda) := \mathcal{E}^\pm(\lambda, x,y)$ and $h := h(x,y)$. 
	The phase has the form
	\[
	\Phi(\lambda)
	=
	-t\lambda^4+\varepsilon r\lambda,
	\qquad
	\varepsilon\in\{-1,1\}.
	\]
	Thus
	\[
	\Phi'(\lambda)
	=
	-4t\lambda^3+\varepsilon r,
	\qquad
	\Phi''(\lambda)
	=
	-12t\lambda^2.
	\] 
	We use the identity $1
	=
	\chi_{1}(t\lambda^4)
	+
	\chi_{2}(t\lambda^4)$ to decompose into the low-frequency part $I_{\mathrm{lo}}$ and high-frequency part $+I_{\mathrm{hi}},$ i.e.,
	$
	\mathcal I^\pm
	=
	I_{\mathrm{lo}}+I_{\mathrm{hi}}.
	$

	A positive stationary point exists only when \(\varepsilon r=|r|>0\), in
	which case it is uniquely given by
	$
	\rho
	=
	\bigl(\frac{|r|}{4t}\bigr)^{1/3}>0.
	$
	Choose \(\eta\in C_c^\infty((0,\infty))\) such that
	\[
	\operatorname{supp}\eta\subset[1/4,4],
	\qquad
	\eta(s)=1
	\quad\text{for }\frac12\leq s\leq2.
	\]
	If a positive stationary point exists (i.e., \(\varepsilon r=|r|>0\)), set
	$
	\eta_\rho(\lambda):=\eta(\lambda/\rho);
	$
	otherwise, set \(\eta_\rho=0\). We decompose
	$
	I_{\mathrm{hi}}
	=
	I_{\mathrm{stat}}+I_{\mathrm{non}},
	$
	where
	\[
	I_{\mathrm{stat}}
	=
	\int_0^\infty
	e^{i\Phi(\lambda)}
	\eta_\rho(\lambda)
	\chi_{2}(t\lambda^4)
	\lambda\mathcal E(\lambda)\,d\lambda,\quad
	I_{\mathrm{non}}
	=
	\int_0^\infty
	e^{i\Phi(\lambda)}
	\bigl(1-\eta_\rho(\lambda)\bigr)
	\chi_{2}(t\lambda^4)
	\lambda\mathcal E(\lambda)\,d\lambda.
	\]
	On the support of $\eta_\rho$, one has
	$	\lambda\sim\rho$, whereas
	on the support of the non-stationary amplitude, $
	|\Phi'(\lambda)|
	\gtrsim
	t\lambda^3.$
	\vskip0.2cm
	
	\noindent
	\textit{Proof of part~\textup{(i)}.}
	For the low-frequency part, assumption \eqref{quartic-cond-1} implies that
	\[
	\begin{aligned}
		|I_{\mathrm{lo}}|
		&\lesssim
		h\int_0^{Ct^{-1/4}}\lambda^{1+\gamma}\,d\lambda\lesssim
		h\,t^{-\frac{2+\gamma}{4}},
	\end{aligned}
	\]
	where the assumption \(\gamma>-2\) guarantees integrability at the
	origin.

	We deal with the  high-frequency  part  $I_{\mathrm{hi}}$ by considering the stationary contribution $I_{\mathrm{stat}}$ and 
	the nonstationary contribution $I_{\mathrm{non}},$ respectively.
	For the stationary contribution, set
	$$
	a_{\mathrm{stat}}(\lambda)
	:=
	\eta_\rho(\lambda)
	\chi_{2}(t\lambda^4)
	\lambda\mathcal E(\lambda).
	$$
	Since $
	|\Phi''(\lambda)|
	\sim
	t\rho^2
	$
	on the support of \(a_{\mathrm{stat}}\), the van der Corput lemma
	therefore yields
	\begin{align}\label{I_start}
		|I_{\mathrm{stat}}|
		\lesssim
		(t\rho^2)^{-1/2}
		\left(
		\|a_{\mathrm{stat}}\|_{L^\infty}
		+
		\|\partial_\lambda a_{\mathrm{stat}}\|_{L^1}
		\right).
	\end{align}
	Using \eqref{quartic-cond-1}  and
	\(\lambda\sim\rho\), we obtain
	\[
	\|a_{\mathrm{stat}}\|_{L^\infty}
	+
	\|\partial_\lambda a_{\mathrm{stat}}\|_{L^1}
	\lesssim
	h\rho^{1+\gamma}
	\langle\rho r\rangle^{-1/2}.
	\]
	Consequently, combining this with
	$
	\rho|r|=4t\rho^4
	$ and $t\rho^4\gtrsim1$ (given by $t\lambda^4\gtrsim1$), for \(\gamma\leq2\), it follows that
	\[
	\begin{aligned}
		|I_{\mathrm{stat}}|&\lesssim
		h\,t^{-1/2}\rho^\gamma
		\langle\rho r\rangle^{-1/2}\lesssim 	h\,t^{-\frac{2+\gamma}{4}}\bigl(t^{\frac{1}{4}}\rho\bigr)^\gamma\bigl(4t\rho^4\bigr)^{-1/2}
		\lesssim h\,t^{-\frac{2+\gamma}{4}}\bigl(t^{\frac{1}{4}}\rho\bigr)^{\gamma-2}\lesssim
		h\,t^{-\frac{2+\gamma}{4}}.
	\end{aligned}
	\]
	
	It remains to estimate \(I_{\mathrm{non}}\). Choose
	\(\psi\in C_c^\infty((1/2,2))\) satisfying
	\[
	\sum_{j\in\mathbb Z}\psi(2^{-j}s)=1,
	\qquad s>0.
	\]
	Set
	$
	\Lambda_j:=2^jt^{-1/4}$ and 
	$
	\psi_j(\lambda):=\psi(\lambda/\Lambda_j).$
On the support of \(\psi_j\), one has
$
\lambda\sim\Lambda_j,
$
with constants independent of \(j\). Hence, since
\(\lambda\gtrsim t^{-1/4}\) on the high-frequency region, we have
$
\Lambda_j\gtrsim t^{-1/4}.
$ Hence
\begin{align}\label{sum-Ij}
		I_{\mathrm{non}}
		=
		\sum_{\Lambda_j\gtrsim t^{-1/4}}I_j,
	\end{align}
	where
	\[
	I_j
	=
	\int_0^\infty e^{i\Phi(\lambda)}a_j(\lambda)\,d\lambda,\quad \text{with}\ 
	a_j(\lambda)
	=
	\psi_j(\lambda)
	\bigl(1-\eta_\rho(\lambda)\bigr)
	\chi_{\mathrm{hi}}(t\lambda^4)
	\lambda\mathcal E(\lambda).
	\]
	On the support of \(a_j\), one has
	\[
	\lambda\sim\Lambda_j,
	\quad
	|\Phi'(\lambda)|\gtrsim t\Lambda_j^3,\quad 	|\Phi''(\lambda)|
	\lesssim
	t\Lambda_j^2\quad \text{and}\quad	|\Phi'''(\lambda)|
	\lesssim
	t\Lambda_j.
	\]
	Moreover, the derivatives of all cutoff factors are bounded by the
	corresponding powers of \(\Lambda_j^{-1}\). Hence
	\eqref{quartic-cond-1} gives
	\begin{equation}\label{quartic-amplitude-alpha}
		|\partial_\lambda^\ell a_j(\lambda)|
		\lesssim
		h\Lambda_j^{1+\gamma-\ell}
		\langle\Lambda_jr\rangle^{-1/2},
		\qquad \ell=0,1,2.
	\end{equation}
	
	Define
	\[
	Lf
	:=
	\partial_\lambda
	\left(
	\frac{f}{i\Phi'(\lambda)}
	\right).
	\]
	Since \(a_j\) is compactly supported in \((0,\infty)\),     integrations by parts produces no boundary terms. Hence two  
	integrations by parts give
	\begin{align}\label{Ii-2part}
		I_j
		=
		\int_0^\infty
		e^{i\Phi(\lambda)}L^2a_j(\lambda)\,d\lambda.
	\end{align}
	Using $|\Phi'(\lambda)|\gtrsim t\Lambda_j^3$ and
	\eqref{quartic-amplitude-alpha}, a direct calculation gives
	\[
	|L^2a_j(\lambda)|
	\lesssim
	h\,t^{-2}
	\Lambda_j^{\gamma-7}
	\langle\Lambda_jr\rangle^{-1/2}.
	\]
	Since the support of \(a_j\) has length \(O(\Lambda_j)\),
	\[
	|I_j|
	\lesssim
	h\,t^{-2}
	\Lambda_j^{\gamma-6}
	\langle\Lambda_jr\rangle^{-1/2}
	\lesssim
	h\,t^{-2}\Lambda_j^{\gamma-6}.
	\]
	Therefore, for any \(\gamma<6\),
	\[
	\begin{aligned}
		|I_{\mathrm{non}}|
		&\lesssim
		h\,t^{-2}
		\sum_{\Lambda_j\gtrsim t^{-1/4}}
		\Lambda_j^{\gamma-6}\lesssim
		h\,t^{-\frac{2+\gamma}{4}}
		\sum_{j=0}^\infty2^{-(6-\gamma)j}\lesssim
		h\,t^{-\frac{2+\gamma}{4}}. \end{aligned}
	\]
 This proves part~\textup{(i)}. In particular, 
	if \(r=0\), there is no positive stationary point. Thus only the
	low-frequency and non-stationary estimates  are needed. The bound $	h\,t^{-\frac{2+\gamma}{4}}$  is valid whenever
	$
	-2<\gamma<6$.
	
	\vskip 0.2cm
	\noindent
	\textit{Proof of part~\textup{(ii)}.}
	For \(0<t\lesssim1\),  by the compact
	support of $\widetilde{\chi}_1(\lambda/2)$ and assumption
	\eqref{quartic-cond-2} imply
	\begin{align}\label{Itsmall}
		|\mathcal I^\pm(t,x,y)|
		\lesssim
		h\int_0^{1/2}
		\lambda^{1+\sigma}
		|\log\lambda|^{-\nu}\,d\lambda.
	\end{align}
	This integral is finite when \(\sigma>-2\), or when
	\(\sigma=-2\) and \(\nu>1\). Hence it remains to consider
	\(t\gg1\).

	The low-frequency component extracts the precise logarithmic decay via the upper incomplete Gamma function $\Gamma(a, x) = \int_x^\infty u^{a-1} e^{-u} du$:
	\begin{align*}
		|I_{\mathrm{lo}}|
		\lesssim h	\int_0^{t^{-1/4}} \frac{\lambda^{1+\sigma}}{ |\log \lambda|^{\nu}} d\lambda 
		= h\int_{\frac{1}{4}\log t}^{\infty} \frac{u^{-\nu}}{e^{u(2+\sigma)}} du	
		=\begin{cases}
			\frac{4^{\nu-1} h}{(\nu-1)(\log t)^{\nu-1}}, & \sigma=-2, \nu>1, \\[4pt]
			(2+\sigma)^{\nu-1}  h\Gamma( 1-\nu, \frac{2+\sigma}{4} \log t ), & \sigma>-2, \nu \in \mathbb{R}.
		\end{cases}
	\end{align*}
	Using the standard asymptotic expansion $\Gamma(a,x) \sim x^{a-1}e^{-x}$ as $x \to \infty$, we obtain:
	\begin{align}\label{low-fre}
		|I_{\mathrm{lo}}| \lesssim 
		\begin{cases}
			h(\log t)^{-\nu+1},               & \sigma=-2, \nu>1,              \\[4pt]
			ht^{-\frac{2+\sigma}{4}} (\log t)^{-\nu}, & \sigma>-2, \nu \in \mathbb{R}.
		\end{cases}
	\end{align}
    
	For the stationary contribution, the same van der Corput argument as \eqref{I_start}
	gives
	\[
	|I_{\mathrm{stat}}|
	\lesssim
	h\,t^{-1/2}\rho^{\sigma}
	|\log\rho|^{-\nu}
	\langle\rho r\rangle^{-1/2}.
	\]
	Using \(\rho|r|=4t\rho^4\) and \(t\rho^4\gtrsim1\), we obtain
	\begin{equation}\label{quartic-stationary-sigma}
		|I_{\mathrm{stat}}|
		\lesssim
		h\,t^{-1}
		\rho^{\sigma-2}
		|\log\rho|^{-\nu}.
	\end{equation}
	Note that under the parameter conditions $ \sigma<2 $ with $\nu \in \mathbb{R}$ or $\sigma = 2$ with $\nu \le 0$,
	the map $\lambda \mapsto \lambda^{\sigma-2}|\log \lambda|^{-\nu}$ is decreasing  on the sufficiently small support of
	$\widetilde{\chi}_1(\lambda/2)$.  Hence, whenever $t^{-1/4}\lesssim \lambda\sim \rho\ll1$, one has
	\begin{align}\label{decreasing}
		\rho^{\sigma-2}|\log \rho|^{-\nu}
		\lesssim
		t^{-\frac{-2+\sigma}{4}}
		(\log t)^{-\nu}.
	\end{align}
	Combining \eqref{quartic-stationary-sigma} and
	\eqref{decreasing}, we obtain  under the given parameter conditions $-2 < \sigma < 2$, $\nu \in \mathbb{R}$; or $\sigma = 2$, $\nu \le 0$; or $\sigma = -2$, $\nu > 1$,
	\begin{align}\label{I-stat}
		|I_{\mathrm{stat}}|
		\lesssim
		h\,t^{-\frac{2+\sigma}{4}}
		(\log t)^{-\nu}.
	\end{align}

	We next estimate the non-stationary contribution. By the compact
	support of $\widetilde{\chi}_1(\lambda/2)$, there exists a small $0<\lambda_*\ll1$ such that $\operatorname{supp}\widetilde{\chi}_1(\lambda/2)\subset[-\lambda_*, \lambda_*].$ Moreover,
	the same dyadic
	decomposition as in \eqref{sum-Ij} applies, except that only the scales
	satisfying
	$
	t^{-1/4}
	\lesssim
	\Lambda_j
	\leq
	\lambda_*
	$
	contribute. 	After performing two integrations by parts with no boundary terms, we arrive at the same identity as in \eqref{Ii-2part}.
	Similarly, two integrations by parts yield
	\[
	|I_j|
	\lesssim
	h\,t^{-2}
	\Lambda_j^{\sigma-6}
	|\log\Lambda_j|^{-\nu}
	\langle\Lambda_jr\rangle^{-1/2}.
	\]
	Consequently,
	\begin{equation}\label{quartic-non-log-sum}
		|I_{\mathrm{non}}|
		\lesssim
		h\,t^{-2}
		\sum_{t^{-1/4}\lesssim\Lambda_j\leq\lambda_*}
		\Lambda_j^{\sigma-6}
		|\log\Lambda_j|^{-\nu}.
	\end{equation}
	
	We use the following elementary dyadic estimate: for every 
	\(\epsilon>0\) and \(\nu\in\mathbb R\),
	\begin{equation}\label{quartic-dyadic-log-sum}
		\sum_{t^{-1/4}\lesssim\Lambda_j\leq\lambda_*}
		\Lambda_j^{-\epsilon}
		|\log\Lambda_j|^{-\nu}
		\lesssim
		t^{\epsilon/4}(\log t)^{-\nu}.
	\end{equation}
	To prove it, we introduce the positive integer $J$ satisfying $2^Jt^{-1/4}\sim\lambda_*.$ Thus $J\sim\log t$ and 
	\begin{equation*}
		\sum_{t^{-1/4}\lesssim\Lambda_j\leq\lambda_*}
		\Lambda_j^{-\epsilon}
		|\log\Lambda_j|^{-\nu}
		\lesssim t^{\frac{\epsilon}{4}}\sum_{j=0}^J 2^{-j\epsilon}|\log\Lambda_j|^{-\nu}.
	\end{equation*}
	We split the sum into the ranges  \(0\leq j\leq J/2\) and  \(J/2<j\leq J\).	For \(0\leq j\leq J/2\), one has
	$
	|\log\Lambda_j|\sim\log t, 
	$ and hence
	the required estimate follows from the geometric summability of
	\(2^{-\epsilon j}\), i.e.,
	\begin{align*}
		t^{\frac{\epsilon}{4}}\sum_{j=0}^{J/2} 2^{-j\epsilon}|\log\Lambda_j|^{-\nu}\lesssim
		t^{\frac{\epsilon}{4}}(\log t)^{-\nu}\sum_{j=0}^{\infty}2^{-j\epsilon}\lesssim	t^{\frac{\epsilon}{4}}(\log t)^{-\nu}.
	\end{align*}
	For \(J/2<j\leq J\), the fixed restriction $t^{-1/4}\lesssim\Lambda_j\leq\lambda_*\ll1$ implies 	$
	0<	|\log\lambda_*|\leq	|\log\Lambda_j|
	\lesssim\log t
	.
	$
	Then, 
	$$
	t^{\frac{\epsilon}{4}} \sum_{J/2<j\leq J}
	2^{-j\epsilon}
	\left|\log
	\Lambda_j
	\right|^{-\nu}
	\lesssim t^{\frac{\epsilon}{4}} 
	(\log t)^{\max\{-\nu,0\}} \sum_{J/2<j\leq J}
	2^{-j\epsilon}	\lesssim t^{\frac{\epsilon}{4}} t^{-\frac{\epsilon}{8}} (\log t)^{\max\{-\nu,0\}}\lesssim t^{\frac{\epsilon}{4}} 
	(\log t)^{-\nu},
	$$
	This proves
	\eqref{quartic-dyadic-log-sum}.
	Taking \(\epsilon=-\sigma+6>0\), we obtain, whenever \(\sigma<6\),
	\begin{align}\label{I-non2}
		|I_{\mathrm{non}}|
		&\lesssim
		h\,t^{-2}t^{\frac{-\sigma+6}{4}}
		(\log t)^{-\nu}=
		h\,t^{-\frac{2+\sigma}{4}}
		(\log t)^{-\nu}.
	\end{align}
	
	Combining the low-frequency, stationary, and non-stationary estimates (i.e., \eqref{low-fre}, \eqref{I-stat} and \eqref{I-non2})
	proves the desired  estimate  of Part \emph{(ii)}. In particular, 
	if \(r=0\), there is no positive stationary point. Thus only the
	low-frequency and non-stationary estimates  are needed. The bound $	h\,t^{-\frac{2+\sigma}{4}}
	(\log t)^{-\nu}$  is valid whenever
	$
	-2<\sigma<6$ and $
	\nu\in\mathbb R.
	$
\end{proof}

\begin{remark}\label{remark:osci}
	{\rm	In the critical case \(\sigma=-2\) and \(\nu>1\), the estimates
		\eqref{Itsmall}, \eqref{low-fre}, \eqref{I-stat}, and
		\eqref{I-non2} imply that, for \(|t|\geq2\),
		\[
		\Big| \int_0^{\infty} e^{-it\lambda^4}
		e^{\pm i\lambda r} \chi_1(|t|\lambda^2) \lambda \mathcal{E}(\lambda) \, d\lambda \Big| \lesssim \frac{h(x,y)}{(\log |t|)^{\nu-1}}, \quad
		\Big| \int_0^{\infty} e^{-it\lambda^4}
		e^{\pm i\lambda r} \chi_2(|t|\lambda^2) \lambda \mathcal{E}(\lambda) \, d\lambda \Big| \lesssim \frac{h(x,y)}{(\log |t|)^{\nu}}.
		\]
		Thus the high-frequency part decays faster than the
		low-frequency part by an additional factor of
		\((\log|t|)^{-1}\). Consequently, the overall decay rate in the
		critical case is determined by the low-frequency part.}
\end{remark}



{\bf Acknowledgements:} 
Zijun Wan is partially supported  by the Postdoctoral Fellowship Program of CPSF under Grant Number GZB20260742.
Xiaohua Yao is partially supported by NSFC (grants No. 12171182 and 12531005). The authors would like to express their thanks to Professor Avy Soffer for his interests and insightful discussions.


\end{document}